\documentclass{amsart}
\usepackage{amsmath}
\usepackage{amssymb}
\usepackage{amsthm}
\usepackage{enumerate}

\usepackage{textcomp}
\theoremstyle{definition}
\newtheorem{definition}{Definition}[section]
\theoremstyle{plain}
\newtheorem{lemma}[definition]{Lemma}
\newtheorem{theorem}[definition]{Theorem}
\newtheorem{proposition}[definition]{Proposition}
\newtheorem{corollary}[definition]{Corollary}
\theoremstyle{remark}
\newtheorem{remark}[definition]{Remark}
\newtheorem{notation}[definition]{Notation}
\newtheorem{example}[definition]{Example}

\usepackage{caption}
\usepackage{graphicx}

\newcommand{\mylc}{\operatorname{\mbox{\bf an}}}
\newcommand{\mylcp}{\operatorname{\mbox{\bf p.an}}}
\newcommand{\val}{\operatorname{\mbox{\bf val}}}
\newcommand{\supp}{\operatorname{\mbox{\rm supp}}}
\newcommand{\myConv}{\operatorname{\mbox{\rm Conv}}}
\newcommand{\myVal}{\operatorname{\mbox{\rm Val}}}
\newcommand{\po}{\operatorname{\mbox{\rm PO}}}

\makeatletter
\@namedef{subjclassname@2010}{
  \textup{2010} Mathematics Subject Classification}
\makeatother

\begin{document}

\title[Quasi-quadratic modules]{Quasi-quadratic modules in valuation rings and valued fields}
\author[M. Fujita]{Masato Fujita}
\address{Department of Liberal Arts,
Japan Coast Guard Academy,
5-1 Wakaba-cho, Kure, Hiroshima 737-8512, Japan}
\email{fujita.masato.p34@kyoto-u.jp}
\author[M. Kageyama]{Masaru Kageyama}
\address{Department of General Education, 
National Institute of Technology (KOSEN), Kure College,
2-2-11 Agaminami, Kure, Hiroshima 737-8506, Japan}
\email{mkageyama@kure-nct.ac.jp}

\begin{abstract}
We define quasi-quadratic modules in a commutative ring generalizing the notion of quadratic modules. 

The main theorem is a structure theorem of quasi-quadratic modules in a subring $A$ of a $2$-henselian valued field $(K,\val)$ whose residue class field $F$ is of characteristic $\neq 2$.
We further assume that the valuation ring $B$ is contained in $A$.
Set $H=\val(A^\times)$ and $G_{\geq e}=\{g \in G\;|\; g \geq e\}$.
The notation $\mathfrak X_R$ denotes the set of all the quasi-quadratic modules in a commutative ring $R$.
Our structure theorem asserts that there exists a one-to-one correspondence between $\mathfrak X_A$ and a subset 
$\mathcal T_F^{ H \cup G_{\geq e}}$ of $\prod_{g \in H \cup G_{\geq e}}\mathfrak X_F$.
We explicitly construct the map $\Theta: \mathfrak X_A \rightarrow \mathcal T_F^{ H \cup G_{\geq e}}$ and its inverse.
We also give explicit expressions of $\Theta(\mathcal M \cap \mathcal N)$ and $\Theta(\mathcal M+\mathcal N)$ for $\mathcal M, \mathcal N \in \mathfrak X_A$.

%%%%%%%%%%%%
In addition, we briefly investigate the case in which the field $F$ is of characteristic two in the appendix as well.
\end{abstract}



\maketitle
\tableofcontents

\section{Introduction}\label{sec:intro}

Quadratic modules in polynomial rings are extensively studied.
A major result in this direction gives a sufficient condition for a polynomial being positive on a basic closed semialgebraic set.
Many works have been done in this direction.
They are summarized in \cite{M,PD} and applied to polynomial optimization problems \cite{L2}.

Apart from them, quadratic modules in the ring of univariate formal power series $E[\![X]\!]$ were completely classified in \cite{AK} when $E$ is a euclidean field.
This result is due to the simple form of the elements in the ring $E[\![X]\!]$ and to the fact that a quadratic module is finitely generated.
The same classification as in \cite{AK} is obtained when the ring in consideration is the ring of univariate convergent power series $\mathbb R\{x\}$ and the univariate Nash ring with real coefficients following literally the same proof as \cite{AK}.
See \cite[p.106]{R} for the definition of Nash rings.
However, they are only true in nonsingular cases.
The authors demonstrated that quadratic modules in convergent power series rings defined on some singular curves are not always finitely generated \cite{FK}.

The authors anticipate that assertions similar to \cite{AK} hold true in a more general setting.
We consider valuation rings and valued fields in this paper.
Let us review the definitions and notations used in this paper before we state our fundamental assumptions.

A commutative ring in this paper means a commutative ring with the multiplicative identity element $1$.
Let $(G,\geq,\cdot,e)$ be a fully ordered abelian group with 
identity element $e$. As usual, we write $G$ for $(G,\geq,\cdot,e)$.
Let $(K,\val)$ be a valued field, where $K$ is a field and $\val:K \twoheadrightarrow G\cup\{\infty\}$ is a valuation.
The notation $B$ denotes the valuation ring with respect to $\val$ and $\pi:B \rightarrow F$ denotes the residue homomorphism, where $F$ is the residue class field.
The notation $\mathfrak{m}$ denotes the maximal ideal of the valuation ring $B$.
The multiplicative group of all units in a commutative ring $A$ is denoted by $A^{\times}$.
We use these notations through the paper unless explicitly specified.
The reference \cite{EP} provides a through introduction for valuation rings and valued fields.
We could not treat valuation rings and valued fields in full generality.

We employ the following two extra assumptions which are satisfied by the above univariate power series rings:
\begin{enumerate}[(I)]
\item The residue class field is of characteristic $\neq 2$.
\item Any strict units admit a square root.
\end{enumerate}
Here are the definitions of the relevant notions.
\begin{definition}
Recall that an element $x \in K$ is a unit in $B$ if and only if $\val(x)=e$.
A unit $x$ in $B$ is called a \textit{strict unit} if $\pi(x)=1$.
We say that an element $x$ in $K$ \textit{admits a square root} if there exists $y \in K$ with $x=y^2$. 
\end{definition}

Let $p$ be a rational prime. We call $(K,\val)$ {\it $p$-henselian} if the following statement holds:
Every polynomial $f(X)=X^n+X^{n-1}+a_{n-2}X^{n-2}+\cdots+a_0\in B[X]$ with
$a_{n-2},\ldots,a_0 \in\mathfrak{m}$ has a zero in $K$ if $f(X)$ splits in $K(p)$,
where $K(p)$ is the maximal Galois $p$-extension of $K$.
See \cite[p.94]{EP} for the definition of the maximal Galois $p$-extension  $K(p)$ of $K$.

It is well known that, for a valued field whose residue class field is of characteristic $\neq 2$, any strict units admit a square root if and only if the valued field $(K,\val)$ is $2$-henselian by \cite[Corollary 4.2.4]{EP}.
Various valued fields  fall into the family of the valued fields satisfying the above two assumptions.
For instance, the field of Hahn series $F (\!(G)\!)$ with natural valuation \cite{F} is categorized into this family when $F$ is a field of characteristic zero and $G$ is a fully ordered abelian group. 
In particular, the field $F(\!(G)\!)$ is the iterated Laurent series field $F(\!(t_1)\!) \cdots (\!(t_n)\!)$ when $G$ is $\mathbb Z^n$ with the lexicographic order under the order $t_1 > t_2 > \cdots > t_n$.

Unlike \cite{AK}, we focus on quasi-quadratic modules rather than quadratic modules.
A quasi-quadratic module in a commutative ring $A$ is defined as follows:

\begin{definition}
Let $A$ be a commutative ring.
A subset $M$ of $A$ is called a \textit{quasi-quadratic module} in $A$ if $M+M \subseteq M$ and $a^2 M \subseteq M$ for all $a \in A$.
Note that we always have $0 \in M$.
A quasi-quadratic module distinct from $A$ is called a {\it proper} quasi-quadratic module.
A quasi-quadratic module is called a {\it quadratic module} if it contains sums of squares of elements in $A$.
\end{definition}
We study quasi-quadratic modules rather than quadratic modules because we can get simpler structure theorems especially when $A=B$ and $A=K$.

We are now ready to explain the target of this paper.
Our target is to prove structure theorems for quasi-quadratic modules of a valued field $K$ and its valuation ring $B$ when the residue class field $F$ is of characteristic $\neq 2$ and any strict units admit square roots.
In order to treat $K$ and $B$ simultaneously, we consider a subring $A$ of $K$ containing $B$ in this paper.  
Set $H=\val(A^\times)$ and $G_{\geq e}=\{g \in G\;|\; g \geq e\}$.
The notation $\mathfrak X_R$ denotes the set of all quasi-quadratic modules in a commutative ring $R$.
Our structure theorem asserts that there exists a one-to-one correspondence between $\mathfrak X_A$ and a subset 
$\mathcal T_F^{ H \cup G_{\geq e}}$ of $\prod_{g \in H \cup G_{\geq e}}\mathfrak X_F$.
We explicitly construct the map $\Theta: \mathfrak X_A \rightarrow \mathcal T_F^{ H \cup G_{\geq e}}$ and its inverse.
We also give explicit expressions of $\Theta(\mathcal M \cap \mathcal N)$ and $\Theta(\mathcal M+\mathcal N)$ for $\mathcal M, \mathcal N \in \mathfrak X_A$.

Similar but not identical assertions are obtained when any strict units admit a square root and the residue class field is of characteristic two.
%%%%%%%%%%%%%%%
We will treat this case in the appendix.

The paper is organized as follows:
Section \ref{sec:preliminary} is a preliminary section.
We introduce several technical concepts necessary in the succeeding sections.
A pseudo-angular component map $\mylcp: K^\times \rightarrow F^\times$ defined in subsection \ref{sec:pseudo-angular} is a key concept.
Its definition is developed inspired by the notion of angular component maps given in \cite{P1,P2}, which are used for the model-theoretic study of henselian valued fields.
We demonstrate that a pseudo-angular component map always exists in our setting in Proposition \ref{prop:psudo-angular}.
Section \ref{sec:main} is the main part of this paper.
We demonstrate the structure theorem introduced above in this section.
The assertions in Section \ref{sec:main} become simpler in the case when the ring $A$ is the valuation ring $B$ and the valued field $K$.
We discuss them in Section \ref{sec:special}.
In particular, we derive Augutin and Knebusch's assertions in \cite{AK} from 
our results in subsection \ref{sec:AK}.
As a more special case, we consider the case in which the residue class field $F$ is euclidean and
classify monogenic quadratic modules of iterated Laurent series fields over a euclidean field.
The final section is an appendix.
We treat the case in which the residue class field $F$ is of characteristic two.
We can get a much simpler structure theorem in this case.

\section{Preliminaries}\label{sec:preliminary}
\subsection{Convex subgroups}
We first introduce the definitions and assertions on ordered abelian groups.
\begin{definition}[Convex subgroups]
A subgroup $H$ of an ordered abelian group $(G,e, \geq)$ is \textit{convex} if each $g \in G$ with $e \leq g \leq h \in H$ already belongs to $H$. 
\end{definition}

\begin{notation}
We introduce three notations for different equivalence classes.
Let $G$ be an ordered abelian group and $H$ be a convex subgroup of $G$.
The equivalence class of $g \in G$ in $G/G^2$ is denoted by $\overline{g}$.
The notation $[\![g]\!]$ denotes the equivalence class of $g \in G$ in $G/H^2$.
We can consider the equivalence class of an element of $G/H^2$ in $G/G^2$ because $H^2$ is a subgroup of $G^2$.
We also use the same notation $\overline{g}$ for the equivalence class of $g \in G/H^2$.
Finally, the notation $[ g ]$ denotes the equivalence class of $g \in G$ in $G/H$.
The same notation $[ g ]$ is also used for the equivalence class of $g \in G/H^2$ in $G/H$.
\end{notation}

\begin{lemma}\label{lem:convex_basic11}
Let $(G,e, \geq)$ be an ordered abelian group and $H$ be a convex subgroup.
For any $g,h \in G$, if $\overline{g}=\overline{h}$ and $[ g ] = [ h ]$, we have $[\![g]\!]=[\![h]\!]$.
\end{lemma}
\begin{proof}
We have nothing to prove when $g=h$.
We may assume that $g<h$ without loss of generality.
We can take $u \in G$ so that $h=gu^2$ and $u^2 \in H$ by the assumption.
We also have $u>e$.
We get $e<u<u^2 \in H$.
Since $H$ is convex, we have $u \in H$.
It means that $[\![g]\!]=[\![h]\!]$.
\end{proof}

\begin{lemma}\label{lem:convex_basic12}
Let $(G,e, \geq)$ be an ordered abelian group and $H$ be a convex subgroup.
Take $g_1,g_2,h \in G$ such that $[\![g_1]\!]=[\![g_2]\!]$, $\overline{h}=\overline{g_1}$ and $[\![h]\!] \neq [\![g_1]\!]$.
\begin{enumerate}
\item[(1)] If $h > g_1$, we have $h > g_2$.
\item[(2)] If $h < g_1$, we have $h < g_2$.
\end{enumerate}
\end{lemma}
\begin{proof}
We only prove the assertion (1).
The assertion (2) is proved similarly.

We can take $e < u \in G$ such that $h=g_1u^2$ by the assumption that $\overline{h}=\overline{g_1}$ and $h > g_1$.
We get $u \not\in H$ because $[\![h]\!] \neq [\![g_1]\!]$.
We have $g_1=g_2v^2$ for some $v \in H$ because $[\![g_1]\!]=[\![g_2]\!]$.
Since $H$ is convex, we have $v^{-1}<u$. It follows that $uv > e$.
Thus we obtain $h > g_2$ because $h=(uv)^2g_2$.
\end{proof}

\begin{proposition}\label{prop:convex_basic}
Let $(G,e, \geq)$ be an ordered abelian group and $H$ be a convex subgroup.
The binary relation $\preceq$ in $G/H$ is defined by 
\begin{center}
$[ g] \preceq [ h]$ if and only if $[ g]= [ h]$ or $g<h$.
\end{center}
The binary relation $\preceq$ is a well-defined total order in $G/H$, and $(G/H,[ e ],\preceq)$ is an ordered group.
\end{proposition}
\begin{proof}
We first show the well-definedness of the binary relation $\preceq$.
We have only to show that the inequality $g<h$ and $[g] \neq [h]$ imply the inequalities $gu_1<hu_2$ for all $u_1, u_2 \in H$.
It is also equivalent to the condition that $gu<h$ for all $u \in H$.
We have $hg^{-1} \not\in H$ because $[g] \neq [h]$.
We get $gu<h$ for all $u \in H$ by the definition of convex subgroups.

It is obvious that $(G/H,[e],\preceq)$ is an ordered group once the well-definedness of $\preceq$ is proved.
\end{proof}

\begin{lemma}\label{lem:convex_basic2}
Let $A$ be a ring with $B \subseteq A \subseteq K$.
Set $H:=\val(A^\times)$. 
Then the following statements hold true:
\begin{enumerate}
\item[(1)] $H$ is a convex subgroup of $G$ and $\val(A)= H \cup G_{\geq e}$.
\item [(2)] Fix an element $g\in H\cup G_{\geq e}$. Then we have $Hg\cup G_{\geq g}\subseteq H\cup G_{\geq e}$,
where $Hg=\{hg\in G\;|\;h\in H\}$ and $G_{\geq g}=\{f \in G\;|\; f \geq g\}$.
\end{enumerate}
\end{lemma}
\begin{proof}
(1) The set $H$ is a group because $A^\times$ is a group and the restriction of $\val$ to $A^\times$ is a group homomorphism.

We next demonstrate that, for any $g \in \val(A \setminus A^\times)$ and $h \in H$, we have $g>h$.
Assume for contradiction that $g \leq h$.
Take $x \in A\setminus A^\times$ and $y \in A^\times$ so that $g=\val(x)$ and $h=\val(y)$.
We have $yx^{-1} \in B \subseteq A$ because $\val(yx^{-1})\geq e$.
On the other hand, we get $xy^{-1} \in A$ because $y \in A^\times$.
It follows that $xy^{-1} \in A^\times$ and $x \in A^\times$.
Contradiction.

It is easy to demonstrate that $H$ is convex.
Let $e \leq g \leq h$ and $h \in H$.
We have $g \in \val(B) \subseteq \val(A)$ because $e \leq g$.
We immediately get $g \in H$ by the above claim.
The inclusion $A \setminus A^\times \subseteq B \setminus B^\times$ and the definition of $H$ deduce the last equality $\val(A)= H \cup G_{\geq e}$.

(2) We can take an element $w\in A$ such that $g=\val(w)$ by (1).
The inclusion $Aw\subseteq A$ deduces the inclusion $\val(Aw)\subseteq\val(A)$.
It immediately follows from (1) that $(H\cup G_{\geq e})g\subseteq H\cup G_{\geq e}$.
Hence we have $Hg\cup G_{\geq g}\subseteq H\cup G_{\geq e}$, as desired.
\end{proof}

We introduce the notations
\begin{align*}
&\myConv(G):=\{H\;|\; H \text{ is a convex subgroup of }G\} \text{ and }\\
&\myVal(K,\val):=\{A\;|\; A \text{ is a subring with } B \subseteq A \subseteq K\}\text{.}
\end{align*}

\begin{proposition}\label{prop:cor_convex}
The map $\mathfrak V: \myConv(G) \rightarrow \myVal(K,\val)$ given by  
$$\mathfrak V(H)=\{x \in K\;|\; \val(x) \in H \text{ or } \val(x) \geq e\}$$
is a well-defined bijection whose inverse is given by $\val(A^\times)$ for any $A \in \myVal(K,\val)$.
\end{proposition}
\begin{proof}
Take $H \in \myConv(G)$.
Set $G_0=H \cup G_{\geq e}$ and $A=\mathfrak V(H)$.
We first show that $A \in \myVal(K,\val)$.
Let $x$ and $y$ be arbitrary elements in $A$.
It is obvious that $-x \in A$.
We get $\val(x) \in G_0$ and $\val(y) \in G_0$.
We have $\val(x+y) \geq \min\{\val(x),\val(y)\} \in G_0$.
When one of $\val(x)$, $\val(y)$ and $\val(x+y)$ is not smaller than $e$, we obviously get $\val(x+y) \in G_0$.
When $\val(x+y)<e$, we get $\val(x+y) \in G_0$ because $H$ is convex.
It means that $x+y \in A$.
It is clear that $xy \in A$ when 
$\val(x)$, $\val(y) \in H$  or
$\val(x)$, $\val(y) \in G_{\geq e}$. If 
$\val(x) \in H \setminus G_{\geq e}$ and  $\val(y) \in G_{\geq e} \setminus H$,
we have $\val(x)^{-1}<\val(y)$ because $H$ is convex.
Thus we get $\val(xy)=\val(x)\val(y)>e$ and $xy\in A$.
It is obvious that $B$ is contained in $A$.
We have demonstrated that $A \in \myVal(K,\val)$.

Next we want to demonstrate that $\val((\mathfrak V(H))^\times)=\val(A^\times)=H$.
The definition of $\mathfrak V$ implies that $H \subseteq \val(A^\times)$.
We demonstrate the opposite inclusion.
We have $$G_0 \cap \{x \in G\;|\; x <e\} \subseteq H$$ by the definition of $G_0$.
Take $g \in \val(A^\times)$.
The inclusion is obvious when $g=e$.
We have $g^{-1} \in \val(A^\times)$ because $\val(A^\times)$ is a group.
If $g<e$, we get $g \in H$ because $g \in G_0$ and $g<e$.
If $g>e$, we get $g^{-1} \in H$ for the same reason and obtain $g \in H$ because $H$ is a group. 
We have proven the equality $\val((\mathfrak V(H))^\times)=H$.
This equality immediately yields the injectivity of the map $\mathfrak V$.

For any $A \in \myVal(K,\val)$, we set $H=\val(A^\times)$. Then $H$ is a convex subgroup by Lemma \ref{lem:convex_basic2}(1).
We have only to demonstrate that the equality $A=\mathfrak V(H)$ is satisfied so as to prove the surjectivity of $\mathfrak V$ and the map given by $\val(A^\times)$ is the inverse of $\mathfrak V$.
The inclusion $A \subseteq \mathfrak V(H)$ is obvious from the definition of $\mathfrak V$ and Lemma \ref{lem:convex_basic2}(1).
For the opposite inclusion, take $x \in \mathfrak V(H)$.
We can choose $y \in A$ so that $\val(x)=\val(y)$.
We have $x/y \in B \subseteq A$, and get $x = (x/y) \cdot y \in A$.
\end{proof}

\begin{corollary}\label{cor:cor_convex}
Let $A \in \myVal(K,\val)$ and set $H=\val(A^\times)$.
Then we have $A^\times=\{x \in K\;|\; \val(x) \in H\}$ and $A=\{x \in K\;|\; \val(x) \in H \text{ or }\val(x) \geq e\}$.
\end{corollary}
\begin{proof}
We get the equality $A=\{x \in K\;|\; \val(x) \in H \text{ or }\val(x) \geq e\}$ by Proposition \ref{prop:cor_convex}.
By this equality, an element $x \in A$ with $\val(x) \in H$ is a unit, and $x \in A$ with $\val(x) \not\in H$ is not a unit.
\end{proof}

\begin{corollary}\label{cor:cor_convex01}
Let $A \in \myVal(K,\val)$ and set $H=\val(A^\times)$.
Take an element $x\in K$ with $\val(x)\geq h$ for some $h\in H$. 
Then we have $x\in A$.
\end{corollary}
\begin{proof}
There exists an element $a\in A^{\times}$ such that $h=\val(a)$.
Since $\val(xa^{-1})\geq e$, we have $xa^{-1}\in A$ by Corollary \ref{cor:cor_convex}.
Hence it follows that $x\in A$. 
\end{proof}

\begin{corollary}\label{cor:cor_convex1}
Let $A \in \myVal(K,\val)$ and set $H=\val(A^\times)$.
Let $\val_1:K \rightarrow G/H \cup \{\infty\}$ be the composition of the valuation $\val:K^\times \rightarrow G \cup \{\infty\}$ with the natural surjection $G \rightarrow G/H$.
The map $\val_1:K \rightarrow G/H\cup \{\infty\}$ is a valuation of $K$ whose valuation ring is $A$.
\end{corollary}
\begin{proof}
The set $H$ is a convex subgroup of $G$ by Lemma \ref{lem:convex_basic2}(1).
The group $G/H$ is an ordered group under the ordering $\preceq$ by Proposition \ref{prop:convex_basic}.
Set $G_0=H \cup G_{\geq e}$.
It is a routine to demonstrate that $\val_1$ is a valuation.
We have 
\begin{align*}
&\val_1^{-1}(\{[g] \in G/H\;|\; [g] \succeq [e]\})
=\val_1^{-1}(\{[g]\in G/H\;|\; g\in G_0\})
=\val^{-1}(G_0)=A
\end{align*}
by Corollary \ref{cor:cor_convex}.
It means that the valuation ring of the valuation $\val_1$ is $A$.
\end{proof}

\begin{corollary}\label{cor:cor_convex2}
Let $A \in \myVal(K,\val)$ and set $H=\val(A^\times)$.
Let $K_A$ be the residue field of $A$.
Let $\val_2:K_A \rightarrow H \cup \{ \infty\}$ be the map given by $\val_2([x]_A)=\val(x)$ for all $x \in A^\times$ and $\val_2([x]_A)=\infty$ for all $x \in \mathfrak m_A$, where $\mathfrak m_A$ is the maximal ideal of $A$ and $[x]_A$ is the equivalence class in $K_A$ of $x$.
It is a well-defined valuation whose valuation ring is $B/(\mathfrak m_A \cap B)$.
\end{corollary}
\begin{proof}
We first demonstrate the well-definedness of $\val_2$.
Set $G_0=H \cup G_{\geq e}$.
Let $x_1$ and $x_2$ be elements in $A^\times$ such that $[x_1]_A=[x_2]_A$.
We have $\val(x_1), \val(x_2) \in H$ and 
also have $\val(x_2 - x_1) \in G_0 \setminus H$ by Corollary \ref{cor:cor_convex}.
If $\val(x_1)\neq\val(x_2)$, say $\val(x_1)<\val(x_2)$, then $\val(x_2-x_1)=\val(x_1)\in H$, contradiction.
Hence we get $\val(x_1)=\val(x_2)$.
We have proven the well-definedness.
It is easy to validate that $\val_2$ is a valuation.
We omit the details.

We show that $B/(\mathfrak m_A \cap B)$ is the valuation ring of $\val_2$.
We get 
\begin{align*}
&\val_2^{-1}(\{g \in H\cup \{\infty\}\;|\;g \geq e\}) \\
&=\{[x]_A \in K_A\;|\;(\val (x) \in H \text{ and } \val(x) \geq e) \text{ or } x \in \mathfrak m_A\} \\
&=\{[x]_A \in K_A\;|\;(\val (x) \in H \text{ and } \val(x) \geq e) \text{ or } \val_1(x)\succ [e]\}\\
&=\{[x]_A \in K_A\;|\;(\val (x) \in H \text{ and } \val(x) \geq e) \text{ or } (\val(x) \in G \setminus H \text{ and } \val(x) > e)\}\\
&=\{[x]_A \in K_A\;|\;\val(x) \geq e\} = B/(\mathfrak m_A \cap B)
\end{align*}
by Corollary \ref{cor:cor_convex1}.
We have demonstrated the corollary.
\end{proof}

\subsection{Pseudo-angular component maps}\label{sec:pseudo-angular}
\begin{definition}\label{def:pseudo}
A \textit{pseudo-angular component map} is a map $\mylcp:K^{\times} \rightarrow F^{\times}$ satisfying the following conditions:
\begin{enumerate}
\item[(1)] We have $\mylcp(u)=\pi(u)$ for any $u \in B^{\times}$;
\item[(2)] The equality $\mylcp(ux)=\pi(u) \cdot \mylcp(x)$ holds true for all $u \in B^{\times}$ and $x \in K^{\times}$;
\item[(3)] For any $g \in G$ and nonzero $c \in F$, there exists an element $w \in K$ with $\val(w)=g$ and $\mylcp(w)=c$;
\item[(4)] For any nonzero elements $x_1,x_2 \in K$ with $x_1+x_2 \not=0$, we have
\begin{itemize}
\item $\mylcp(x_1+x_2)=\mylcp(x_1)$ if $\val(x_1) < \val(x_2)$;
\item $\val(x_1+x_2)=\val(x_1)$ and $\mylcp(x_1+x_2)=\mylcp(x_1)+\mylcp(x_2)$ if $\val(x_1) = \val(x_2)$ and $\mylcp(x_1)+\mylcp(x_2) \not=0$;
\end{itemize}
\item[(5)] For any $x,y \in K^{\times}$ with $\overline{\val(x)}=\overline{\val(y)}$ and $\mylcp(x)=\mylcp(y)$, there exists $u \in K^{\times}$ such that $y=u^2x$;
\item[(6)] For any $a,u \in K^{\times}$, there exists $k \in F^{\times}$ such that $\mylcp(au^2)=\mylcp(a)k^2$.
\end{enumerate}
A pseudo-angular component map that is a group homomorphism is said to be an {\it angular component map}.
We denote by $\mylc$ an angular component map.
\end{definition}

\begin{remark}\label{rem:angular}
When a valued field whose strict units always admit a square root, any group homomorphism $\mylc:K^\times \rightarrow F^\times$ whose restriction to $B^\times$ is $\pi$ is an angular component map.
\end{remark}
\begin{proof}
It is obvious that the conditions (1), (2) and (6) in Definition \ref{def:pseudo} are satisfied.

We demonstrate that the map $\mylc$ satisfies the condition Definition \ref{def:pseudo}(3).
Let $g \in G$ and $0 \not=c \in F$ be arbitrary elements.
By the definition of valued field, there exists a nonzero element $x\in K$ with $\val(x)=g$.
Set $d=\mylc(x)$.
We can find a nonzero $y \in B^{\times}$ with $\pi(y)=cd^{-1}$.
Set $w=x \cdot y$.
We get $\val(w)=\val(x) \cdot \val(y)=g$ and $\mylc(w)=\mylc(x) \cdot \mylc(y)=c$, as desired.

We next check the condition Definition \ref{def:pseudo}(4) is satisfied.
Take nonzero elements $x_1,x_2 \in K^{\times}$ with $x_1+x_2 \not=0$.
We first consider the case in which $\val(x_1) < \val(x_2)$.
%The equality $\val(x_1+v_2)=\val(x_1)$ immediately follows from the definition of valuations.
Since we have $\val\left(\dfrac{x_2}{x_1}\right)>e$,
we get $\val\left(1+\dfrac{x_2}{x_1}\right)=e$. Hence
we have 
\begin{align*}
\mylc(x_1+x_2) &= \mylc(x_1) \cdot \mylc\left(1+\dfrac{x_2}{x_1}\right) = \mylc(x_1)\cdot \pi\left(1+\dfrac{x_2}{x_1}\right) 
=\mylc(x_1)\text{.}
\end{align*}

We next consider the remaining case.
Set $g=\val(x_1)=\val(x_2)$.
We can take a nonzero element $w \in K$ with $\val(w)=g$ and $\mylc(w)=1$ by Definition  \ref{def:pseudo}(3). 
The element $(x_1+x_2)\cdot w^{-1}$ is contained in $B$ because we have
\begin{align*}
\val((x_1+x_2)\cdot w^{-1}) &=\val(x_1\cdot w^{-1}+x_2\cdot w^{-1})\\
& \geq \min \{\val(x_1\cdot w^{-1}),\val(x_2\cdot w^{-1})\} = e\text{.}
\end{align*}
We get $\pi((x_1+x_2)\cdot w^{-1})=\pi(x_1 \cdot w^{-1})+\pi(x_2 \cdot w^{-1})=\mylc(x_1)+\mylc(x_2) \not=0$.
The element $(x_1+x_2)\cdot w^{-1}$ is a unit in $B$.
It means that $\val((x_1+x_2)\cdot w^{-1})=e$; that is, $\val(x_1+x_2)=g$.
The equality $\mylc(x_1+x_2)=\mylc(x_1)+\mylc(x_2)$ follows from the following calculation:
\begin{align*}
\mylc(x_1+x_2) &= \dfrac{\mylc(x_1+x_2)}{\mylc(w)}=\mylc\left(\dfrac{x_1+x_2}{w}\right) = \pi\left(\dfrac{x_1+x_2}{w}\right)\\
&=\pi\left(\dfrac{x_1}{w}\right)+\pi\left(\dfrac{x_2}{w}\right) = \mylc\left(\dfrac{x_1}{w}\right)+\mylc\left(\dfrac{x_2}{w}\right)\\
& = \dfrac{\mylc(x_1)}{\mylc(w)} +\dfrac{\mylc(x_2)}{\mylc(w)} = \mylc(x_1)+\mylc(x_2)\text{.}
\end{align*}

The remaining task is to demonstrate that the map $\mylc$ satisfies the condition Definition \ref{def:pseudo}(5).
Let $x_1,x_2$ be nonzero elements of $K$ such that 
$\overline{\val(x_1)}=\overline{\val(x_2)}$ and $\mylc(x_1)=\mylc(x_2)$.
Since $\overline{\val(x_1)}=\overline{\val(x_2)}$, we can get $h \in G$ with 
$\val(x_1)=\val(x_2)h^2$.
Set $c=\mylc(x_1)=\mylc(x_2)$.
Take a nonzero element $w \in K$ with $\val(w)=h$ and $\mylc(w)=1$ by Definition \ref{def:pseudo}(3). 
Consider the element $y=x_1 \cdot (x_2 \cdot w^2)^{-1}$.
We have $\val(y)=\val(x_1) \cdot (\val(x_2) \cdot \val(w)^2)^{-1} = \val(x_1) \cdot (\val(x_2)h^2)^{-1}=e$.
We also get 
\begin{align*}
\pi(y)=\mylc(y)=\dfrac{\mylc(x_1)}{\mylc(x_2) \cdot \mylc(w)^2} = \dfrac{c}{c \cdot 1^2}=1\text{.}
\end{align*}
The element $y$ is a strict unit.
Therefore, there exists an element $z\in K$ with $y=z^2$
by the assumption.
Set $u=w\cdot z$, then we have $x_1=x_2u^2$.
\end{proof}
The above remark is familiar to anyone with an in-depth understanding of valued fields.
In fact, the equivalent assertions are found in \cite[Appendix A]{S}.
We proved them here for the sake of completeness.

The following proposition is used without citation in this paper.
\begin{proposition}\label{prop:psudo-angular}
A valued field whose strict units always admit a square root has a pseudo-angular component map.
\end{proposition}
\begin{proof}
Let $(K,\val)$ be the valued field.
We fix a map $\omega:G \rightarrow K^{\times}$ such that $\val(\omega(g)) = g$ and $\omega(e)=1$.
It always exists because the valuation $\val$ is surjective.
We also fix a representative map $r:G/G^2 \rightarrow G$ such that $r(\overline{g})$ is a representative of the equivalence class $\overline{g}$ with $r(\overline{e})=e$.
For any $x \in K^{\times}$, 
set $$\mylcp(x)=\pi\left(\dfrac{x}{\omega(r(\overline{\val(x)}))\omega(h)^2}\right)\text{,}$$
where $h \in G$ with $\val(x)=r(\overline{\val(x)})h^2$.
Note that the element $h\in G$ is uniquely determined by $x$.
Since we have $\val(x/\omega(r(\overline{\val(x)}))\omega(h)^2)=e$, 
it follows that the map $\mylcp:K^{\times}\to F^{\times}$ is well-defined.

We demonstrate that the above map satisfies the conditions (1) through (6) in Definition \ref{def:pseudo}.
It is easy to check that the conditions (1) and (2) are satisfied.
We omit the proof.
We consider the condition (3).
Fix $g \in G$ and $c \in F^{\times}$.
Set $d=c/\mylcp(\omega(g))$ and take $v \in B^{\times}$ with $\pi(v)=d$.
The element $w=v\omega(g)$ satisfies the equalities $\val(w)=g$ and $\mylcp(w)=c$ by condition (2).

The next target is the condition (4).
We take $x_1,x_2 \in K^{\times}$ with $x_1+x_2 \neq 0$.
When $\val(x_1)<\val(x_2)$, we have $\val(x_1+x_2)=\val(x_1)$.
Take $h \in G$ with $\val(x_1+x_2) = r(\overline{\val(x_1+x_2)})h^2$.
We get
\begin{align*}
\mylcp(x_1+x_2) &= \pi\left(\dfrac{x_1+x_2}{\omega(r(\overline{\val(x_1+x_2)}))\omega(h)^2}\right)
= \pi\left(\dfrac{x_1+x_2}{\omega(r(\overline{\val(x_1)}))\omega(h)^2}\right)\\
&=\pi\left(\dfrac{x_1}{\omega(r(\overline{\val(x_1)}))\omega(h)^2}\right)+ \pi\left(\dfrac{x_2}{\omega(r(\overline{\val(x_1)}))\omega(h)^2}\right)\\
&=\pi\left(\dfrac{x_1}{\omega(r(\overline{\val(x_1)}))\omega(h)^2}\right)=\mylcp(x_1)\text{.}
\end{align*}

When $\val(x_1) = \val(x_2)$ and $\mylcp(x_1)+\mylcp(x_2) \not=0$, we first demonstrate that $\val(x_1+x_2)=\val(x_1)$.
Otherwise, we have $\val(x_1+x_2)>\val(x_1)$.
In particular, we get $\val((x_1+x_2)/x_1)>e$.
Hence, we get $\pi(-x_2/x_1)=\pi(1-(x_1+x_2)/x_1))=\pi(1)=1$.
It means that $-x_2/x_1$ is a strict unit and there exists $u \in K^{\times}$ with $x_2=-x_1u^2$
% because $(K,\val)$ is $2$-henselian.
because any strict units in $(K,\val)$ admit a square root.
It is easy to check that $u \in B^{\times}$ and $\pi(u)^2=1$.
We then have $\mylcp(x_1)+\mylcp(x_2)=\mylcp(x_1)+\mylcp(-u^2x_1)=\mylcp(x_1)-\pi^2(u)\mylcp(x_1)=\mylcp(x_1)-\mylcp(x_1)=0$ by condition (2).
It contradicts the assumption that $\mylcp(x_1)+\mylcp(x_2) \not=0$.
We have demonstrated that $\val(x_1+x_2)=\val(x_1)$.

We now get
\begin{align*}
\mylcp(x_1+x_2) &= \pi\left(\dfrac{x_1+x_2}{\omega(r(\overline{\val(x_1+x_2)}))\omega(h)^2}\right)
= \pi\left(\dfrac{x_1+x_2}{\omega(r(\overline{\val(x_1)}))\omega(h)^2}\right)\\
&=\pi\left(\dfrac{x_1}{\omega(r(\overline{\val(x_1)}))\omega(h)^2}\right)+ \pi\left(\dfrac{x_2}{\omega(r(\overline{\val(x_1)}))\omega(h)^2}\right)\\
&=\mylcp(x_1)+\mylcp(x_2)\text{.}
\end{align*}

We next treat the condition (5).
Take $x,y \in K^{\times}$ with $\overline{\val(x)}=\overline{\val(y)}$ and $\mylcp(x)=\mylcp(y)$.
Set $\overline{g}=\overline{\val(x)}=\overline{\val(y)}$.
We can take $h_1,h_2 \in G$ with $\val(x)=r(\overline{g})h_1^2$ and $\val(y)=r(\overline{g})h_2^2$.
We get $$\pi\left(\dfrac{x}{\omega(r(\overline{g}))\omega(h_1)^2}\right)
=\pi\left(\dfrac{y}{\omega(r(\overline{g}))\omega(h_2)^2}\right)$$ because $\mylcp(x)=\mylcp(y)$.
Set $t=y/x \cdot (\omega(h_1)/\omega(h_2))^2$.
Since $t$ is a strict unit, we have $t=v^2$ for some $v \in B^{\times}$.
Set $u=v\cdot \omega(h_2)/\omega(h_1)$.
We have $y=xu^2$.

The remaining task is to show that the map $\mylcp$ satisfies condition (6).
Take $a,u \in K^{\times}$.
Set $\overline{g}=\overline{\val(a)}=\overline{\val(au^2)} \in G/G^2$.
Take $h_1,h_2 \in G$ such that $\val(a)=r(\overline{g})h_1^2$ and $\val(au^2)=r(\overline{g})h_2^2$.
We obviously have $\val(u)=h_2/h_1$.
We get 
\begin{align*}
\mylcp(au^2)&=\pi\left(\dfrac{au^2}{\omega(r(\overline{g}))\omega(h_2)^2}\right) = \left(\pi\left(\dfrac{u\cdot \omega(h_1)}{\omega(h_2)}\right)\right)^2\pi\left(\dfrac{a}{\omega(r(\overline{g}))\omega(h_1)^2}\right)\\
&= k^2\mylcp(a) \text{,}
\end{align*}
where $k=\pi\left(\dfrac{u\cdot \omega(h_1)}{\omega(h_2)}\right)$.
\end{proof}

We fix a pseudo-angular component map $\mylcp:K^{\times} \rightarrow F^{\times}$ through the paper.
 
We prove a lemma which holds true regardless of the characteristic of $F$.

\begin{lemma}\label{lem:psudo_lemlem}
Let $(K,\val)$ be a valued field whose strict units always admit a square root.
%and $\mylcp:K^\times \rightarrow F^\times$ be its pseudo-angular component map.
Let $x_1,x_2 \in K^\times$.
If $\val(x_1)=\val(x_2)$ and $\mylcp(x_1)=\mylcp(x_2)$, we have $x_2=x_1u^2$ for some $u \in B^\times$ with $\pi(u^2)=1$.
\end{lemma}
\begin{proof}
We can take $u \in K^\times$ with $x_2=u^2x_1$ by Definition \ref{def:pseudo}(5).
We have $\val(u)=e$ because $\val(x_1)=\val(x_2)$.
It means that $u \in B^{\times}$.
We get $\mylcp(x_2)=\mylcp(u^2x_1)=\pi(u^2)\mylcp(x_1)$ by Definition \ref{def:pseudo}(2).
We also have $\mylcp(x_2)=\mylcp(x_1)$ by the assumption.
They imply that $\pi(u^2)=1$.
\end{proof}

\begin{corollary}\label{cor:psudo_lemlem}
Let $(K,\val)$ be a valued field whose strict units always admit a square root.
% and $\mylcp:K^\times \rightarrow F^\times$ be its pseudo-angular component map.
Let $x_1,x_2 \in K^\times$.
If $\val(x_1)=\val(x_2)$ and $\mylcp(x_1)+\mylcp(x_2)=0$, we have $\val(x_1+x_2)>\val(x_1)$.
\end{corollary}
\begin{proof}
We can take $u \in B^\times$ with $x_2=-u^2x_1$ and $\pi(u^2)=1$ by Lemma \ref{lem:psudo_lemlem}.
Setting $v=u^2-1$, we have $\val(v)>e$.
Since $x_1+x_2=x_1-u^2x_1=-vx_1$, we have $\val(x_1+x_2)=\val(v) \cdot \val(x_1)>\val(x_1)$.
\end{proof}

Recall that  a valued field $(K,\val)$ whose residue class field is of characteristic $\neq 2$ is $2$-henselian
if and only if  any strict units admit a square root by \cite[Corollary 4.2.4]{EP}.

\begin{lemma}
\label{coro:basic004}
Let $(K,\val)$ be a $2$-henselian valued field whose residue class field is of characteristic $\not=2$ and $B$ be its valuation ring.
Let $A$ be a subring with $B \subseteq A \subseteq K$.
Set $H=\val(A^\times)$. Let $g$ be an element of $H\cup G_{\geq e}$.
Then there exists an element $w \in A$ such that $\val(w)=g$, $\mylcp(w)=1$ and $x_1+x_2\in Aw$
for any nonzero elements $x_1,x_2\in A$ with $\val(x_1)=\val(x_2)$ such that at least one of the conditions $\val(x_2)\geq g$ and $\val(x_2)=[\![g]\!]$ is satisfied.
\end{lemma}
\begin{proof}
By Definition \ref{def:pseudo}(3), we can take an element $w\in K$ with $\val(w)=g$
and $\mylcp(w)=1$. We see that $w\in A$ from Corollary \ref{cor:cor_convex}.

We first consider the case in which $\mylcp(x_1)+\mylcp(x_2)= 0$.
Since $\mylcp(x_1)=\mylcp(-x_2)$,
there exists an element $u\in K$ such that $x_1=-x_2u^2$
by Definition \ref{def:pseudo}(5). 
We get $-u^2 \in B^{\times}$ because $\val(x_1)=\val(x_2)$.
We have
$x_1+x_2=(1-u^2)x_2$ and $\val(1-u^2)\geq e$.
If $\val(x_2)\geq g$, we have $\val((x_1+x_2)w^{-1})=\val(1-u^2)\val(x_2)\val(w^{-1})\geq e$.
Hence we get $(x_1+x_2)w^{-1}\in A$ from Corollary \ref{cor:cor_convex}. This implies that
$x_1+x_2\in Aw$.
If $\val(x_2)=[\![g]\!]$, there exists an element $h\in H$ such that $\val(x_2)=h^2g$.
It follows that $\val((x_1+x_2)w^{-1})\geq \val(x_2)\val(w)^{-1}=h^2$.
Thus we have $(x_1+x_2)w^{-1}\in A$ from Corollary \ref{cor:cor_convex01}.
Hence we have $x_1+x_2\in Aw$. 

We consider the case in which $\mylcp(x_1)+\mylcp(x_2)\neq 0$. 
If $x_1+x_2=0$, we have  $\mylcp(x_1)+\mylcp(x_2)= 0$.
This contradicts the assumption. Hence it follows that $x_1+x_2\neq 0$.
By Definition \ref{def:pseudo}(4), we have $\val(x_1+x_2)=\val(x_2)$.
We can take an element $h\in H$ with $\val((x_1+x_2)w^{-1})\geq h$ in any case.
It follows that $(x_1+x_2)w^{-1}\in A$ by Corollary \ref{cor:cor_convex01}.
Thus we have $x_1+x_2\in Aw$.
\end{proof}

A \textit{cross section} $\omega:G \rightarrow K^{\times}$ is a group homomorphism such that the composition 
$\val \circ\; \omega$ is the identity map.
A valued field does not necessarily have a cross section as demonstrated in \cite{K}.
We give sufficient conditions for a valued field $(K,v)$ to have a cross section.

\begin{proposition}\label{prop:wmap}
A valued field $(K,\val)$ has a cross section $\omega$ in the following cases:
\begin{enumerate}
\item[(1)] The valued field $K$ is the field of Hahn series $F(\!(G)\!)$.
\item[(2)] The value group $G$ has a system of generators such that any element in $G$ is uniquely represented by the generators.
\item[(3)] The valued field $K$ is a henselian valued field with divisible value group $G$ and
the residue class field $F$ is of characteristic zero such that $F^{\times}$ is divisible.
\end{enumerate}
\end{proposition}
\begin{proof}
(1) The proposition is obvious. The cross section $\omega$ is defined by $\omega(g)=g$ for all $g\in G$.

(2)
Let $\{g_{\lambda}\}_{\lambda \in \Lambda}$ be a system of generators such that any element in $G$ is uniquely represented by the generators.
Find an element $w(g_{\lambda}) \in K$ so that $\val(\omega(g_{\lambda}))=g_{\lambda}$ for any $\lambda \in \Lambda$.
Let $g$ be an element of $G$ and $g=\prod_{\lambda \in \Lambda} g_{\lambda}^{\sigma_\lambda}$ be the unique representation of $g$ by the generators, where $\sigma_\lambda \in \mathbb Z$ and $\sigma_\lambda=0$ except for a finite number of $\lambda$.
Set $\omega(g)=\prod_{\lambda \in \Lambda} \omega(g_{\lambda})^{\sigma_\lambda}$.
It is obvious that the map $w$ is a group homomorphism.

(3) We first demonstrate that the abelian group $K^{\times}$ is divisible.
Take a element $a\in K^{\times}$ and a positive integer $n$.
We have to prove $a=b^n$ for some $b\in K$. Since $G$ is divisible, we have $\val(a)=g^n$ for some $g\in G$.
There exists an element $d\in K^{\times}$ with $\val(d)=g$.
It follows that $\val(a)=\val(d^n)$. Thus we can take an element $u\in B^{\times}$ with
$a=ud^n$. From the divisibility of $F^{\times}$ we have an element $w\in B^{\times}$ with
$\pi(w)^n=\pi(u)$. Consider the polynomial $f(t)=t^n-u\in B[t]$.
It immediately follows that $\pi(f(w))=0$ and $\pi(f'(w))=n\pi(w)^{n-1}\neq 0$ in $F^{\times}$
because $F^{\times}$ is of characteristic zero. Since the valued field $K$ is henselian,
we can take an element $c\in B$ with $u=c^n$. Hence we have $a=(cd)^n$. 

The group $B^{\times}$ is divisible because $K^{\times}$ is divisible.
It follows from \cite[Theorem 2]{Kap} that $B^{\times}$ is a direct summand of $K^{\times}$.
We see that the exact sequence
$$1\to B^{\times}\to K^{\times}\xrightarrow{\val} G\to e$$
is split. Hence we get a group homomorphism $\omega: G\to K^{\times}$ with $\val\circ\;\omega$ is the
identity map.
\end{proof}

\begin{example}
The condition (2) in Proposition \ref{prop:wmap} is satisfied by the fundamental theorem of abelian groups when $G$ is finitely generated because an ordered abelian group $G$ is always torsion-free.
\end{example}

We demonstrate that a $2$-henselian valued field $(K,\val)$
whose residue class field is of characteristic $\neq 2$ has an angular component map when it has a cross section.

\begin{proposition}\label{prop:angular}
Let $(K,\val)$ be a valued field which has a cross section and $B$ be its valuation ring.
Assume that $(K, \val)$ is a $2$-henselian valued field whose residue class field is of characteristic $\neq 2$.
Then it admits an angular component map.
\end{proposition}
\begin{proof}
Let $\omega:G\to K^{\times}$ be a cross section.
We define the map $\mylc:K^{\times} \rightarrow F^{\times}$ by $$\mylc(x)=\pi(x \cdot \omega(\val(x))^{-1})\text{.}$$

We demonstrate that the above map is an angular component map.
It is sufficient from Remark \ref{rem:angular} to demonstrate that the map
$\mylc:K^{\times} \rightarrow F^{\times}$ is a group homomorphism and its restriction to $B^{\times}$ is $\pi$. 
For any non-zero elements $x, y\in K^{\times }$, we have
\begin{align*}
\mylc(x \cdot y) &= \pi \left( \dfrac{x \cdot y}{\omega(\val(x \cdot y))}\right) = \pi \left( \dfrac{x \cdot y}{\omega(\val(x)) \cdot \omega(\val(y))}\right) \\&= \pi \left( \dfrac{x }{\omega(\val(x))}\right) \cdot  \pi \left( \dfrac{y }{\omega(\val(y))}\right) \\
&=\mylc(x) \cdot \mylc(y).
\end{align*}
When $x \in B^{\times}$, we have $\val(x)=e$ and $\omega(\val(x))=1$ because $w$ is a group homomorphism.
In particular, we have $\mylc(x)=\pi(x \cdot \omega(\val(x))^{-1})=\pi(x)$.
\end{proof}

\subsection{Other basic lemmas}

Let $M$ be a quasi-quadratic module in a commutative ring $R$.
We define the {\it support} of 
$M$ as follows:
$$
\mbox{supp}(M)=M \cap (-M),
$$
where $-M=\{x \in R\;|\;-x \in M\}$.
When $M$ is a quadratic module, the following lemma is well-known as in \cite[Proposition 5.1.3]{PD}
\begin{lemma}\label{lem:basic4}
Let $M$ be a quasi-quadratic module in a commutative ring $R$ such that $2$ is a unit.
The set $\supp(M)$ is an ideal of $R$.
\end{lemma}
\begin{proof}
It is easy to demonstrate that $x+y \in \supp(M)$ when $x, y \in \supp(M)$.
Let $a$ be an arbitrary element of $R$ and $x$ be an element of $\supp(M)$.
We have $ax = \left(\dfrac{a+1}{2}\right)^2 \cdot x+ \left(\dfrac{a-1}{2}\right)^2 \cdot (-x) \in\supp(M)$ because $\pm x \in\supp(M)$.
\end{proof}

\begin{corollary}
\label{cor:basic4'}
Let $M$ be a quasi-quadratic module in a commutative ring $R$ such that the element $2$
is a unit. Then the following statements hold true:

(1) If there exists an element $b\in R$ such that $\pm b\in M$, then $M\supseteq 
Rb$.

(2) If there exists a unit $c\in R$ such that $\pm c\in M$,
then $M=R$.
\end{corollary}
\begin{proof}
(1) Since $b\in\supp(M)$, we have  $Rb\subseteq\supp(M)\subseteq M$ by Lemma \ref{lem:basic4}.

(2) Applying (1) to the unit $c$, we have $R=Rc\subseteq M$.
\end{proof}

We apply Corollary \ref{cor:basic4'} to a quasi-quadratic module in a field.
\begin{corollary}\label{lem:field0}
Let $M$ be a quasi-quadratic module in a field $E$ of characteristic $\not=2$.
Assume that there exists a nonzero element $c \in E$ with $\pm c \in M$.
Then, we have $M=E$.
\end{corollary}
\begin{proof}
Immediate from Corollary \ref{cor:basic4'}(2).
\end{proof}

We also obtain from Corollary \ref{lem:field0} the following:
\begin{corollary}\label{lem:field000}
Let $M_1$ and $M_2$ be proper quasi-quadratic modules in a field $E$  of characteristic $\not=2$.
Assume that $M_1+M_2=E$.
Then, there exists a nonzero element $c \in E$ with $c \in M_1$ and $-c \in M_2$.
\end{corollary}
\begin{proof}
It is obvious that $M_1 \not=\{0\}$ because $M_1$ and $M_2$ are proper and $M_1+M_2=E$.
Take a nonzero element $d \in M_1$.
Since $M_1+M_2=E$, there exist $d_1 \in M_1$ and $d_2 \in M_2$ with $-d = d_1+d_2$.
If $d_2 =0$, we have $\pm d \in M_1$, and we get $M_1=E$ by Corollary \ref{lem:field0}.
It contradicts the assumption that $M_1$ is proper.
We have shown that $d_2 \not=0$.
Set $c=-d_2=d+d_1$.
It is nonzero and we have $c=d+d_1 \in M_1$ and $-c=d_2 \in M_2$.
\end{proof}

\section{Quasi-quadratic modules in ring $A$}\label{sec:main}
We fix a subring $A$ such that $B \subseteq A \subseteq K$ in this section.
We use the notation $H=\val(A^\times)$ throughout this section.
It is a convex subgroup of $G$ by Lemma \ref{lem:convex_basic2}(1).

\begin{notation}
For simplicity, the notation ``$\val(x)=\overline{g}$'' denotes the condition that $(\val(x) \mod G^2)
=\overline{g}$ for some $g\in G$. 
We define the notations ``$\val(x)=[g]$'' and ``$\val(x)=[\![g]\!]$'' similarly.
We use these notations throughout the paper.
\end{notation}

\begin{definition}
\label{Phi-M_g}
Let $(K,\val)$ be a $2$-henselian valued field whose residue class field is of characteristic $\neq 2$.
Let $A$ be a subring with $B \subseteq A \subseteq K$. Set $H=\val(A^{\times})$.
Consider an element $g \in G$ and a quasi-quadratic module $M$ 
in the residue class field $F$. 
The subset $\Phi^A(M,[\![g]\!])$ of $A$ is defined as follows:
\begin{align*}
\Phi^A(M,[\![g]\!]) 
&= \{x \in A\setminus\{0\} \;|\; 
\val(x) =\overline{g}\text{, } (\val(x)=[\![g]\!] \text{ or }\val(x) > g)\\
&\qquad \text{ and } \mylcp(x) \in M \}
\cup \{0\}.
\end{align*}
The pseudo-angular component map $\mylcp:K^{\times} \rightarrow F^{\times}$ exists because
of Proposition \ref{prop:psudo-angular} and the text before Lemma \ref{coro:basic004}. 
The right-hand of the equality is independent of the representation $g$ of $[\![g]\!]$ by Lemma \ref{lem:convex_basic12}(1).
The subset $\Phi^A(M,[\![g]\!])$ is called the \textit{quasi-quadratic module of $A$
generated by $M$ and $[\![g]\!]\in G/H^2$}.
It is simply denoted by $\Phi(M,[\![g]\!])$ when $A$ is clear from the context.

We next define quasi-quadratic modules in $F$ 
constructed from a quasi-quadratic module $\mathcal M$ in $A$.
Here, $(K,\val)$ denotes a valued field whose strict units always admit a square root.
We set as follows:
$$M_{g}^A(\mathcal M) = \{\mylcp(x)\in 
F\setminus\{0\}\;|\; x \in \mathcal M
\setminus\{0\} \text{ with } \val(x)= g \} \cup \{0\}
$$
for all $g \in G$.
The subset $M_{g}^A(\mathcal M)$ is called the 
\textit{quasi-quadratic module of $F$
generated by $\mathcal M$ and $g\in G$}.
It is simply denoted by $M_{g}$ when $A$ and $\mathcal M$ are clear from the context.
\end{definition}

The following proposition claims that the subset $\Phi^A(M,[\![g]\!])$ is a quasi-quadratic module
in $A$
when $M$ is a proper quasi-quadratic module in the residue class field $F$.

\begin{proposition}\label{prop:ring1}
Let $(K,\val)$ be a $2$-henselian valued field whose residue class field is of characteristic $\not=2$ and $B$ be its valuation ring.
Let $A$ be a subring with $B \subseteq A \subseteq K$.
Let $g\in G$ and $M$ be a proper quasi-quadratic module in the residue class field $F$.
Then the set $\Phi^A(M,[\![g]\!])$ is a quasi-quadratic module in $A$.
\end{proposition}
\begin{proof}
Set $\mathcal M = \Phi^A(M, [\![g]\!])$ and $H=\val(A^\times)$.
%The proposition is obvious when $g \not\in \val(A)$ because $\mathcal M=\{0\}$.
%We only consider the case in which $g \in \val(A)$ in the rest of the proof.
We first show that $\mathcal M$ is closed under multiplication by the squares of elements in $A$.
Let $u \in A$ and $x \in \mathcal M$.
If $u=0$ or $x=0$, it is obvious that $u^2x \in \mathcal M$.
Consider the case in which $u \not=0$ and $x \not=0$.
It is obvious that $\val(u^2x) = \val(u)^2\val(x) = \overline{g}$.
When $\val(x) =[\![g]\!]$ and $u \in A^\times$, we obviously have $\val(u^2x) =[\![g]\!]$.
When $\val(x) =[\![g]\!]$ and $u \not\in A^\times$, we have $\val(u)>e$ because
$u\in A\setminus A^{\times}\subseteq B\setminus B^{\times}$. Note that $\val(u)\not\in H$ by
Corollary \ref{cor:cor_convex}.
We can take $h\in H$ with $\val(x)=gh^2$. Since $H$ is convex by Lemma \ref{lem:convex_basic2}(1), we have $h^{-1}<\val(u)$. Thus we get $\val(u)h>e$.
It follows that $\val(u^2x)=(\val(u)h)^2g>g$.
We next consider the case in which $\val(x) \neq [\![g]\!]$ and $u \in A^\times$.
Since $\val(x)=\overline{g}$ and $\val(x)>g$, we can take $h \in G$ with $h>e$ and $\val(x)=gh^2$.
We get $h \not\in H$ because $\val(x) \neq [\![g]\!]$.
Since $H$ is convex, we have $h > \val(u)^{-1}$.
We get $\val(u^2x)=gh^2\val(u)^{2}>g$.
It is obvious that $\val(u^2x) > g$ when $\val(x) \neq [\![g]\!]$ and $u \not\in A^\times$.
We have $\mylcp(u^2x)=k^2 \cdot \mylcp(x) \in M$ for some $k \in F^\times$ by Definition \ref{def:pseudo}(6).
We have shown that $u^2x \in \mathcal M$.

Let $x_1, x_2 \in \mathcal M$.
When $x_1+x_2=0$, $x_1=0$ or $x_2=0$, it is obvious that $x_1+x_2 \in \mathcal M$.
We next demonstrate that $x_1+x_2 \in \mathcal M$ when $x_1+x_2 \not=0, x_1\neq 0$ and $x_2\neq 0$.
We first consider the case in which $\val(x_1) \not= \val(x_2)$.
We may assume that $\val(x_1)<\val(x_2)$ by symmetry.
We have $\val(x_1+x_2)=\val(x_1)$.
We have $\overline{\val(x_1+x_2)}=\overline{\val(x_1)} = \overline{g}$ and one of the conditions that $\val(x_1+x_2)=\val(x_1)=[\![g]\!]$ and $\val(x_1+x_2)=\val(x_1)>g$ is satisfied.
We also get $\mylcp(x_1+x_2)=\mylcp(x_1) \in M$ by Definition \ref{def:pseudo}(4).
We obtain $x_1+x_2 \in \mathcal M$.

We next consider the case in which $\val(x_1)=\val(x_2)=g$.
Set $c_i = \mylcp(x_i)$ for $i=1,2$.
When $c_1+c_2=0$, we get $M=F$ by Corollary \ref{lem:field0}.
Contradiction to the assumption that $M$ is proper.
Hence we see that $c_1+c_2 \not=0$.
We have $\val(x_1+x_2)=g$ and $\mylcp(x_1+x_2)=c_1+c_2$ by Definition \ref{def:pseudo}(4).
Since $c_1, c_2 \in M$, we also have $\mylcp(x_1+x_2) \in M$.
We obtain $x_1+x_2 \in \mathcal M$.
\end{proof}

\begin{remark}\label{rem:properness2}
As seen above, in the proof of Proposition \ref {prop:ring1}, 
the assumption that $M$ is proper is not necessary
when we show that $\mathcal M$ is closed under multiplication by the squares of elements of $A$.
\end{remark}

The next lemma is a direct consequence of the definition.

\begin{lemma}\label{lem:triv}
Let $(K,\val)$ and $A$ be the same as in Proposition \ref {prop:ring1}.
Let $M_1$ and $M_2$ be quasi-quadratic modules in the residue class field $F$ with $M_1 \subseteq M_2$.
For any $g_1, g_2 \in G$ with $g_1 \geq g_2$ and $\overline{g_1}=\overline{g_2}$,
we have $\Phi^A(M_1, [\![g_1]\!]) \subseteq \Phi^A(M_2,[\![g_2]\!])$.
\end{lemma}
\begin{proof}
Trivial when $[\![g_1]\!]=[\![g_2]\!]$.
Assume that $[\![g_1]\!]\neq [\![g_2]\!]$. 
Take a nonzero element $x\in\Phi^A(M_1, [\![g_1]\!])$. When $\val(x)>g_1$, there is nothing to prove.
When $\val(x)=[\![g_1]\!]$, it follows that $g_2<\val(x)$ by applying Lemma \ref{lem:convex_basic12}(2)
to $g_2<g_1, \overline{g_2}=\overline{g_1}$ and  $[\![g_1]\!]\neq [\![g_2]\!]$.
Thus we have $x\in\Phi^A(M_2, [\![g_2]\!])$.
%When $\val(x)
%We can easily prove it using Lemma \ref{lem:convex_basic12}(2) in the other case.
\end{proof}

In this section we often use the following lemma:

\begin{lemma}\label{lem:star}
Let $(K,\val)$, $A$ and $H$ be the same as in Proposition \ref {prop:ring1}.
Let $\mathcal M$ be a quasi-quadratic module in $A$.
Consider $x, y\in A$ with $\mylcp(x)=\mylcp(y)$ and $\overline{\val(x)}=\overline{\val(y)}$.
If $x\in \mathcal M$ and at least one of the following conditions is satisfied, 
\begin{enumerate}
\item[(1)] $\val(x)\leq \val(y)$;
\item[(2)] $[\![\val(x)]\!]=[\![\val(y)]\!]$,
\end{enumerate}
then $y\in \mathcal M$.  
\end{lemma}
\begin{proof}
There exists an element $u\in K$ with $y=xu^2$ by Definition \ref{def:pseudo}(5).
It immediately follows that $\val(u)\geq e$ and $u\in B$ when condition (1) is satisfied.
In particular, the element $u$ belongs to $A$.
When condition (2) is satisfied, we get $u \in A^\times$ by Corollary \ref{cor:cor_convex}.
Hence we have $y\in \mathcal M$.
\end{proof}

We demonstrate that $M_{g}^A(\mathcal M)$ is a quasi-quadratic module.
\begin{proposition}\label{prop:ring2}
Let $(K,\val)$ be a valued field whose strict units always admit a square root and $B$ be its valuation ring.
Let $A$ be a subring with $B \subseteq A \subseteq K$.
Set $H=\val(A^\times)$.
Consider a quasi-quadratic module $\mathcal M$ in $A$.
Then the following assertions hold true:
\begin{enumerate}
\item[(1)] The sets $M_{g}^A(\mathcal M)$ are quasi-quadratic modules in $F$ for all $g \in G$.
\item[(2)] If $[\![g_1]\!]=[\![g_2]\!]$ for $g_1, g_2\in H \cup G_{\geq e}$, we have $M_{g_1}^A(\mathcal M)= M_{g_2}^A(\mathcal M)$.
\item[(3)] If $\overline{g_1}= \overline{g_2}$ and $g_1\leq g_2$ for $g_1, g_2\in H \cup G_{\geq e}$,
 then $M_{g_1}^A(\mathcal M)\subseteq M_{g_2}^A(\mathcal M)$.
 \end{enumerate}
\end{proposition}
\begin{proof}
(1) We have $M_{g}=\{0\}$ by Lemma \ref{lem:convex_basic2}(1) when $g \not\in H \cup G_{\geq e}$.
The assertion is trivial in this case.
We consider the other case.
 
Take elements $c_1,c_2 \in M_g$.
We first demonstrate that $c_1+c_2 \in M_g$.
We have nothing to show when $c_1+c_2=0$, $c_1=0$ or $c_2=0$.
We assume that $c_1+c_2 \not=0$, $c_1\neq 0$ and $c_2\neq 0$.
Take nonzero elements $x_i \in \mathcal M$ with $\val(x_i)=g$ and $c_i = \mylcp(x_i)$ for $i=1,2$.
If $x_1+x_2=0$, we get $c_1+c_2=0$. It contradicts that $c_1+c_2\neq 0$. By Definition \ref{def:pseudo}(4),
it is immediately follows  that 
$\val(x_1+x_2)=g$ and $\mylcp(x_1)+\mylcp(x_2)=\mylcp(x_1+x_2)$. Since
$x_1+x_2\in \mathcal M$, we get $c_1+c_2 \in M_g$.

We next prove that $c^2a \in M_g$ when $c \in F$ and $a \in M_g$.
We have nothing to prove when $c=0$ or $a=0$.
When $c \not=0$ and $a \not=0$, there exists an element $x \in B$ with $\pi(x)=c$.
In particular, we have $\val(x)=e$. We can 
take a nonzero element $y \in \mathcal M$ with $\val(y)=g$ and $\mylcp(y)=a$
by the definition of $M_g$.
It is easy to show that $\val(x^2y)=g$.
We get $\mylcp(x^2y)=c^2a$ by Definition \ref{def:pseudo}(2).
It means that $c^2a \in M_g$ because $x^2y \in \mathcal M$.
We have demonstrated that $M_g$ is a quasi-quadratic module in $F$.

(2) By symmetry, we have only to demonstrate that $M_{g_1}$ is contained in $M_{g_2}$.
Take a nonzero $c \in M_{g_1}$.
We can choose $x \in \mathcal M$ with $\val(x)=g_1$ and $\mylcp(x)=c$.
We also take $h \in H$ such that $g_2g_1^{-1}=h^2$.
We can take $w \in K$ with $\val(w)=h$.
We have $w \in A^\times$ by Corollary \ref{cor:cor_convex}.
We get $w^2x \in \mathcal M$, $\val(w^2x)=h^2g_1=g_2$ and $\mylcp(w^2x)=k^2\mylcp(x)=k^2c$ for some $k \in F^\times$ by Definition \ref{def:pseudo}(6).
It implies that $k^2c \in M_{g_2}$.
Since $M_{g_2}$ is a quasi-quadratic module, we get $c \in M_{g_2}$.

(3) Let $c$ be a nonzero element of $M_{g_1}$.
There exists a nonzero element 
$x\in\mathcal M$ such that $\mylcp(x)=c$ and $\val(x)=g_1$.
We can take an element $h\in G_{\geq e}$ with $g_2=g_1h^2$,
because $\overline{g_1}=\overline{g_2}$ and $g_1\leq g_2$.
There exists an element $w\in K$ such that $\val(w)=h$.
The element $w$ belongs to $B$.
We have $\mylcp(xw^2)=k^2\mylcp(x)=k^2c$ for some $k \in F^\times$ by Definition \ref{def:pseudo}(6).
We also get $xw^2\in \mathcal M$ and 
$\val(xw^2)=g_2$. 
It follows that $k^2c\in M_{g_2}$, and consequently, $c \in M_{g_2}$.
\end{proof}

\begin{corollary}
\label{cor:d-circle}
Let $(K,\val)$, $B$, $F$, $A$ and $\mathcal M$ be the same as in Proposition \ref{prop:ring2}.
Assume further that the residue class field is of characteristic $\neq 2$.
Let $g\in G$. If there exists a nonzero element $c\in F$ with $\pm c\in M_g^A(\mathcal M)$,
then $M_g^A(\mathcal M)=F$.
\end{corollary}
\begin{proof}
By Proposition \ref{prop:ring2}(1), we see that $M_g$ is a quasi-quadratic module of $F$.
The assertion is immediate from Corollary \ref{lem:field0}.
\end{proof}

We also give the following corollary for later use.
\begin{corollary}\label{cor:mymymy}
Let $(K,\val)$, $B$, $F$, $A$ and $\mathcal M$ be the same as in Proposition \ref{prop:ring2}.
Let $\mathcal N$ be a quasi-quadratic module in $A$.
The notations $M_g$ and $N_g$ denote the quasi-quadratic modules $M_g^A(\mathcal M)$ and $M_g^A(\mathcal N)$, respectively, for all $g \in H \cup G_{\geq e}$.
Set $$L=\{l \in H \cup G_{\geq e}\;|\; l \geq g \text{ for some } g \in H \cup G_{\geq e} \text{ with } M_g+N_g=F\}.$$
Let $m \in H \cup G_{\geq e}$.
If there exists an $l \in L$ with $[l]=[m]$, then we have $m \in L$.
\end{corollary}
\begin{proof}
We have nothing to show when $m \geq l$.
Assume that $m<l$.
Set $h=ml^{-1}$. 
We have $h <e$ and $h \in H$.
There exists $g \in H \cup G_{\geq e}$ with $l \geq g$ and $M_g + N_g =F$ because $l \in L$.
Note that the element $gh^{2}$ is also contained in $H \cup G_{\geq e}$ because $H$ is convex by Lemma \ref{lem:convex_basic2}(1).
In fact, it is obvious when $g \in H$.
In the other case, we have $g \in G_{\geq e} \setminus H$ and we get $h^{-2} < g$ from the convexity of $H$.
By Proposition \ref{prop:ring2}(2), we have $M_{gh^{2}}=M_g$ and $N_{gh^{2}}=N_g$; and consequently we obtain $M_{gh^{2}}+N_{gh^{2}}=F$.
On the other hand, the inequalities $m=lh >lh^{2} \geq gh^{2}$ hold true.
They imply that $m \in L$.
\end{proof}

The following structure theorem for valuation rings guarantees 
that their quasi-quadratic modules have a simple form.

\begin{theorem}
[Canonical representation theorem for quasi-quadratic modules]\label{thm:ring1}
Let $(K,\val)$ be a $2$-henselian valued field whose residue class field is of characteristic $\not=2$ and $B$ be its valuation ring.
Let $A$ be a subring with $B \subseteq A \subseteq K$.
Set $H=\val(A^\times)$.
Consider a quasi-quadratic module $\mathcal M$ in $A$.
We have 
$$\mathcal M = \bigcup_{g\in H\cup G_{\geq e}}\Phi^A(M_{g},[\![g]\!]),$$
where $M_g$ denotes the quasi-quadratic module $M_{g}^A(\mathcal M)$ of $F$ generated by $\mathcal M$ and $g$.
\end{theorem}
\begin{proof}
Set $\mathcal N=\bigcup_{g\in H \cup G_{\geq e}}\Phi^A(M_{g},[\![g]\!])$.
We demonstrate that $\mathcal M=\mathcal N$.
We first show that $\mathcal M \subseteq \mathcal N$.
Take a nonzero element $x\in \mathcal M$ and set $g=\val(x)$.
Since $g\in H\cup G_{\geq e}$ by Lemma \ref{lem:convex_basic2}(1),
it is easily seen that $x \in\Phi(M_{g},[\![g]\!])$.

We demonstrate the opposite inclusion.
Take a nonzero  element $y \in \mathcal N$.
Then there exists an element 
$g \in H\cup G_{\geq e}$ such that $y\in\Phi(M_{g},[\![g]\!])$.

It follows from the definition that $\mylcp(y)\in M_{g}$, $\val(y)= \overline{g}$
and  $\val(y)=[\![g]\!]$ or $\val(y)> g$.
There exists a nonzero element $z\in\mathcal M$ such that
$\mylcp(y)=\mylcp(z)$ and $\val(z)=g$. 
Applying Lemma \ref{lem:star} to $y$ and $z$, we get $y \in \mathcal M$.
\end{proof}

\begin{definition}
Consider a quasi-quadratic module $\mathcal M$ in $A$.
The decomposition 
$\mathcal M = \bigcup_{g\in H\cup G_{\geq e}} \Phi^{A}(M_{g},[\![g]\!])$
in Theorem \ref{thm:ring1} is called the 
\textit{canonical representation} of $\mathcal M$.
\end{definition}

The next proposition and corollaries give useful characterizations of the 
support of a quasi-quadratic module.

\begin{proposition}\label{prop:ring3'}
Let $(K,\val)$ be a $2$-henselian valued field and $B$ be its valuation ring.
Let $A$ be a subring such that $B \subseteq A \subseteq K$ and assume that $2$ is a unit in $B$.
Set $H=\val(A^\times)$.
Consider a quasi-quadratic module $\mathcal M$ in $A$
and a nonzero element $x\in A$ with $\val(x)=g$.
Then the following conditions are equivalent:
\begin{enumerate}
\item[(1)] $x\in\supp(\mathcal M)$;
\item[(2)] $M_{g}^A(\mathcal M)=F$.
\end{enumerate}
Furthermore, $M_h^A(\mathcal M)=F$ whenever $M_g^A(\mathcal M)=F$ 
and $h\geq g$.
\end{proposition}
\begin{proof}
(1) $\Rightarrow$ (2):
By the assumption, we have $\pm x\in\mathcal M$ and $\val(\pm x)=g$.
Thus we get $\pm\mylcp(x)\in M_g$.
It is trivial that the characteristics of both $K$ and $F$ are not equal to two 
when $2$ is a unit in $B$. Hence we see that $M_g=F$ by Corollary \ref{cor:d-circle}.

(2) $\Rightarrow$ (1):
Since $\pm\mylcp(x)=\mylcp(\pm x)\in M_g$,
there exist  nonzero elements $x_1, x_2\in\mathcal M$ such that
$\mylcp(x)=\mylcp(x_1)$, $\mylcp(-x)=\mylcp(x_2)$ and $\val(x_1)=\val(x_2)=g$. 
By using Lemma \ref{lem:star}, we have $\pm x \in \mathcal M$.
Hence we get $x\in\supp(\mathcal M)$.

We next demonstrate the `furthermore' part.
Since the valuation $\val$ is surjective and $hg^{-1}\geq e$,
there exists an element $w\in B$ such that $\val(w)=hg^{-1}$.
Set $y=wx$. Since $x\in\supp(\mathcal M)$, it follows from 
Lemma \ref{lem:basic4} that 
$y\in\supp(\mathcal M)$. Thus we have $M_h=F$ because $\val(y)=h$. 
\end{proof}

\begin{corollary}\label{cor:ring3}
Let $(K,\val)$, $B$, $F$, $A$ and $\mathcal M$ be the same as in
Proposition \ref{prop:ring3'}. 
Set $H=\val(A^\times)$.
The following conditions are equivalent:
\begin{enumerate}
\item[(1)] $\mathcal M=A$;
\item[(2)] $M_{g}^A(\mathcal M)=F$ for all $g \in H \cup G_{\geq e}$;
\item[(3)] $M_{e}^A(\mathcal M)=F$. 
\end{enumerate}
\end{corollary}
\begin{proof}
The implication (1) $\Rightarrow$ (2) immediately follows from Proposition \ref{prop:ring3'}
and Corollary \ref{cor:cor_convex}.
Condition (3) obviously follows from condition (2).
We demonstrate that (3) implies (1). By Proposition \ref{prop:ring3'},
we have $1\in\supp(\mathcal M)$. It is immediate from Corollary \ref{cor:basic4'}(2) 
that $\mathcal M=A$.
\end{proof}

\begin{corollary}\label{cor:ring4}
Let $(K,\val)$, $B$, $F$, $A$ and $\mathcal M$ be the same as in
Proposition \ref{prop:ring3'}.
\begin{enumerate}
\item[(1)] $\supp(\mathcal M)=\{x\in A\setminus\{0\}\;|\; M_{\val(x)}^A(\mathcal M)
=F\}\cup\{0\}$.
\item[(2)] When $\mathcal M\neq \{0\}$,  $\mathcal M$ becomes an ideal of $A$
if and only if
$M_{\val(x)}^A(\mathcal M)=F$ for any $0\neq x\in \mathcal M$.
\end{enumerate}
\end{corollary}
\begin{proof}
(1) It is obvious from the proposition.

(2) Let $\mathcal M$ be an ideal of $A$. Take a nonzero element $x\in\mathcal M$.
Since $-x\in \mathcal M$, we have $\pm\mylcp(x)\in M_{\val(x)}$.
It follows that $M_{\val(x)}=F$ by
% Lemma \ref{lem:basic03} and 
Corollary \ref{cor:d-circle}.

We demonstrate the opposite implication. Take a nonzero element $x\in \mathcal M$.
We have to prove that $Ax\subseteq M$.
By Corollary \ref{cor:basic4'}(1), it is enough to show 
that $-x\in\mathcal M$.
Since we have $M_{\val(x)}=F$ by the assumption, we see that
$x\in\supp(\mathcal M)$ by Proposition \ref{prop:ring3'}.
Thus we get $-x\in\mathcal M$.
\end{proof}

\begin{corollary}\label{cor:ring41}
Let $(K,\val)$, $B$, $F$, $A$ and $\mathcal M$ be the same as in
Proposition \ref{prop:ring3'}.
Let $g_1, g_2 \in G$.
Set $H=\val(A^\times)$.
If $[g_1]=[g_2]$ and $M_{g_1}^A(\mathcal M)=F$, we have $M_{g_2}^A(\mathcal M)=F$.
\end{corollary}
\begin{proof}
We can take $h \in H$ such that $g_2 =g_1h$.
We get a nonzero element $x \in \supp(\mathcal M)$ with $g_1=\val(x)$ by Proposition \ref{prop:ring3'}.
There exists $u \in A^\times$ with $\val(u)=h$. Hence 
we have $ux \in \supp(\mathcal M)$ by Lemma \ref{lem:basic4}.
Since $\val(ux)=g_2$, we obtain $M_{g_2}^A(\mathcal M)=F$ by Proposition \ref{prop:ring3'}.
\end{proof}

We can show that an ideal of the valuation ring is represented 
as a union of sets of the form $\Phi^{A}(F,[\![g]\!])$ for $g \in G$.

\begin{lemma}\label{lem:basic5}
Let $(K,\val)$ be a $2$-henselian valued field and $B$ be its valuation ring.
Let $A$ be a subring such that $B \subseteq A \subseteq K$ and assume that $2$ is a unit in $B$.
Set $H=\val(A^\times)$.
For any ideal $I$ of $A$,
we have
$$
I=\bigcup_{[\![g]\!]\in S(I)}\Phi^A(F,[\![g]\!])\cup\{0\},
$$
where $S(I)=\{[\![\val(x)]\!]\in G/H^2\;|\; x\in I\setminus\{0\}\}$.
%where $S(I)=\val(I\setminus\{0\})/H^2$.
\end{lemma}
\begin{proof}
When $I=\{0\}$, the equality holds because $S(I)=\emptyset$.
We assume that $I\neq \{0\}$. Take a nonzero element $x\in I$.
Since $[\![\val(x)]\!]\in S(I)$, it can be easily seen that $x$ is included in the 
right-hand side of the equality.

We next demonstrate the opposite inclusion.
For that purpose, we show that any $g \in G$ with $[\![g]\!] \in S(I)$ is contained in $\val(I\setminus\{0\})$.
By the definition, we can take $0 \neq y \in I$ and $h \in H$ such that $g=\val(y) \cdot h^2$.
Take $u \in A^\times$ with $\val(u)=h$.
We have $u^2y \in I$ and $\val(u^2y)=\val(y) \cdot h^2=g$.
It implies that $g \in \val(I\setminus\{0\})$.

Let $x$ be a nonzero element of $\bigcup_{[\![g]\!]\in S(I)}\Phi(F,[\![g]\!])\cup\{0\}$.
There exists an element $[\![g]\!]\in S(I)$ such that $x \in \Phi(F,[\![g]\!])$.
Let $g \in G$ be a representative of $[\![g]\!]$.
We get $g \in \val(I\setminus\{0\})$.
We have $\val(x)=\overline{g}$.
In addition, at least one of the conditions that $\val(x)=[\![g]\!]$ and $\val(x)>g$ is satisfied.
By Proposition \ref{prop:ring2}(2)(3) and Corollary \ref{cor:ring4}(2), we have $M_{\val(x)}^A(I)=F$.
It implies that there exists a nonzero element $z\in I$ such that
$\mylcp(x)=\mylcp(z)$ and $\val(x)=\val(z)$. By Definition \ref{def:pseudo}(5), we can take $u\in K$ with
$x=u^2z$. We clearly see that $u\in B^\times \subseteq A^\times$ and $x\in I$.
\end{proof}

We are going to study presentations of the sum and the intersection of two quasi-quadratic modules
in a valuation ring.

\begin{lemma}\label{lem:ring3}
Let $(K,\val)$ be a $2$-henselian valued field and $B$ be its valuation ring.
Let $A$ be a subring such that $B \subseteq A \subseteq K$ and assume that $2$ is a unit in $B$.
Set $H=\val(A^\times)$.
Take elements $g_1,g_2 \in H\cup G_{\geq e}$ and quasi-quadratic modules $M_1$ and $M_2$ in the residue class field $F$.
Set $g_{\rm max}=\max\{g_1,g_2\}$.
\begin{enumerate}[(1)]
\item For any nonzero elements $x_1\in\Phi^A(M_1,[\![g_1]\!])$ and $x_2\in \Phi^A(M_2,[\![g_2]\!])$ with $\val(x_1)\neq \val(x_2)$, 
we have $x_1+x_2\in\Phi^A(M_1,[\![g_1]\!])\cup\Phi^A(M_2,[\![g_2]\!])$.
\item If $M_1$ and $M_2$ are proper, the following equality holds true:
\end{enumerate}
\begin{align*}
&\Phi^A(M_1,[\![g_1]\!]) + \Phi^A(M_2,[\![g_2]\!]) \\
\!
&= 
\!
\left\{
\!\!\!
\begin{array}{ll}
\Phi^A(M_1,[\![g_1]\!]) \cup \Phi^A(M_2,[\![g_2]\!]) 
& 
\text{ 
\!\!\!\!\!\!\!\!\!\!\!\!\!\!\!\!\!\!\!\!\!\!\!\!\!\!\!\!\!\!
\!\!\!\!\!\!\!\!\!\!\!\!\!\!\!\!\!\!\!\!\!\!\!\!\!\!\!\!\!\!\!\!\!\!\!\!\!\!\!\!\!\!\!\!\!\!\!\!\!\!\!
 if \;} \overline{g_1} \not= \overline{g_2},\\
\Phi^A(M_{1},[\![g_{1}]\!])\cup\Phi^A(M_2,[\![g_{2}]\!])
\cup\Phi^A(M_1+M_2,[\![g_{\rm max}]\!])  & \\
& \text{ \!\!\!\!\!\!\!\!\!\!\!\!\!\!\!\!\!\!\!\!\!\!\!\!\!\!\!\!\!\!\!\!\!\!\!\!\!\!\!\!
\!\!\!\!\!\!\!\!\!\!\!\!\!\!\!\!\!\!\!\!\!\!\!\!\!\!\!\!\!\!\!\!\!\!\!\!\!\!\!\!\!
 if  \;}
\overline{g_1} = \overline{g_2}\text{ and }M_1+M_2\not =F,\\
\Phi^A(M_1,[\![g_1]\!]) \cup \Phi^A(M_2,[\![g_2]\!])
\cup\bigcup_{g\in H{g_{\rm max}}\cup G_{\geq g_{\rm max}}}\Phi^A(F,[\![g]\!]) & \\
&
\text{ \!\!\!\!\!\!\!\!\!\!\!\!\!\!\!\!\!\!\!\!\!
\!\!\!\!\!\!\!\!\! otherwise}.
%\text{if } \overline{g_1} = \overline{g_2} \text{ and }M_1+M_2=F.
\end{array}
\right.
\end{align*}
\end{lemma}
\begin{proof}
Set $\mathcal M_i = \Phi^A(M_i,[\![g_i]\!])$ for $i=1,2$
and assume that $g_1\leq g_2$ without loss of generality.
Note that Corollary \ref{lem:field0} and Corollary \ref{lem:field000} 
are valid for $F$.
% by Lemma \ref{lem:basic03}.

(1) We may assume $x_1+x_2\neq 0$ and $\val(x_1)<\val(x_2)$ by symmetry.
By Definition \ref{def:pseudo}(4), it follows that
$\mylcp(x_1+x_2)=\mylcp(x_1)\in M_1$. We have
$\val(x_1+x_2)=\val(x_1)=\overline{g_1}$.
We get $\val(x_1+x_2)=\val(x_1)=[\![g_1]\!]$
or
$\val(x_1+x_2)=\val(x_1)>g_1$.
Hence we have
$x_1+x_2\in \mathcal M_1\cup \mathcal M_2$.

(2) We first consider the case in which 
$\overline{g_1} \not = \overline{g_2}$.
It is obvious that 
$\mathcal M_1+\mathcal M_2\supseteq \mathcal M_1\cup \mathcal M_2$.
We only have to show that
$\mathcal M_1+\mathcal M_2\subseteq \mathcal M_1\cup \mathcal M_2$.
Take arbitrary elements $x_i\in\mathcal M_i$ for $i=1,2$. There is nothing to show 
when $x_1=0$, $x_2=0$ or $x_1+x_2=0$. Assume $x_1\neq 0$, $x_2\neq0$ and
$x_1+x_2\neq0$. Since $\overline{g_1}\neq\overline{g_2}$, we have
$\val(x_1)\neq \val(x_2)$. It is immediate from (1) that
$x_1+x_2\in \mathcal M_1\cup \mathcal M_2$.

We next consider the case in which $\overline{g_1} = \overline{g_2}$
and $M_1+M_2\neq F$. 
We first demonstrate that the left hand side of the
equality is included in the right-hand side.
Take nonzero elements $x_1\in \mathcal M_1$ and $x_2\in \mathcal M_2$ with $x_1+x_2\neq 0$.
%If $x_1=0$, we have $x_2\in \Phi(M_1+M_2,[\![g_2]\!])$.
%If $x_2=0$, we have $x_1\in \Phi(M_1,[\![g_1]\!])$.
%Assume $x_1\neq 0$ and $x_2\neq 0$.
When $\val(x_1)\neq \val(x_2)$, we have $x_1+x_2\in \mathcal M_1\cup\mathcal M_2$ from (1).
We consider the case in which $\val(x_1)=\val(x_2)$ and 
$\mylcp(x_1)+\mylcp(x_2)\neq 0$.
By Definition \ref{def:pseudo}(4), we have
$\mylcp(x_1+x_2)=\mylcp(x_1)+\mylcp(x_2)\in M_1+M_2$ and
$\val(x_1+x_2)=\val(x_1)= \overline{g_2}$. We get 
$\val(x_1+x_2)=\val(x_2)=[\![g_2]\!]$ or $\val(x_1+x_2)=\val(x_2)>g_2$.
Hence we have $x_1+x_2\in \Phi(M_1+M_2,[\![g_2]\!])$.
We consider the case in which $\val(x_1)=\val(x_2)$ and 
$\mylcp(x_1)+\mylcp(x_2)= 0$. However this case does not occur.
In fact, suppose that the case holds. Since $\pm\mylcp(x_1)\in M_1+M_2$,
we have $M_1+M_2=F$ by Corollary \ref{lem:field0}.
It is a contradiction to the assumption.

We show the opposite inclusion.
Take a nonzero element $x\in\Phi(M_1,[\![g_1]\!])\cup\Phi(M_2,[\![g_2]\!])\cup\Phi(M_1+M_2,[\![g_2]\!])$.
We have nothing to show when $x\in\Phi(M_i,[\![g_i]\!])$ for $i=1,2$.
Assume $x\in\Phi(M_1+M_2,[\![g_2]\!])$.
There exist elements 
$c_1\in M_1$ and $c_2\in M_2$ such that $\mylcp(x)=c_1+c_2 \neq 0$.

If $c_1=0$, we have $x\in\mathcal M_2$
because $\val(x)=\overline{g_2}$ and $\val(x)=[\![g_2]\!]$ or $\val(x)> g_2$. 

If $c_2=0$, it is obvious that $x\in\mathcal M_1$ when $g_1=g_2$ or $\val(x)=[\![g_1]\!]$.
We assume $g_2>g_1$ and $\val(x)\neq[\![g_1]\!]$. It follows that $x\in\mathcal M_1$ when $\val(x) \geq g_2$ because $\val(x)>g_1$ in this case.
Let $\val(x)< g_2$. 
%It is sufficient to show the case in which $\val(x)<g_2$.
We have $\val(x)=[\![g_2]\!]$ by the assumption.
If $g_1>\val(x)$, we have $g_1>g_2$ by Lemma \ref{lem:convex_basic12}(1).
This is a contradiction. We get $\val(x)>g_1$. Hence we have $x\in\mathcal M_1$.

Thus we assume $c_1\neq 0$ and $c_2\neq 0$.
We can take nonzero elements $x_i\in K$ with $\mylcp(x_i)=c_i$ and $\val(x_i)=g_2$ for $i=1,2$
by Definition \ref{def:pseudo}(3).
It follows from Corollary \ref{cor:cor_convex} that $x_i\in A$ for $i=1,2$.
We have
$x_i\in\Phi(M_i,[\![g_i]\!])$ for $i=1,2$ by the assumption.
If $x_1+x_2=0$, it follows that $c_1+c_2=0$. It contradicts
the fact that $c_1+c_2\neq 0$. Hence we see that
$\mylcp(x)=c_1+c_2=\mylcp(x_1)+\mylcp(x_2)=\mylcp(x_1+x_2)$
and  $\val(x_1+x_2)=\val(x_1)=g_2$
by Definition \ref{def:pseudo}(4).
Since $\val(x)=\overline{g_2}$ and $\val(x)=[\![g_2]\!]$ or $\val(x)> g_2$,
we have $x\in\mathcal M_1+\mathcal M_2$ by using Proposition \ref{prop:ring1} and Lemma \ref{lem:star}.

We consider the remaining case in which $\overline{g_1} = \overline{g_2}$
and  $M_1+M_2=F$. 
Take $w \in A$ satisfying the condition in Lemma \ref{coro:basic004}.
We have $\val(w)=g_2$ and $\mylcp(w)=1$.
We demonstrate that $\mathcal M_1+\mathcal M_2= \mathcal M_1\cup \mathcal M_2\cup Aw$.
We first show that $\mathcal M_1+\mathcal M_2$ is contained in the right-hand of
the equality.
Take arbitrary elements $x_1+x_2\in\mathcal M_1
+\mathcal M_2$, where $x_i\in\mathcal M_i$ for $i=1,2$.
We may assume that $x_1\neq 0$, $x_2\neq 0$ and $x_1+x_2\neq 0$.
When $\val(x_1)\neq \val(x_2)$, we have $x_1+x_2\in \mathcal M_1\cup \mathcal M_2$ from (1).
When $\val(x_1)=\val(x_2)$ and $\val(x_2)=[\![g_2]\!]$ or
$\val(x_1)=\val(x_2)$ and $\val(x_2)>g_2$, it follows that $x_1+x_2\in Aw$ by Lemma \ref{coro:basic004}. 
Hence we have
$\mathcal M_1+\mathcal M_2\subseteq \mathcal M_1\cup \mathcal M_2\cup Aw$.

We show the opposite inclusion. 
By Corollary \ref{lem:field000},
there exists a nonzero element $c\in M_1$ such that $-c\in M_2$.
We can take a unit $u\in B$ with $\pi(u)=c$.
Set $x=uw$. It is easily seen from Lemma \ref{lem:triv} that
$x\in\Phi(M_1,[\![g_2]\!])\subseteq\mathcal M_1$. It is immediate that  $-x\in\mathcal M_2$.
By Corollary \ref{cor:basic4'}(1), we get
$\mathcal M_1+\mathcal M_2 \supseteq Aw$.
It is clear that $\mathcal M_1+\mathcal M_2\supseteq \mathcal M_i$ 
for $i=1,2$. Hence we have
$\mathcal M_1+\mathcal M_2\supseteq \mathcal M_1\cup \mathcal M_2\cup Aw$.

Consequently, we have
$\mathcal M_1+\mathcal M_2=\mathcal M_1\cup \mathcal M_2\cup Aw.$
By Lemma \ref{lem:basic5}, we see that 
$
Aw=\bigcup_{[\![g]\!]\in S(Aw)}\Phi(F,[\![g]\!]).
$
Hence we have proven
$
\mathcal M_1+\mathcal M_2=\mathcal M_1 \cup \mathcal M_2
\cup\bigcup_{g\in Hg_2\cup G_{\geq g_2}}\Phi(F,[\![g]\!]).
$
\end{proof}

Note that the assumption that $M_1$ and $M_2$ are proper is not necessary
in the proof of Lemma \ref{lem:ring3}(2)
when we show that 
the left hand side of the equality is included 
the right hand side of the equality.

Combining the above lemma with Remark \ref{rem:properness2}, we have 
the following corollary:
\begin{corollary}\label{cor:ring1}
Let $(K,\val)$ be a $2$-henselian valued field and $B$ be its valuation ring.
Let $A$ be a subring such that $B \subseteq A \subseteq K$ and assume that $2$ is a unit in $B$.
Set $H=\val(A^\times)$.
Let $(M_{g})_{g \in H\cup G_{\geq e}}$ 
be a family of  quasi-quadratic modules in the residue class field $F$ such that
$M_g\subseteq M_h$ whenever one of the following two conditions is satisfied:
\begin{itemize}
%\item $h<g$ and $[g]=[h]$;
\item $g\leq h$ and $\overline{g}= \overline{h}$;
%\item $g\leq h$ and $M_g=F$.
\item $M_g=F$ and ( $[g]=[h]$ or $g \leq h$).
\end{itemize}
Then, the union $\bigcup_{g \in H\cup G_{\geq e}}\Phi^A(M_{g},[\![g]\!])$
is a quasi-quadratic module in $A$.
\end{corollary}
\begin{proof}
Set  $\mathcal M=\bigcup_{g \in H\cup G_{\geq e}}\Phi^{A}(M_{g},[\![g]\!])$.
We first show that $\mathcal M$ is closed under multiplication by the squares of elements in $A$.
However, it immediately follows from Remark \ref{rem:properness2}.

We next want to show that $\mathcal M$ is closed under addition.
Take elements $x_1, x_2 \in \mathcal M$.
When $x_1+x_2=0$, $x_1=0$ or $x_2=0$, it is obvious that $x_1+x_2 \in \mathcal M$.
We demonstrate that $x_1+x_2 \in \mathcal M$ when $x_1+x_2 \not=0, x_1\neq 0$ and $x_2\neq 0$.
There exist elements $g_1, g_2\in H\cup G_{\geq e}$ such that $x_1\in\Phi(M_{g_1},[\![g_1]\!])$
and $x_2\in\Phi(M_{g_2},[\![g_2]\!])$.

When $\overline{g_1}\neq \overline{g_2}$, we have $x_1+x_2\in\Phi(M_{g_1},[\![g_1]\!])\cup\Phi(M_{g_2},[\![g_2]\!])$ 
by Lemma \ref{lem:ring3}(1).

We next consider the case in which $\overline{g_1}=\overline{g_2}$.
We may assume $g_1\leq g_2$ without loss of generality.
If $M_{g_1}+M_{g_2}\neq F$, we see 
that $x_1+x_2\in\Phi(M_{g_1},[\![g_1]\!]) \cup \Phi(M_{g_2},[\![g_2]\!])$ by Lemma \ref{lem:ring3}(2) 
because $M_{g_1}+M_{g_2}=M_{g_2}$ by the assumption.

If $M_{g_1}+M_{g_2}=F$, we see that $M_{g_2}=F$. This implies that $M_g=F$ for any $g\in H{g_2}\cup G_{\geq g_2}$
because $Hg_2\cup G_{\geq g_2}\subseteq H\cup G_{g\geq e}$ by Lemma \ref{lem:convex_basic2}(2).
Hence it follows from Lemma \ref{lem:ring3}(2) and Lemma \ref{lem:convex_basic2}(2) that 
\begin{eqnarray*}
x_1+x_2&\in&
\Phi(M_{g_1},[\![g_1]\!]) \cup \Phi(M_{g_2},[\![g_2]\!])\cup\bigcup_{g\in H{g_2}\cup G_{\geq g_2}}\Phi(F,[\![g]\!]) \\
&=&\Phi(M_{g_1},[\![g_1]\!]) \cup \Phi(M_{g_2},[\![g_2]\!])\cup\bigcup_{g\in H{g_2}\cup G_{\geq g_2}}\Phi(M_g,[\![g]\!])\\
&\subseteq& \mathcal M.
\end{eqnarray*}
\end{proof}

\begin{lemma}\label{lem:ring4}
Let $(K,\val)$ be a $2$-henselian valued field and $B$ be its valuation ring.
Let $A$ be a subring such that $B \subseteq A \subseteq K$.
% and $2$ is a unit in $B$.
Set $H=\val(A^\times)$.
Take elements $g_1,g_2 \in H\cup G_{\geq e}$ and quasi-quadratic modules $M_1$ and $M_2$ in the residue class field $F$.
The following equality holds true:
$$
\Phi^A(M_1,[\![g_1]\!]) \cap \Phi^A(M_2,[\![g_2]\!]) = \left\{
\begin{array}{ll}
\{0\} & \text{if } \overline{g_1} \not= \overline{g_2},\\
\Phi^A(M_1 \cap M_2,[\![\max\{g_1,g_2\}]\!]) & \text{otherwise.}
\end{array}
\right.
$$
\end{lemma}
\begin{proof}
Set $\mathcal M=\Phi^A(M_1,[\![g_1]\!]) \cap \Phi^A(M_2,[\![g_2]\!])$.
We first consider the case in which $\overline{g_1} \not= \overline{g_2}$.
If there exists a nonzero element $x \in \mathcal M$, 
we have $\overline{g_1} = \overline{\val(x)}= \overline{g_2}$.
A contradiction.

We next consider the case in which $\overline{g_1}=\overline{g_2}$.
We may assume $g_1\leq g_2$ without loss of generality.
Take a nonzero element $x\in\mathcal M$.
Since it follows that $\val(x)=\overline{g_2}$,
$\mylcp(x)\in M_1\cap M_2$ and $\val(x)=[\![g_2]\!]$ or $\val(x)>g_2$, 
we have $x\in\Phi(M_1\cap M_2,[\![g_2]\!])$.
We demonstrate the opposite inclusion.
Take a nonzero element $x\in\Phi(M_1\cap M_2,[\![g_2]\!])$.
It follows from the definition and the assumption that 
$\val(x)=\overline{g_1}=\overline{g_2}$ and $\mylcp(x)\in M_1\cap M_2$.
When $\val(x)>g_2$, it is obvious that $x\in\mathcal M$.
Assume that $\val(x)=[\![g_2]\!]$. We have nothing to prove when $\val(x)=[\![g_1]\!]$.
We consider the remaining case in which $\val(x)\neq [\![g_1]\!]$.
If $g_1>\val(x)$, we get $g_1>g_2$ by Lemma \ref{lem:convex_basic12}(1).
A contradiction. Hence we have $x \in\mathcal M$.
\end{proof}

We are now ready to demonstrate the following theorem
which is a main result of this section.

\begin{theorem}\label{thm:ring2}
Let $(K,\val)$ be a $2$-henselian valued field and $B$ be its valuation ring
with residue class field $F$.
Let $A$ be a subring such that $B \subseteq A \subseteq K$ and assume that $2$ is a unit in $B$.
Set $H=\val(A^\times)$.
Consider quasi-quadratic modules $\mathcal M$ and $\mathcal N$ in $A$.
Let $\mathcal M = 
\bigcup_{g \in H\cup G_{\geq e}}\Phi^{A}(M_{g},[\![g]\!])$
and 
$\mathcal N = 
\bigcup_{g \in H\cup G_{\geq e}}\Phi^{A}(N_{g},[\![g]\!])$
be the canonical representations of $\mathcal M$ and $\mathcal N$
where $M_{g}=M_{g}(\mathcal M)$ and $N_{g}=M_{g}
(\mathcal N)$, respectively.
We get the following equalities:

(1) 
$$\mathcal M + \mathcal N = \bigcup_{g \in H\cup G_{\geq e}}
\Phi^{A}(M_{g}+N_{g},[\![g]\!])
\cup\bigcup_{l \in L} \Phi^{A}(F,[\![l]\!]),
$$
where $L=\{l\in H\cup G_{\geq e}\;|\;
l\geq g \text{ for some } 
g\in H\cup G_{\geq e} \text{ with } M_g+N_g=F \}$.

(2) 
\begin{eqnarray*}
\mathcal M \cap \mathcal N 
&=& 
\bigcup_{g\in H\cup G_{\geq e}}
\Phi^{A}(M_g\cap N_g,[\![g]\!]).
\end{eqnarray*}
\end{theorem}
\begin{proof}

(1) Let $\mathcal L_1$ be the right-hand side of the equality for the
simplicity of notation.
We first demonstrate 
$\mathcal M+\mathcal N \subseteq \mathcal L_1$.
Let $x$ be a nonzero element of $\mathcal{M+N}$.
We can take $x_1\in\mathcal M$ and $x_2\in\mathcal{N}$ with
$x=x_1+x_2$. There exist $g_1,g_2\in H\cup G_{\geq e}$ such that 
$x_1\in\Phi(M_{g_1},[\![g_1]\!])$ and $x_2\in\Phi(N_{g_2},[\![g_2]\!])$.
We may assume that $x_1\neq 0$, $x_2\neq 0$ and $g_1\leq g_2$.

We consider the case in which $\val(x_1)\neq \val(x_2)$.
By Lemma \ref{lem:ring3}(1), we have
$x=x_1+x_2\in\Phi(M_{g_1},[\![g_1]\!])\cup\Phi(N_{g_2},[\![g_2]\!])
\subseteq\bigcup_{g\in H\cup G_{\geq e}}\Phi(M_{g}+N_{g},[\![g]\!])\subseteq\mathcal L_{1}$.

We next consider the case in which $\val(x_1)=\val(x_2)$.
If $M_{g_1}+N_{g_2}\neq F$, then
we see that 
$$x\in\Phi(M_{g_1},[\![g_1]\!])\cup\Phi(N_{g_2},[\![g_2]\!])\cup\Phi(M_{g_1}+N_{g_2},[\![g_2]\!])$$
by Lemma \ref{lem:ring3}(2). It is trivial that $x\in\mathcal L_1$ when
$x\in\Phi(M_{g_1},[\![g_1]\!])\cup\Phi(N_{g_2},[\![g_2]\!])$. So we assume that $x\in\Phi(M_{g_1}+N_{g_2},[\![g_2]\!])$.
By Proposition \ref{prop:ring2}(3), it follows that 
$\Phi(M_{g_1}+N_{g_2},[\![g_2]\!])\subseteq\Phi(M_{g_2}+N_{g_2},[\![g_2]\!])\subseteq \mathcal L_1$.

If $M_{g_1}+N_{g_2}=F$, we have
$$
x\in\Phi(M_{g_1},[\![g_1]\!])\cup\Phi(N_{g_2},[\![g_2]\!])\cup
\bigcup_{g\in H{g_{2}}\cup G_{\geq g_{2}}}\Phi(F,[\![g]\!])
$$
by Lemma \ref{lem:ring3}(2).
Note that  $M_{g_2}+N_{g_2}=F$ because 
$M_{g_1}\subseteq M_{g_2}$ by the assumption and Proposition \ref{prop:ring2}(3).
It is sufficient to demonstrate the result in the case where $x\in\Phi(F,[\![g]\!])$ for some $g\in Hg_2\cup G_{\geq g_2}$.
When $g\in G_{\geq g_2}$, we get $g\in L$ by Lemma \ref{lem:convex_basic2}(2).
Hence it follows that $x\in \bigcup_{l \in L} \Phi(F,[\![l]\!])\subseteq
\mathcal L_1$.
When $g\not\in G_{\geq g_2}$, we have $g\in Hg_2$. 
It follows from  Lemma \ref{lem:convex_basic2}(2) and Corollary \ref{cor:mymymy} 
that $g \in L$. Thus we have  $x\in \bigcup_{l \in L} \Phi(F,[\![l]\!])\subseteq
\mathcal L_1$.

We next demonstrate the opposite inclusion.
Let $x$ be a nonzero element of $\mathcal L_1$.
We first assume that $x\in\Phi(F,[\![l]\!])$ for some $l\in L$.
There exists an element $g\in H\cup G_{\geq e}$ such that
$l \geq g$ with  $M_g+N_g=F$. 
We have $M_{g}(\mathcal M+\mathcal N)=F$ by Definition \ref{def:pseudo}(4).
It is trivial that
$M_{l}(\mathcal M+\mathcal N)=F$ by Proposition \ref{prop:ring3'}.

We have $\val(x)=\overline{l}$ and either $\val(x)=[\![l]\!]$ or $\val(x)>l$ by the definition of $\Phi(F,[\![l]\!])$.
Note that $\val(x)\in H\cup G_{\geq e}$ from Lemma \ref{lem:convex_basic2}(1).
By Proposition \ref{prop:ring3'}, it is enough to demonstrate that
$$M_{\val(x)}(\mathcal M+\mathcal N)=F.$$
However the equality immediately follows from Proposition \ref{prop:ring2}(2)(3).

We consider the remaining case
where $x\not\in\Phi(F, [\![l]\!])$ for any $l\in L$.
We can take an element $g\in H\cup G_{\geq e}$ with
$x\in\Phi(M_g+N_g,[\![g]\!])$. 
If $M_g+N_g=F$, we see that $g\in L$. It contradicts the assumption 
that $x\not\in\Phi(F, [\![l]\!])$ for any $l\in L$.
Hence it follows that $M_g+N_g\neq F$.
By Lemma \ref{lem:ring3}(2) and Theorem \ref{thm:ring1},
we see that
$$x\in\Phi(M_g+N_g,[\![g]\!])\subseteq
\Phi(M_g,[\![g]\!])+\Phi(N_g,[\![g]\!])\subseteq\mathcal{M+N}.$$

(2) We have $\mathcal M = \bigcup_{g\in H\cup G_{\geq e}}\Phi(M_{g},[\![g]\!])$ and $\mathcal N = \bigcup_{g\in H\cup G_{\geq e}}\Phi(N_{g},[\![g]\!])$ by Theorem \ref{thm:ring1}.
We use these equalities in the proof.
Let $\mathcal L_2$ be the right-hand side of the equality.
We first demonstrate 
$\mathcal M\cap \mathcal N \subseteq \mathcal L_2$.
Take a nonzero element $x\in\mathcal M\cap \mathcal N$.
Then there exist $g_1,g_2\in H\cup G_{\geq e}$ such that
$x\in \Phi(M_{g_1},[\![g_1]\!])\cap \Phi(N_{g_2},[\![g_2]\!])$.
We may assume that $g_1\leq g_2$.
By Lemma \ref{lem:ring4}, we have $\overline{g_1}=\overline{g_2}$ and 
$\Phi(M_{g_1},[\![g_1]\!])\cap \Phi(N_{g_2},[\![g_2]\!])=\Phi(M_{g_1}\cap N_{g_2},[\![g_2]\!])$.
Moreover, we can show that 
$\Phi(M_{g_1}\cap N_{g_2},[\![g_2]\!])\subseteq\Phi(M_{g_2}\cap N_{g_2},[\![g_2]\!])$,
because $M_{g_1}\subseteq M_{g_2}$ by Proposition \ref{prop:ring2}(3). 
Hence we have $x\in\Phi(M_{g_2}\cap N_{g_2},[\![g_2]\!])$.

We next demonstrate the opposite inclusion.
However it is immediately follows from 
$\Phi(M_g \cap N_g,[\![g]\!])=\Phi(M_{g},[\![g]\!])\cap\Phi(N_{g},[\![g]\!])$ for any $g\in H\cup G_{\geq e}$
by Lemma \ref{lem:ring4}
\end{proof}

Recall that $\mathfrak X_{R}$
denotes the set of all quasi-quadratic modules in a commutative ring $R$.
Let $(K,\val)$ be a valued field and $B$ be its valuation ring.
Suppose that $A$ is a subring such that $B \subseteq A \subseteq K$.
Set $H=\val(A^\times)$.

We set
\begin{align*}
\mathcal T_R^{H\cup G_{\geq e}}
&= \{(M_g)_{g\in H\cup G_{\geq e}}\in\prod_{g\in H\cup G_{\geq e}}\mathfrak X_{R}\;|\;
M_g\subseteq M_h \text{ whenever }\\
&\qquad\qquad
([\![g]\!]=[\![h]\!])\vee
 (g\leq h \wedge \overline{g}=\overline{h})
\vee (M_g=R \wedge ([g]=[h]\vee g\leq h))\}.
\end{align*}

The following theorem is the main theorem introduced in Section \ref{sec:intro}.
\begin{theorem}\label{thm:ring4}
Let $(K,\val)$ be a $2$-henselian valued field and $B$ be its valuation ring.
Let $A$ be a subring such that $B \subseteq A \subseteq K$ and assume that $2$ is a unit in $B$.
Set $H=\val(A^\times)$.
We define the map 
$\Theta:\mathfrak X_A \rightarrow \mathcal T_F^{H\cup G_{\geq e}}$ by
\begin{equation*}
\Theta(\mathcal M)=(M_{g}(\mathcal M))_{g \in H\cup G_{\geq e}}\text{,}
\end{equation*}
where $F$ is the residue class field.
The map $\Theta$ is a bijection.
\end{theorem}
\begin{proof}
The map $\Theta$ is well-defined by 
Proposition \ref{prop:ring2}, 
Proposition \ref{prop:ring3'}
and Corollary \ref{cor:ring41}.
We define the map 
$\Lambda:\mathcal T_F^{H\cup G_{\geq e}} \rightarrow \mathfrak X_{A}$ by
\begin{equation*}
\Lambda\left((M_{g})_{g\in H\cup G_{\geq e}}\right)
=\bigcup_{g\in H\cup G_{\geq e}}\Phi^{A}(M_{g},[\![g]\!])\text{.}
\end{equation*}
The map $\Lambda$ is well-defined by Corollary \ref{cor:ring1}.
The composition $\Lambda \circ \Theta$ is the identity map by Theorem \ref{thm:ring1}.
We demonstrate that $\Theta \circ \Lambda$ is also the identity map.
Fix $g\in  H\cup G_{\geq e}$.
Let $N$ be the $g$-th coordinate of $\Theta \circ \Lambda\left((M_g)_{g\in 
H\cup G_{\geq e}}\right)$.
We want to show that $N=M_{g}$.
By the assumption and Lemma \ref{lem:triv}, we have 
\begin{align*}
N &= M_{g}\left(\bigcup_{h\in H\cup G_{\geq e}}\Phi(M_{h}, [\![h]\!])\right)\\
&=\left\{\mylcp(x)\;\Bigl|\; x \in \bigcup_{h\in H\cup G_{\geq e}}
\Phi(M_{h},[\![h]\!])
\setminus\{0\}
\text{ and } \val(x)=g\right\} \cup \{0\}\\
&=\left\{\mylcp(x)\;|\; x \in \Phi(M_{g},[\![g]\!])\setminus\{0\}\text{ and } 
\val(x)=g\right\}\cup \{0\}\\
&\subseteq M_g.
\end{align*}

We demonstrate the opposite inclusion.
Take a nonzero element $c \in M_{g}$.
There exists a nonzero element $w \in K$ with $\val(w)=g$ and $\mylcp(w)=c$ by 
Definition \ref{def:pseudo}(3).
By Corollary \ref{cor:cor_convex}, we get $w\in A$. 
It follows that  
$w\in\Phi(M_g,[\![g]\!])$. Thus we obtain $c \in N$.
We have finished to prove that $\Theta$ and $\Lambda$ are the inverses of the others.
\end{proof}

We give explicit expressions of $\Theta(\mathcal M \cap \mathcal N)$ and $\Theta(\mathcal M+\mathcal N)$ for 
all $\mathcal M, \mathcal N \in \mathfrak X_A$.

\begin{theorem}\label{cor:exp}
Let $(K,\val)$ be a $2$-henselian valued field and $B$ be its valuation ring.
Let $A$ be a subring such that $B \subseteq A \subseteq K$ and assume that $2$ is a unit in $B$.
Set $H=\val(A^\times)$.
Consider quasi-quadratic modules $\mathcal M$ and $\mathcal N$ in $A$.
Fix an element $g\in H\cup G_{\geq e}$.
Let $M_{g}=M_{g}(\mathcal M)$ and $N_{g}=M_{g}(\mathcal N)$ be the quasi-quadratic modules of the
residue class field $F$ generated by $\mathcal M$, $\mathcal N$ and $g$, respectively.
Then we get the following equalities:

(1)
$$
\Theta(\mathcal M \cap \mathcal N) 
=  (M_g\cap N_g)_{g\in H\cup G_{\geq e}}
.$$

(2)
$$
\Theta(\mathcal M + \mathcal N)_g
=\left\{
\begin{array}{ll}
F & \text{if there exists } h \in H \cup G_{\geq e} \\
&\quad \text{ with } g \geq h \text{ and } M_h+N_h=F,\\
M_g+N_g & \text{otherwise.}
\end{array}
\right.
$$
Here, the notation $\Theta(\mathcal M + \mathcal N)_g$ denotes the $g$-th coordinate of $\Theta(\mathcal M + \mathcal N)$.
\end{theorem}
\begin{proof}
(1) It it is immediate from the definition that $M_g(\mathcal M\cap\mathcal N)\subseteq
M_g\cap N_g$. We demonstrate the opposite inclusion. Take a nonzero element 
$c\in M_g\cap N_g$.
There exist some elements $x_1\in\mathcal M$ and $x_2\in\mathcal N$ with
$c=\mylcp(x_1)=\mylcp(x_2)$ and $g=\val(x_1)=\val(x_2)$.
By Definition \ref{def:pseudo}(5), we can take an element $b\in B^{\times}$ with
$x_1=b^2x_2$. Hence we have $c\in M_g(\mathcal M\cap\mathcal N)$.

(2) Set $L=\{l \in H \cup G_{\geq e}\;|\; l \geq h \text{ for some } h \in H \cup G_{\geq e} \text{ with } M_h + N_h =F\}$.
When $g \in L$, we have $\Phi(F,[\![g]\!]) \subseteq \mathcal M + \mathcal N$ by Theorem \ref{thm:ring2}(1).
On the other hand, we always have $M_g(\Phi(F,[\![g]\!]))=F$ by the definition of the operator $\Phi$ and Definition \ref{def:pseudo}(3).
It implies that $\Theta(\mathcal M + \mathcal N)_g=F$.

We next consider the case in which $g \not\in L$.
Note that $\val(\Phi(F,[\![l]\!])) \subseteq L$ for all $l \in L$ by Corollary \ref{cor:mymymy}.
It means that $\bigcup_{l \in L} \Phi(F,[\![l]\!])$ does not contain an element $x$ with $\val(x)=g$.
Therefore, we have $\Theta(\mathcal M + \mathcal N)_g=M_g(\bigcup_{h \in H \cup G_{\geq e}} \Phi(M_h+N_h,[\![h]\!]))$ by Theorem \ref{thm:ring2}(1).
It is obvious that $M_g+N_g$ is contained in $\Theta(\mathcal M + \mathcal N)_g$.
We prove the opposite inclusion.
Take a nonzero $c \in \Theta(\mathcal M + \mathcal N)_g$.
There exist $h \in H \cup G_{\geq e}$ and $x \in \Phi(M_h+N_h,[\![h]\!])$ with $\val(x)=g$ and $\mylcp(x)=c$.
We have (i) $[\![g]\!]=[\![h]\!]$ or (ii) $\overline{g}=\overline{h}$ and $h<g$.
In both cases, we have $c \in M_g + N_g$ by Proposition \ref{prop:ring2}(2)(3).
We have demonstrated the assertion (2). 
\end{proof}

We review the definitions of quadratic modules, preorderings and semiorderings given in \cite{AK,M}.
\begin{definition}
Let $R$ be a commutative ring.
A quasi-quadratic module $M$ of $R$ is a \textit{quadratic module} in $R$ if $1 \in M$.
% and $-1 \not\in M$.
A quadratic module $M$ of $R$ is a \textit{preordering} in $R$ if $M \cdot M \subseteq M$.
A quasi-quadratic module $M$ of $R$ is a \textit{quasi-semiordering} in $R$ if $M \cup (-M) =R$ and 
$\supp(M)=M \cap (-M)$ is a prime ideal.
A \textit{semiordering} is a quasi-semiordering which is simultaneously a quadratic module.
\end{definition}

The following theorem gives necessary and sufficient conditions 
for a quasi-quadratic module of the valuation ring
to be a quadratic module/pre-ordering/quasi-semiordering.

\begin{theorem}\label{thm:ring3}
Let $(K,\val)$ be a $2$-henselian valued field and $B$ be its valuation ring
with residue class field $F$.
Let $A$ be a subring such that $B \subseteq A \subseteq K$ and assume that $2$ is a unit in $B$.
Set $H=\val(A^\times)$.
Consider a quasi-quadratic module $\mathcal M$ in $A$.
\begin{enumerate}
\item[(1)] The quasi-quadratic module $\mathcal M$ is a proper quadratic module if and only if $M_{e}(\mathcal M)$ is a proper quadratic module, where $e$ is the 
identity element of $G$.
\item[(2)]The quasi-quadratic module $\mathcal M$ is a quasi-semiordering
if and only if $M_{g}(\mathcal M)$ is a quasi-semiordering for any 
$g\in H\cup G_{\geq e}$ with $M_g(\mathcal M)\neq F$ and, for any $g_1, g_2\in H\cup G_{\geq e}$ 
with $M_{g_1}(\mathcal M)\neq F$ and $M_{g_2}(\mathcal M)\neq F$, we have 
$M_{g_1g_2}(\mathcal M)\neq F$.
\item[(3)] 
Assume that $\mylcp:K^{\times}\rightarrow F^{\times}$ is an angular component map $\mylc$.
When $\mathcal M$ is a quadratic module, $\mathcal M$ is a preordering if and only if, for any $g_1, g_2 \in H\cup G_{\geq e}$, nonzero elements $c_1 
\in M_{g_1}(\mathcal M)$ and $c_2 \in M_{g_2}(\mathcal M)$, 
we have $c_1 c_2 \in M_{g_1 g_2}(\mathcal M)$.
In particular, $M_{e}(\mathcal M)$ is a preordering if $\mathcal M$ is.

\end{enumerate}
\end{theorem}
\begin{proof}
(1) Recall that $M_{e} = \{\mylcp(x)\;|\; x \in \mathcal M\setminus\{0\}
\text{ with }\val(x)=e\} \cup \{0\}$.
We first assume that $\mathcal M$ is a proper quadratic module.
By Corollary \ref{cor:ring3}, we see that $M_e\neq F$.
It follows that $1\in M_e$ because $1_K\in \mathcal M$ by the assumption.
Since $M_e$ is a quasi-quadratic module from Proposition \ref{prop:ring2}(1),
we see that $M_e$ is a proper quadratic module.
%If $-1\in M_e$ then we have $M_e=F$ by Corollary \ref{lem:field0}. This is 
%a contradiction. Hence we have $-1\not\in M_e$.

We next assume that $M_{e}$ is a proper quadratic module.
We have $1 \in  M_e$ by the assumption.
There exists a nonzero element $x \in \mathcal M$ such that 
$\val(x)=e$ and $\mylcp(x)=1$. Since $\val(x)=\val(1_K)$ and $\mylcp(x)=\mylcp(1_K)$,
it follows from Lemma \ref{lem:star} that $1_K\in \mathcal M$.
If $-1_K\in\mathcal M$, then we have $M_e=F$
% by Lemma \ref{lem:basic03} and
by Corollary \ref{lem:field0}. 
This contradicts the 
assumption that $M_e$ is proper.
We have demonstrated the assertion (1).

(2) We first assume that the quasi-quadratic module $\mathcal M$ is 
a quasi-semiordering. 
Take an arbitrary element $g\in H\cup G_{\geq e}$ with $M_g \neq F$.
It is obviously true that the ideal $\mbox{supp}(M_{g})$
becomes the prime ideal of $F$ because $\supp(M_g)\subseteq M_g$.

We next want to show $M_{g}\cup (-M_{g})=F$.
Take an arbitrary nonzero element $c \in F$. 
By Definition \ref{def:pseudo}(3), we can take an element $y\in K$ such
that $\val(y)=g$ and $\mylcp(y)=c$.
It follows from Corollary \ref{cor:cor_convex} that $y\in A$.
Since $y\in A=\mathcal M\cup (-\mathcal M)$,
we have 
$y \in \mathcal M$ or $-y \in \mathcal M$.
In the former case, we have $c \in M_g$.
We also have $-c \in M_{g}$ in the latter case.
Hence we see that $F\subseteq M_{g} \cup (-M_{g})$.
We have proven that $F =M_{g} \cup (-M_{g})$.

We prove the remaining assertion.
Take elements $g_1,g_2\in H\cup G_{\geq e}$ with $M_{g_1}\neq F$ and $M_{g_2}\neq F$.
By Lemma \ref{lem:convex_basic2}(1), we have 
$x_1,x_2\in A$ with $\val(x_1)=g_1$ and $\val(x_2)=g_2$.
By Proposition \ref{prop:ring3'}, it follows 
that $x_i\in A\setminus\supp(\mathcal M)$ for $i=1,2$.
Since the ideal $\supp(\mathcal M)$ is a prime ideal, 
we have $x_1x_2\in A\setminus\supp(\mathcal M)$.
This means that $M_{\val(x_1x_2)}=M_{g_1g_2}\neq F$ again by Proposition \ref{prop:ring3'}.

We demonstrate the opposite implication.
We first want to show that $A=\mathcal M \cup (-\mathcal M)$.
Take an arbitrary nonzero element $x \in A$.
There is nothing to prove when $x\in\supp(\mathcal M)$. We assume
$x\not\in\supp(\mathcal M)$. 
Set $g=\val(x)$ and $c=\mylcp(x)$.
We have $M_g \neq F$ by Proposition \ref{prop:ring3'}.
Since $M_{g}$ is a quasi-semiordering by the assumption, we have 
$M_g \cup (-M_{g})=F$.
If $c\in M_g$, then we have $x\in \Phi(M_g,[\![g]\!])\subseteq\mathcal M$ 
by Theorem \ref{thm:ring1}.
If $-c\in M_g$, then we have $-x\in \Phi(M_g,[\![g]\!])\subseteq\mathcal M$ 
by Theorem \ref{thm:ring1} and hence we see that $x\in -\mathcal M$.
We have proven that $x \in \mathcal M \cup (-\mathcal M)$.

We next want to show the ideal $\supp(\mathcal M)$ is a prime ideal.
Take elements $x_1,x_2 \in A$ such that $x_1\not\in \supp(\mathcal M)$
and $x_2\not\in \supp(\mathcal M)$.
Since $M_{\val(x_1)}\neq F$ and $M_{\val(x_2)}\neq F$, we have $x_1x_2\not\in \supp(\mathcal M)$ by the assumption and Proposition \ref{prop:ring3'}.
Hence we see that the ideal $\supp(\mathcal M)$ is prime.

(3) We assume that $\mathcal M$ is a preordering.
Take arbitrary elements $g_1, g_2 \in H\cup G_{\geq e}$, arbitrary nonzero elements 
$c_1 \in M_{g_1}$ and $c_2 \in M_{g_2}$.
By the assumption, there exist 
nonzero elements $x_i \in \mathcal M$ with $\val(x_i)=g_i$
and $\mylc(x_i)=c_i$ for $i=1,2$.
Set $x=x_1x_2$.
We get $x \in \mathcal M$ because $\mathcal M$ is a preordering.
It is easily seen that $\val(x)=g_1g_2\in H\cup G_{\geq e}$ and $\mylc(x)=c_1c_2$.
We have demonstrated $c_1c_2 \in M_{{g_1}{g_2}}$.

We next show the opposite implication.
Take nonzero elements $x_1,x_2 \in \mathcal M$.
We want to show that $x_1 x_2 \in \mathcal M$.
Set $g_i =\val(x_i)$ and $c_i = \mylc(x_i)$ for $i=1,2$.
Note that $c_i\in M_{g_i}$ for $i=1,2$.
We have $c_1c_2 \in M_{g_1g_2}$ by the assumption.
Therefore, there exists a nonzero element $x \in \mathcal M$ with 
$\val(x)=g_1g_2$ and $\mylc(x)=c_1c_2$.
We also have $\val(x_1x_2)=g_1g_2$ and $\mylc(x_1x_2)=c_1c_2$.
Therefore we get $x_1x_2 \in \mathcal M$ by Lemma \ref{lem:star}.
The `in particular' part is obvious.
We have demonstrated assertion (3).
\end{proof}

We consider the set $\mathfrak Y_R$ of all  quasi-semiorderings in a commutative ring $R$.
Note that the sets
discussed in \cite[Section 5.3]{M} are similar but not identical to $\mathfrak Y_R$.

We set
\begin{eqnarray*}
\mathcal S_R^{H\cup G_{\geq e}}\!=\!
\{(M_g)_{g\in H\cup G_{\geq e}}\!\!\in\mathcal T_R^{H\cup G_{\geq e}}\;|\;
M_{g_1g_2}\neq R \text{ whenever } M_{g_1}\neq R \text{ and } M_{g_2}\neq R,\\
M_g \in\mathfrak Y_R \text{ for any } g\in H\cup G_{\geq e} \text{ with } M_g\neq R 
\}.
\end{eqnarray*}

We have the following corollary:
\begin{corollary}\label{cor:ring2}
Let $(K,\val)$ be a $2$-henselian valued field and $B$ be its valuation ring
with residue class field $F$.
Let $A$ be a subring such that $B \subseteq A \subseteq K$ and assume that $2$ is a unit in $B$.
Set $H=\val(A^\times)$.
%Assume that $\mylcp:K^{\times}\rightarrow F^{\times}$ is an angular component map $\mylc$.
Then 
there exists a bijection between $\mathfrak Y_A$ and 
$\mathcal S_F^{H\cup G_{\geq e}}$.
\end{corollary}
\begin{proof}
Let $\Theta$ and $\Lambda$ be the bijections defined in Theorem \ref{thm:ring4} and its proof.
The maps $\theta:\mathfrak Y_A \rightarrow \mathcal S_F^{H\cup G_{\geq e}}$ and 
$\lambda : \mathcal S_F^{H\cup G_{\geq e}} \rightarrow \mathfrak Y_A$ 
are defined by $\theta(\mathcal M)=\Theta(\mathcal M)$ 
and $\lambda \left((M_{g})_{g\in H\cup G_{\geq e}}\right)
=\Lambda\left((M_{g})_{g \in H\cup G_{\geq e}}\right)$, respectively.
They are well-defined by Theorem \ref{thm:ring3}(2) and Theorem \ref{thm:ring4}.
It is obvious that $\theta$ and $\lambda$ are bijections again by Theorem \ref{thm:ring4}.
\end{proof}

We end this section by proving the following lemma 
which is used in the next section.

\begin{lemma}\label{lem:monogenic_ring}
Let $(K,\val)$ be a $2$-henselian valued field whose residue class field is of characteristic $\not=2$ and $B$ be its valuation ring.
Let $A$ be a subring with $B \subseteq A \subseteq K$.
Assume further that $(K,\val)$ admits an angular component map $\mylc:K^{\times}\to F^{\times}$
and the residue class field $F$ is a formally real field.
(See \cite[Section II.5]{L} for the definition of a formally real field.)
For any nonzero $f \in A$, the quasi-quadratic module in $A$ generated by $f$ is of the form $\Phi^A(M_f,[\![\val(f)]\!])$, where $M_f$ is the quasi-quadratic module of $F$ generated by $\mylc(f)$.
\end{lemma}
\begin{proof}
Let $\mathcal M_f$ be the quasi-quadratic module generated by $f$.
We first show that $\mathcal M_f$ is contained in $\Phi(M_f,[\![\val(f)]\!])$.
Take a nonzero element $x \in  \mathcal M_f$.
There are finite nonzero elements $u_1,\ldots, u_m \in A$ with $x=f(u_1^2+\cdots +u_m^2)$.
Set $g_{\min}=\min\{\val(u_i)\;|\; 1 \leq i \leq m\}$ and $S=\{i\;|\;1 \leq i \leq m, \val(u_i)=g_{\min}\}$.
We have $g_{\min} \in H \cup G_{\geq e}$ by Corollary \ref{cor:cor_convex}.
We easily have $\val(x)=\val(f)g_{\min}^2$ and $\mylc(x)=\sum_{i \in S}\mylc(f)\mylc(u_i)^2$ by Definition \ref{def:pseudo}(4) because $F$ is formally real.
We easily get $x \in \Phi(M_f,[\![\val(f)]\!])$.

We next demonstrate the opposite inclusion.
Take a nonzero element $x \in \Phi(M_f,[\![\val(f)]\!])$.
We can get an element $g \in H \cup G_{\geq e}$ and a finite sequence of nonzero elements $c_1, \ldots, c_m$ in $F$ such that $\val(x)=\val(f)g^2$ and $\mylc(x)=\mylc(f)(c_1^2+\cdots+c_m^2)$.
There exist nonzero elements $u_i \in K$ such that $\val(u_i)=g$ and $\mylc(u_i)=c_i$ for all $1 \leq i \leq m$ by
Definition \ref{def:pseudo}(3).
The elements $u_i$ are in $A$ for all $1\leq i \leq m$ by Corollary \ref{cor:cor_convex}.
Set $y=fu_1^2+ \cdots +fu_m^2$.
The element $y$ belongs to $\mathcal M_f$.
Since $F$ is formally real, we have $\val(y)=\val(f)g^2=\val(x)$ and $\mylc(y)=\mylc(f)(c_1^2+\cdots+c_m^2)=\mylc(x)$ by
Definition \ref{def:pseudo}(4).
Hence we have $x\in \mathcal M_f$ by Lemma \ref{lem:star}.
\end{proof}

\section{Special cases}\label{sec:special}
We demonstrated the structure theorems for the ring $A$ in the previous section, where $A$ is a subring of the valued field $K$ containing the valuation ring $B$.
In this section, we treat two important special cases; that is, $A=B$ and $A=K$. 

\subsection{When the ring $A$ is the valuation ring $B$}\label{sec:ring}
We first consider the case in which $A=B$.
In this case, we have $H=\val(A^{\times})=\{e\}$ and
$$
\Phi(M,g) 
= \{x \in A\setminus\{0\} \;|\; 
\val(x) =\overline{g}\text{, }\val(x) \geq g \text{ and } \mylcp(x) \in M \}
\cup \{0\},
$$
where $M$ is a quasi-quadratic module of $F$.
\begin{theorem}
[Canonical representation theorem for quasi-quadratic modules
in valuation rings]\label{thm:sp_ring1}
Let $(K,\val)$ be a $2$-henselian valued field whose residue class field $F$ is of characteristic $\not=2$ and
$B$ be its valuation ring.
Consider a quasi-quadratic module $\mathcal M$ in $B$.
We have 
$$\mathcal M = \bigcup_{g\in G_{\geq e}}\Phi(M_{g},g),$$
where $M_g=M_{g}(\mathcal M)$.
\end{theorem}
\begin{proof}
Special case of Theorem \ref{thm:ring1}.
\end{proof}

\begin{theorem}\label{thm:sp_ring2}
Let $(K,\val)$ be a $2$-henselian valued field, $B$ be its valuation ring such that $2$ is a unit.
Consider two quasi-quadratic modules $\mathcal M$ and $\mathcal N$ in $B$.
Let $\mathcal M = 
\bigcup_{g \in G_{\geq e}}\Phi(M_{g},g)$
and 
$\mathcal N = 
\bigcup_{g \in G_{\geq e}}\Phi(N_{g},g)$
be the canonical representations of $\mathcal M$ and $\mathcal N$
where $M_{g}=M_{g}(\mathcal M)$ and $N_{g}=M_{g}
(\mathcal N)$, respectively.
We get the following equalities:

(1) 
$$\mathcal M + \mathcal N = \bigcup_{g \in G_{\geq e}}
\Phi(M_{g}+N_{g},g)
\cup\bigcup_{h \in H} \Phi(F,h),
$$
where $H=\{h\in G_{\geq e}\;|\;h\geq g \text{ for some } 
g\in G_{\geq e} \text{ with } M_g+N_g=F\}$.

(2) 
\begin{eqnarray*}
\mathcal M \cap \mathcal N 
&=& 
\bigcup_{g\in G_{\geq e}}
\Phi(M_g\cap N_g,g).
\end{eqnarray*}
\end{theorem}
\begin{proof}
Special case of Theorem \ref{thm:ring2}.
\end{proof}

We have
$$
\mathcal T_F^{G_{\geq e}}=\{(M_g)_{g\in G_{\geq e}}\in\prod_{g\in G_{\geq e}}\mathfrak X_{F}\;|\;
M_g\subseteq M_h \text{ whenever } (g\leq h) \wedge ((\overline{g}=\overline{h}) \vee (M_g=F))
\}
$$
in this case.

\begin{theorem}\label{thm:sp_ring4}
Let $(K,\val)$ be a $2$-henselian valued field and
$B$ be its valuation ring such that $2$ is a unit.
We define the map 
$\Theta:\mathfrak X_B \rightarrow \mathcal T_F^{G_{\geq e}}$ by
\begin{equation*}
\Theta(\mathcal M)=(M_{g}(\mathcal M))_{g \in G_{\geq e}}\text{,}
\end{equation*}
where $F$ is the residue class field.
The map $\Theta$ is a bijection.
\end{theorem}
\begin{proof}
Special case of Theorem \ref{thm:ring4}.
\end{proof}

\begin{theorem}\label{thm:sp_ring3}
Let $(K,\val)$ be a $2$-henselian valued field and
$B$ be its valuation ring such that $2$ is a unit.
Consider a proper quasi-quadratic module $\mathcal M$ in $B$.
\begin{enumerate}
\item[(1)] The quasi-quadratic module $\mathcal M$ is a quadratic module if and only if $M_{e}(\mathcal M)$ is a quadratic module, where $e$ is the 
identity element of $G$.
\item[(2)] 
The quasi-quadratic module $\mathcal M$ is a quasi-semiordering
if and only if $M_{g}(\mathcal M)$ is a quasi-semiordering for any 
$g\in G_{\geq e}$ with $M_g(\mathcal M)\neq F$ and, for any $g_1, g_2\in G_{\geq e}$ 
with $M_{g_1}(\mathcal M)\neq F$ and $M_{g_2}(\mathcal M)\neq F$, we have 
$M_{g_1g_2}(\mathcal M)\neq F$.
\item[(3)] 
Assume that $\mylcp:K^{\times}\rightarrow F^{\times}$ is an angular component map $\mylc$.
When $\mathcal M$ is a quadratic module, $\mathcal M$ is a preordering if and only if, for any $g_1, g_2 \in G_{\geq e}$, nonzero elements $c_1 
\in M_{g_1}(\mathcal M)$ and $c_2 \in M_{g_2}(\mathcal M)$, 
we have $c_1 \cdot c_2 \in M_{g_1\cdot g_2}(\mathcal M)$.
In particular, $M_{e}(\mathcal M)$ is a preordering if so is $\mathcal M$.
\end{enumerate}
\end{theorem}
\begin{proof}
Special case of Theorem \ref{thm:ring3}.
\end{proof}

\subsection{Derivation of Augustin and Knebusch's assertions}\label{sec:AK}
The initial motivation of this study is the generalization of the study \cite{AK}.
We want to clarify that our results are the generalization of their study.  

Let $R$ be a commutative ring. Following the definition in \cite{AK}, for given elements $g_1,\ldots,g_r\in R$, 
let $\po_R(g_1,\ldots,g_r)$ be the preordering of $R$ generated by the elements $g_1,\ldots,g_r\in R$. Namely, 
\begin{eqnarray*}
\po_R(g_1,\ldots,g_r)
:=\Bigl{\{}\sum_{\epsilon}\sigma_{\epsilon}g_{\epsilon}
\;\Bigl{|}\;\epsilon=(\epsilon_1,\ldots,\epsilon_r)\in\{0,1\}^{r},
\sigma_{\epsilon}\in \sum R^2,
g_{\epsilon}=\prod_{i=1}^r g_i^{{\epsilon}_i}
\Bigl{\}},
\end{eqnarray*}
where $\sum R^2$ denotes the set of sums of squares of elements of $R$.
In particular, the set $\po_R(g)$ for any $g\in R$ is called  {\it monogenic quadratic module}
because $\po_R(g)$ coincides with the quadratic module of $R$ generated by $1$ and $g$.
It is simply denoted by $\po(g)$ when $R$ is clear from the context.

A \textit{euclidean field} $F$ is a formally real field with $F=F^2 \cup (-F^2)$. 
For any nonzero $c \in F$, its sign is defined as $1$ when $c \in F^2$ and $-1$ when $-c \in F^2$.
In this subsection, we assume that
$B=F[\![X]\!]$, which is the valuation ring of the Laurent series field
$K=F(\!(X)\!)$ in the indeterminate $X$ over a euclidean field $F$. 
The field $K$ is a valued field whose valuation $\val(f)$ of a nonzero element $f \in K$
is the order of the element $f$.
The value group $G$ is the group of integers $\mathbb Z$.
It is well known that a strict unit has a square root in the ring $B$.
The valued field $(K,\val)$ has a cross section by Proposition \ref{prop:wmap}.
Therefore, it admits the angular component map $\mylc: K^{\times}\to F^{\times}$
by Proposition \ref{prop:angular}.
The notation $\epsilon(f)$ denotes the sign of $\mylc(f)$ in $F$ for any nonzero $f \in B$
according to the study \cite{AK}.

We have the following lemma:
\begin{lemma}\label{lem:aug_mono}
Let $f$ be a nonzero element in $B$.
We have 
\[
\po(f)=\left\{\begin{array}{ll}
\Phi(F^2,0) \cup \Phi(\epsilon(f)F^2,\val(f)) & \!\!\!\!\!\!\!\!\!\!\!\!\!\!\!\!\!\!\!\!\text{ if } \val(f) 
\text{ is odd,}\\
\Phi(F^2,0) &\!\!\!\!\!\!\!\!\!\!\!\!\!\!\!\!\!\!\!\! \text{ if } \val(f) \text{ is even and } \epsilon(f)=1 \text{,}\\
\Phi(F^2,0) \cup \Phi(F,\val(f)) \cup \Phi(F,\val(f)+1) & \text{ otherwise.}
\end{array}
\right.
\]
In particular, we have 
\[
\po(f)=\left\{\begin{array}{ll}
\po(1) & \text{ if } \val(f) \text{ is even and } \epsilon(f)=1 \text{,}\\
\po(\epsilon(f)X^{\val(f)}) & \text{ otherwise.}
\end{array}
\right.
\]
\end{lemma}
\begin{proof}
Note that a sum of squares in $B$ is a square because a strict unit of $B$ has a square root.
%By the definition, $\po(f)$ is the sum of the quasi-quadratic modules generated by $1$ and $f$, respectively.
Since $F$ is a euclidean field,  the quasi-quadratic modules in $F$ are $\{0\}$, $F^2$, $-F^2$ and $F$.
The quasi-quadratic modules generated by $1$ and by $f$ are $\Phi(F^2,0)$ and $\Phi(\epsilon(f)F^2,\val(f))$, respectively, by Lemma \ref{lem:monogenic_ring}.
Hence, we have $$\po(f)=\Phi(F^2,0)+\Phi(\epsilon(f)F^2,\val(f)).$$

When $\val(f)$ is odd, we have $\val(f) \not\equiv 0 \mod 2\mathbb{Z}$.
We get $\po(f)=\Phi(F^2,0) \cup \Phi(\epsilon(f)F^2,\val(f))$ by Lemma \ref{lem:ring3}(2).
When $\val(f)$ is even and $\epsilon(f)=1$, we have $\po(f)=\Phi(F^2,0)$
by Lemma \ref{lem:ring3}(2) and what we noted at the beginning of the proof.
When $\val(f)$ is even and $\epsilon(f)=-1$, we have $$\po(f)=\Phi(F^2,0) \cup \bigcup_{g \geq \val(f)} \Phi(F,g)$$
by Lemma \ref{lem:ring3}(2).
On the other hand, we have $$\bigcup_{g \geq \val(f)} \Phi(F,g)=\Phi(F,\val(f)) \cup \Phi(F,\val(f)+1)$$ by Lemma \ref{lem:triv}.
The `in particular' part of the lemma easily follows from the equality we have just demonstrated. 
\end{proof}

We can completely classify quadratic modules in $F[\![X]\!]$.
\begin{lemma}\label{lem:aug_quad}
For a quadratic module $Q$ in $B=F[\![X]\!]$
there are nonnegative integers $m< n$ and $\varepsilon\in\{-1,1\}$ such that
the quasi-quadratic modules $M_k(Q)$ of $F$ generated by $Q$ and $k \in \mathbb Z _{\geq 0}$ are given by the following table:
\begin{center}
\begin{tabular}{|c|l|}
\hline
 & \;\;\;\;\;\;\;\;\;\;\;\;\;\;\;\;\;\;\;\;Description of $M_k(Q)$\\
 \hline\hline & \vspace{-3mm}\\
 (a) & $M_k(Q)=\left\{\begin{array}{ll} F^2 & \text{ if } k \text{ is even,}\\ \{0\} & \text{ otherwise.}\end{array}\right.$\\
 & \vspace{-3mm}\\
 \hline & \vspace{-3mm}\\
 (b) & $M_k(Q)=F \text{ for all } k\in\mathbb{Z}_{\geq 0}$\\
 & \vspace{-3mm}\\
 \hline & \vspace{-3mm}\\
 (c) & $M_k(Q)=\left\{\begin{array}{ll} F^2 & \!\!\!\!\text{ if } k \text{ is even,}\\ 
 \varepsilon F^2 & \!\!\!\!\text{ if } k \text{ is odd and } k \geq n,\\
 \{0\} & \!\!\!\!\text{ otherwise.}\end{array}\right.$\\ 
 & \vspace{-3mm}\\
 \hline & \vspace{-3mm}\\
 (d) & $M_k(Q)=\left\{\begin{array}{ll} F^2 & \text{ if } k \text{ is even and }k<n,\\ 
 F & \text{ if } k \geq n,\\
 \{0\} & \text{ otherwise.}\end{array}\right.$\\
 & \vspace{-3mm}\\
 \hline & \vspace{-3mm}\\
 (e) & $M_k(Q)=\left\{\begin{array}{ll} F^2 & \!\!\!\!\text{ if } k \text{ is even and }k<n,\\ 
 \varepsilon F^2 & \!\!\!\!\text{ if } k \text{ is odd and } m \leq k<n,\\ 
 F & \!\!\!\!\text{ if } k \geq n,\\
 \{0\} & \!\!\!\!\text{ otherwise.}\end{array}\right.$ \smallskip \\ 
 \hline 
\end{tabular}
\end{center}
The quadratic module $Q$ has one of the following forms in each case:
\begin{enumerate}
\item[(a)] $Q=\Phi(F^2,0)$;
\item[(b)] $Q=B$;
\item[(c)] $Q=\Phi(F^2,0) \cup \Phi(\varepsilon F^2,n)$;% for some positive odd number $n$; 
\item[(d)] $Q=\Phi(F^2,0) \cup \Phi(F,n) \cup \Phi(F,n+1)$;% for some positive number $n$;
\item[(e)] $Q=\Phi(F^2,0) \cup \Phi(\varepsilon F^2,m) \cup \Phi(F,n) \cup \Phi(F,n+1)$.% for some positive odd number $m$ and positive number $n$ with $m<n$.
\end{enumerate}
%All the double signs correspond in all the cases.
In particular, any quadratic module in $F[\![X]\!]$ is a preordering.
\end{lemma}
\begin{proof}
As was stated in Definition \ref{Phi-M_g},
write $M_k$ instead of $M_k(Q)$ for simplicity.
We have $M_m \subseteq M_n$ whenever $m \equiv n \mod 2\mathbb Z$ and $m \leq n$ by Proposition \ref{prop:ring2}(3).
We also have $M_n=F$ whenever $M_m=F$ for some $m<n$ by Proposition \ref{prop:ring3'}.
Note that any quadratic module in $B$ contains $\po(1)$. In particular, we have $F^2\subseteq \po(1) \subseteq Q$.
It follows that $F^2 \subseteq M_n$ for all nonnegative even numbers $n$. In fact, it suffices to demonstrate that
$F^2\subseteq M_0$ because $M_0 \subseteq M_n$ for nonnegative even numbers $n$ by Proposition \ref{prop:ring2}(3). 
Since $M_0=\{\mylc(x)\in F\setminus\{0\}\;|\; x \in Q
\setminus\{0\} \text{ with } \val(x)= 0 \} \cup \{0\}$, we get $F^2\subseteq M_0$.
Since $F$ is a euclidean field, the quasi-quadratic modules in $F$ are $\{0\}$, $F^2$, $-F^2$ and $F$.
The quasi-quadratic modules $F^2$ and $-F^2$ are not contained in each other.
Therefore, $M_k(Q)$ should be one of the five forms in the table.

We have $Q=\Phi(F^2,0)$ in the cases of (a) by 
Theorem \ref{thm:sp_ring1} and Lemma \ref{lem:triv}.
In the case of (b),
we have $Q=B$ by Corollary \ref{cor:ring3}.

We next consider the case (c).
We have $M_n=F^2$ or $M_n=-F^2$.
We only consider the first case.
We can get the similar result in the latter case.
We have $M_k=F^2$ for all positive odd number with $k \geq n$.
Hence, we get $Q=\Phi(F^2,0) \cup \Phi(F^2,n)$ by 
Theorem \ref{thm:sp_ring1} and Lemma \ref{lem:triv}.

The next target is the case (d).
We have $$Q=\bigcup_{k<n,k \text{:even}}\Phi(F^2,k) \cup \bigcup_{k \geq n}\Phi(F,k)=\Phi(F^2,0) \cup \Phi(F,n) \cup \Phi(F,n+1)$$
by Theorem \ref{thm:sp_ring1} and Lemma \ref{lem:triv}.

The remaining case is the case (e).
We only consider the case in which $M_m=F^2$.
We can prove the lemma similarly when $M_m=-F^2$.
We have $$Q=\Phi(F^2,0) \cup \Phi(F^2,m) \cup \Phi(F,n) \cup \Phi(F,n+1)$$ by Theorem \ref{thm:sp_ring1} and Lemma \ref{lem:triv} in the same manner as the case (d).

We finally demonstrate that $Q$ is a preordering.
It is easy to check that $Q$ satisfies the equivalent condition given in Theorem \ref{thm:sp_ring3}(3) for each case by the table in the lemma.
\end{proof}

\begin{lemma}\label{lem:aug_quad2}
Quadratic modules $Q$ in $B=F[\![X]\!]$ are represented as follows in all the cases of Lemma \ref{lem:aug_quad}:
\begin{center}
\begin{tabular}{|c|l|c|}
\hline
 & Classification in Lemma \ref{lem:aug_quad} & Description of $Q$\\
 \hline
  \hline
(i) & (a) & $\po(1)$\\
\hline
(ii) & (b) & $\po(-1)$\\
\hline
(iii) & (c) & $\po(\varepsilon X^n)$\\
\hline
(iv) & (d) and $n$ is even & $\po(-X^n)$\\
\hline
(v) & (d) and $n$ is odd & $\po(X^n)+\po(-X^n)$\\
\hline
(vi) & (e) and $n$ is even & $\po(\varepsilon X^m)+\po(-X^n)$\\
\hline
(vii) & (e) and $n$ is odd & $\po(\varepsilon X^m)+\po(-\varepsilon X^n)$\\
\hline
\end{tabular}
\end{center}
%The double signs correspond in the case (c) and (e) of Lemma \ref{lem:aug_quad}.
In particular, any quadratic module is generated by at most two monogenic submodules.
\end{lemma}
\begin{proof}
It is clear that $Q=\Phi(F^2,0)=B^2=\po(1)$ in the case (i).
We obtain $Q=B=\po(-1)$ by Corollary \ref{cor:basic4'}(2) in the case (ii).
%In the remaining case, we can get the lemma using Lemma \ref{lem:ring3}, Lemma \ref{lem:triv}, 
%Lemma \ref{lem:aug_quad} and Lemma \ref{lem:aug_mono}.
For example, we demonstrate the case (v).
We get 
\begin{eqnarray*}
&&\po(X^n)+\po(-X^n)\\
  &=&\Phi(F^2,0) \cup \Phi(F^2,n)+\Phi(F^2,0) \cup \Phi(-F^2,n) \;\;\;\;\; \mbox{(by Lemma \ref{lem:aug_mono})}\\
  &=& (\Phi(F^2,0)+\Phi(F^2,0)) \cup (\Phi(F^2,0)+ \Phi(F^2,n))\\ 
  &&\;\;\;\;\;\;\;\;\;\;\;\cup (\Phi(F^2,0)+ \Phi(-F^2,n))\cup(\Phi(F^2,n)+\Phi(-F^2,n))\\
  &=& \Phi(F^2,0) \cup (\Phi(F^2,0) \cup \Phi(F^2,n)) \cup (\Phi(F^2,0)\\
  &&\;\;\;\;\;\;\;\;\;\;\;\cup \Phi(-F^2,n))\cup \bigcup_{k \geq n}\Phi(F,k) \;\;\;\; \mbox{(by Lemma \ref{lem:ring3}(2))}\\
%  &=& \Phi(F^2,0) \cup \bigcup_{m \geq n}\Phi(F,m)\\ %// Lemma 3.4
  &=& \Phi(F^2,0) \cup \Phi(F,n) \cup \Phi(F,n+1) \;\;\;\; \mbox{(by Lemma \ref{lem:triv})} \\
  &=&Q  \;\;\;\;\;\;\mbox{(by Lemma \ref{lem:aug_quad})}.
\end{eqnarray*}
The remaining cases follow in the same way.
\end{proof}

For any quadratic module $Q$ of $B=F[\![X]\!]$, we consider the following three quadratic submodules.
Let $n_l$ be the smallest positive odd integer with $X^{n_l} \in Q$.
We set $Q_l=\po(X^{n_l})$.
If such an integer $n_l$ does not exist, we set $Q_l=\po(1)$.
We define $Q_c$ and $Q_r$ similarly:
We put $Q_c=\po(-X^{n_c})$ if the smallest nonnegative even integer with $-X^{n_c} \in Q$ exists and set $Q_c=\po(1)$ otherwise.
We put $Q_r=\po(-X^{n_r})$ if the smallest positive odd integer with $-X^{n_r} \in Q$ exists and set $Q_r=\po(1)$ otherwise.

\begin{lemma}\label{lem:aug_lcr}
Let $Q$ be a quadratic module in $B=F[\![X]\!]$.
The quadratic submodules $Q_l$, $Q_c$ and $Q_r$ are given by the following:
\begin{center}
\begin{tabular}{|l|c|c|c|}
\hline
Classification in Lemma \ref{lem:aug_quad2} & $Q_l$ & $Q_c$ & $Q_r$\\
 \hline
  \hline
 (i) & $\po(1)$ & $\po(1)$ & $\po(1)$\\
 \hline
 (ii) & $\po(X)$ & $\po(-1)$ & $\po(-X)$\\
  \hline
 (iii) and $M_n=F^2$ & $\po(X^n)$ & $\po(1)$ & $\po(1)$\\
 \hline
  (iii) and $M_n=-F^2$ & $\po(1)$ & $\po(1)$ & $\po(-X^n)$\\
 \hline
  (iv)  & $\po(X^{n+1})$ & $\po(-X^n)$ & $\po(-X^{n+1})$\\
 \hline
   (v) & $\po(X^n)$ & $\po(-X^{n+1})$ & $\po(-X^n)$\\
 \hline
   (vi) and $M_m=F^2$ & $\po(X^m)$ & $\po(-X^n)$ & $\po(-X^{n+1})$\\
 \hline
   (vi) and $M_m=-F^2$ & $\po(X^{n+1})$ & $\po(-X^n)$ & $\po(-X^m)$\\
 \hline
   (vii) and $M_m=F^2$ & $\po(X^m)$ & $\po(-X^{n+1})$ & $\po(-X^n)$\\
 \hline
   (vii) and $M_m=-F^2$ & $\po(X^n)$ & $\po(-X^{n+1})$ & $\po(-X^m)$\\
 \hline
 \end{tabular}
\end{center}
In particular, we have $Q=Q_l \cup Q_c \cup Q_r$.
\end{lemma}
\begin{proof}
The table is immediately obtained from Lemma \ref{lem:aug_mono} and Lemma \ref{lem:aug_quad}.
The equality $Q=Q_l \cup Q_c \cup Q_r$ also follows from the table, Lemma \ref{lem:aug_mono} and Lemma \ref{lem:aug_quad}.
\end{proof}

Now, the assertions on $F[\![X]\!]$ in \cite{AK} directly follow from the assertions we gave.
The following table summarizes the assertions in \cite{AK} and the counterparts in this paper.
\begin{center}
\begin{tabular}{|c|c|}
\hline
Assertions in \cite{AK} & Counterparts in this paper\\
 \hline
  \hline
Theorem 2.3. & Lemma \ref{lem:aug_mono} and Lemma \ref{lem:triv}.\\
 \hline
 Proposition 3.2. & Lemma \ref{lem:aug_mono},  Lemma \ref{lem:aug_lcr} and 
 Lemma \ref{lem:triv}.\\
 \hline
 Scholium 3.4. & Lemma \ref{lem:aug_lcr}.\\
 \hline
 Theorem 3.5 and Corollary 3.6. & Lemma \ref{lem:aug_quad2}.\\
 \hline
 Theorem 3.7 through Corollary 3.9. & Lemma \ref{lem:aug_lcr}.\\
 \hline
Theorem 4.1 and Corollary 4.2. & Lemma \ref{lem:aug_quad}.\\
 \hline
Theorem 4.3. & Lemma \ref{lem:aug_mono} and Lemma \ref{lem:triv}.\\
 \hline
 \end{tabular}
\end{center}

\subsection{When the ring $A$ is the valued field $K$}
We next consider the case $A=K$.
In this case, we have $H=\val(A^{\times})=G$ and 
\begin{align*}
\Phi(M,\overline{g}) 
&= \{x \in K\setminus\{0\} \;|\; \val(x) = \overline{g} \text{ and } \mylcp(x) \in M \} \cup \{0\},
\end{align*}
where $M$ is a quasi-quadratic module of $F$.
We can get a simpler structure theorem when we only consider proper quasi-quadratic modules.
 
\begin{theorem}[Canonical decomposition theorem for quasi-quadratic modules]\label{thm:field1}
Let $(K,\val)$ be a $2$-henselian valued field whose residue class field is of characteristic $\not=2$.
Consider a quasi-quadratic module $\mathcal M$ in $K$.
We have $$\mathcal M = \bigcup_{\overline{g} \in G/G^2} \Phi(M_{g}(\mathcal M),\overline{g})\text{.}$$
Here, the subscript $g$ of $M_{g}(\mathcal M)$ denotes an arbitrary representative of $\overline{g}$ in $G$.
\end{theorem}
\begin{proof}
Applying Theorem \ref{thm:ring1}, we get the equality $\mathcal M = \bigcup_{g \in G} \Phi(M_{g}(\mathcal M),\overline{g})\text{.}$
We obtain the equality in the theorem by Proposition \ref{prop:ring2}(2).
\end{proof}

\begin{theorem}\label{thm:field2}
Let $(K,\val)$ be a $2$-henselian valued field whose residue class field is of characteristic $\not=2$.
Consider quasi-quadratic modules $\mathcal M$ and $\mathcal N$ in $K$.
Let $\mathcal M = \bigcup_{\overline{g} \in G/G^2} \Phi(M_{g},\overline{g})$
and $\mathcal N = \bigcup_{\overline{g} \in G/G^2} \Phi(N_{g},\overline{g})$
 be the canonical representations of $\mathcal M$ and $\mathcal N$
where $M_{g}=M_{g}(\mathcal M)$ and
$N_{g}=M_{g}(\mathcal N)$ for some representative $g \in G$ of $\overline{g}$, respectively.
We get the following equalities:
\begin{enumerate}
\item[(1)] $\displaystyle \mathcal M + \mathcal N = \bigcup_{\overline{g} \in G/G^2} \Phi(M_{g}+N_{g},\overline{g})$ when $\mathcal M + \mathcal N \not=K$;
\item[(2)] $\displaystyle\mathcal M \cap \mathcal N = \bigcup_{\overline{g} \in G/G^2} \Phi(M_{g}\cap N_{g},\overline{g})$.
% when $\mathcal M \not = K$ and $\mathcal N \not=K$.
\end{enumerate}
\end{theorem}
\begin{proof}
(1) It follows from Theorem \ref{thm:ring2}(1) that
$$\mathcal M + \mathcal N = \bigcup_{g \in G}
\Phi(M_{g}+N_{g},\overline{g})
\cup\bigcup_{l \in L} \Phi(F,\overline{l}),
$$
where $L=\{l\in G\;|\; l \geq h \text{ for some } h \text{ with }M_h+N_h=F \}$.
By Definition \ref{def:pseudo}(4), we have $M_g(\mathcal M+\mathcal N)\supseteq
M_g+N_g$ for all $g\in G$. If the set $L$ is not empty, the equality 
$M_l+N_l=F$ holds true for some $l\in G$. We get $M_l(\mathcal M+\mathcal N)=F$.
There exists a nonzero $x\in K$ with $\val(x)=l$.
By Proposition \ref{prop:ring3'}, it follows that $x\in \supp(\mathcal M+\mathcal N)$.
It means that $\pm x \in \mathcal M+\mathcal N$.
We have $\mathcal M+\mathcal N=K$ from Corollary \ref{lem:field0}, which is 
a contradiction. 
We get the equality $\mathcal M + \mathcal N = \bigcup_{g \in G}
\Phi(M_{g}+N_{g},\overline{g})$.
We obtain the assertion (1) by this equality together with Proposition \ref{prop:ring2}(2).

(2) Special case of Theorem \ref{thm:ring2}(2) with the application of Proposition \ref{prop:ring2}(2).
\end{proof}

The notation $\mathfrak X^p_R$ denotes the set of all proper quasi-quadratic modules in a commutative ring $R$.
Note that
$$
\mathcal T_F^{G}=\{(M_g)_{g\in G}\in\prod_{g\in G}\mathfrak X_{F}\;|\;
M_g\subseteq M_h \text{ whenever } (\overline{g}=\overline{h}) \vee ((M_g=F)\wedge (g\leq h))
\}
$$
in the case.

\begin{theorem}\label{thm:field4}
Let $(K,\val)$ be a 2-henselian valued field whose residue class field is of characteristic $\not=2$.
We define the map $\Theta^p:\mathfrak X^p_K \rightarrow \prod_{\overline{g} \in G/G^2} \mathfrak X^p_F$ by
\begin{equation*}
\Theta^p(\mathcal M)=(M_{g}(\mathcal M))_{\overline{g} \in G/G^2}\text{.}
\end{equation*}
The map $\Theta^p$ is a bijection.
\end{theorem}
\begin{proof}
Note that $H \cup G_{\geq e}=G$ in this case and $\Theta^p$ is well-defined by Proposition \ref{prop:ring2}(2).
Let $\Theta$ be the bijection defined in Theorem \ref{thm:ring4}.
By Corollary \ref{lem:field0} and Proposition \ref{prop:ring3'}, we have $\mathcal M=K$ when $M_g(\mathcal M)=F$ for some $g \in G$.
Therefore, $\Theta^{-1}(\mathcal T_F^{G} \setminus \prod_{g\in G}\mathfrak X^p_F)=\{K\}$.
The restriction map $\Theta |_{\mathfrak X^p_K}$ of $\Theta$ to $\mathfrak X^p_K$ provides a bijection between $\mathfrak X^p_K$ and $\mathcal T_F^{G} \cap \prod_{g\in G}\mathfrak X^p_F$.
On the other hand, we also have $\mathcal T_F^{G} \cap \prod_{g\in G}\mathfrak X^p_F=\{(X_g)_{g \in G} \in \prod_{g\in G}\mathfrak X^p_F \;|\; X_g=X_h \text{ whenever } \overline{g}=\overline{h}\}$. 
It is clear that the map $\Lambda:\mathcal T_F^{G} \cap \prod_{g\in G}\mathfrak X^p_F \to \prod_{\overline{g} \in G/G^2} \mathfrak X^p_F$
by $\Lambda((X_g)_{g\in G})=(X_g)_{\overline{g}\in G/G^2}$ for an element $(X_g)_{g\in G}\in\mathcal T_F^{G} \cap \prod_{g\in G}\mathfrak X^p_F$ is bijective.

We claim that
the composition $\Lambda\circ(\Theta |_{\mathfrak X^p_K}):\mathfrak X^p_K \rightarrow \prod_{\overline{g} \in G/G^2} \mathfrak X^p_F$
coincides with $\Theta^p$. Suppose a proper quasi-quadratic module $\mathcal M$ in $K$.
It follows that
$$
\Lambda\circ(\Theta |_{\mathfrak X^p_K})(\mathcal M)
=\Lambda((M_g(\mathcal M))_{g\in G})
=((M_g(\mathcal M))_{\overline{g}\in G/G^2})
=\Theta^p(\mathcal M).
$$
We get $\Lambda\circ(\Theta |_{\mathfrak X^p_K})=\Theta^p$. Hence $\Theta^p$ is a bijection.
\end{proof}

\begin{theorem}\label{thm:field3}
Let $(K,\val)$ be a $2$-henselian valued field whose residue class field $F$ is of characteristic $\not=2$.
Consider a quasi-quadratic module $\mathcal M$ in $K$.
\begin{enumerate}
\item[(1)] The quasi-quadratic module $\mathcal M$ is a proper quadratic module if and only if $M_{e}(\mathcal M)$ is a proper quadratic module, where $e$ is the identity element $e$ in $G$.
\item[(2)] The quasi-quadratic module $\mathcal M$ is a quasi-semiordering if and only if $M_{g}(\mathcal M)$ is a quasi-semiordering for any $g \in G$.
\item[(3)] 
Assume that $\mylcp:K^{\times}\rightarrow F^{\times}$ is an angular component map $\mylc$.
When $\mathcal M$ is a quadratic module, $\mathcal M$ is a preordering if and only if, for any $g_1, g_2\in G$, nonzero elements $c_1 \in M_{g_1}(\mathcal M)$ and $c_2 \in M_{g_2}(\mathcal M)$, we have $c_1 c_2 \in M_{g_1g_2}(\mathcal M)$.
In particular, $M_{e}(\mathcal M)$ is a preordering if $\mathcal M$ is.
\end{enumerate}
\end{theorem}
\begin{proof}
Special case of Theorem \ref{thm:ring3}.
\end{proof}

We consider a more special case; that is, the case in which the residue class field is a euclidean field.
Recall that $\po_K(f)$ denotes the monogenic quadratic module of a field $K$ and an element $f\in K$.

\begin{proposition}\label{prop:euclidean}
Let $(K,\val)$ be a $2$-henselian valued field.
We further assume that the residue class field $F$ is a euclidean field
and $\mylcp:K^{\times}\to F^{\times}$ is an angular component map $\mylc$.
The following assertions hold true:
\begin{enumerate}
\item[(1)] The field $K$ is pythagorean; that is, a sum of squares in $K$ is a square.
\item[(2)] $\po_K(1)=K^2$ and $\po_K(-1)=K$.
\item[(3)] A monogenic quadratic module $\po_K(f)$ coincides with $K^2$ or $K$ if and only if $\overline{\val(f)} = \overline{e}$
for any nonzero element $f\in K$.
In particular, there exists an element $f\in K$ with $\po_K(f)\neq K^2$ and $\po_K(f)\neq K$ if and only if $G \neq G^2$.
\item[(4)] 
Any monogenic quadratic module is contained in $K$ and contains $K^2$.
There is no proper inclusion relation between proper monogenic quadratic modules which are not  $K^2$. 
\end{enumerate}
\end{proposition}
\begin{proof}
(1) Let $f_1$ and $f_2$ be nonzero elements of $K$.
Set $c_i=\mylcp(f_i) \in F$ for $i=1,2$.
When $\val(f_1) \neq \val(f_2)$, we may assume that $\val(f_1)<\val(f_2)$ without loss of generality. 
By using Definition \ref{def:pseudo}(4),  we have $\val(f_1^2+f_2^2)=\val(f_1^2)$ and $\mylc(f_1^2+f_2^2)=\mylc(f_1^2)$.
We can find $u \in K$ such that $f_1^2+f_2^2=f_1^2u^2$
by Lemma \ref{lem:psudo_lemlem}.
Hence $f_1^2+f_2^2$ is a square in this case.

We consider the case in which $\val(f_1) = \val(f_2)=g$. We have $\mylc(f_1^2)+\mylc(f_2^2)=c_1^2+c_2^2 \not=0$  because $F$ is formally real.
We obtain $\val(f_1^2+f_2^2)=g^2$ and $\mylc(f_1^2+f_2^2)=c_1^2+c_2^2$
by Definition \ref{def:pseudo}(4).
There exists a nonzero element $c \in F$ with $c^2=c_1^2+c_2^2$ because $F$ is a euclidean field.
From Definition \ref{def:pseudo}(3), we can take a nonzero element $w \in K$ with $\val(w)=g$ and $\mylc(w)=c$.
Since $\val(f_1^2+f_2^2)=\val(w^2)$ and $\mylc(f_1^2+f_2^2)=\mylc(w^2)$, we can take $u \in K$ with $f_1^2+f_2^2 =u^2w^2$
by Lemma \ref{lem:psudo_lemlem}.
Hence $f_1^2+f_2^2$ is also a square in this case.

(2) The equality $\po(1)=K^2$ immediately follows from (1).
We get $\po(-1)=K$ by Corollary \ref{lem:field0} because $\pm 1 \in \po(-1)$.

(3) Take a nonzero element $f\in K$. By (1) and Lemma \ref{lem:monogenic_ring}, we get 
$\Phi(F^2,\overline{e})=K^2$ and $\Phi(M_f,\overline{\val(f)})=K^2f$, where $M_f$ is the
quasi-quadratic module in $F$ generated by $\mylc(f)$.
Hence it follows that 
$$\mbox{PO}(f)=K^2+K^2f=\Phi(F^2, \overline{e})+\Phi(M_f,\overline{\val(f)}).$$
Note that  $M_f=F^2$ or $M_f=-F^2$ because $F$ is euclidean.

We first assume that $\overline{\val(f)} = \overline{e}$. When $M_f=F^2$, we have
$\mbox{PO}(f)=\Phi(F^2,\overline{e})=K^2$
by Proposition \ref{prop:ring1}. When $M_f=-F^2$, we have $\pm 1\in\mbox{PO}(f)$. 
Thus it follows from Corollary \ref{lem:field0} that $\po(f)=K$.

We next demonstrate the opposite implication.
Assume $\overline{\val(f)}\neq \overline{e}$. We want to lead to contradiction.
We have $\po(f)=\Phi(F^2,\overline{e})\cup\Phi(M_f,\overline{\val(f)})$
by Lemma \ref{lem:ring3}(2). When $\po(f)=K^2$, we see that $K^2\supseteq\Phi(M_f,\overline{\val(f)})$.
Hence we have
$\Phi(F^2,\overline{e})\supseteq\Phi(M_f,\overline{\val(f)})$.
Since the quasi-quadratic module $\Phi(M_f,\overline{\val(f)})$ contains $f$,
we have $\overline{\val(f)}=\overline{e}$. This is a contradiction.
When $\po(f)=K$, we have $K=K^2\cup\Phi(M_f,\overline{\val(f)})$.
We get $-1\in \Phi(M_f,\overline{\val(f)})$ because $-1\not\in K^2$. 
This contradicts the assumption that $\overline{\val(f)}\neq\overline{e}$.

The `in particular' part easily follows from the former assertion.

(4) The first claim is obvious.
We demonstrate the latter assertion.
Assume that there exist elements $f_1$, $f_2\in K$ such that
$\mbox{PO}(f_i) \neq K$ and $\neq K^2$ for $i=1,2$ and $\mbox{PO}(f_1)\subsetneq \mbox{PO}(f_2)$.
We set $g_i=\val(f_i)$ for $i=1,2$.
We want to obtain a contradiction.
It follows, as in the proof of (3), that $\mbox{PO}(f_i)=\Phi(F^2, \overline{e})+\Phi(M_{f_i},\overline{g_i})$, where 
$M_{f_i}$ is the quasi-quadratic module in $F$ generated by $\mylc(f_i)$ for $i=1,2$.

We first show that $\overline{g_i} \neq \overline{e}$ for $i=1,2$. Suppose not.
We assume $\overline{g_1}=\overline{e}$ without loss of generality.
If $M_{f_1}=F^2$, we have $\Phi(M_{f_1},\overline{g_1})=K^2$. It contradicts the assumption that $\mbox{PO}(f_1)\neq K^2$.
If $M_{f_1}=-F^2$, we have $-1 \in \Phi(M_{f_1},\overline{g_1})$ and $\pm 1 \in \po(f_1)$. 
Consequently, we have $\po(f_1)=K$ by Corollary \ref{lem:field0}, which is a contradiction.

Since $\mbox{PO}(f_1)\subseteq\mbox{PO}(f_2)$,  we have
$$
\Phi(F^2,\overline{e})\cup\Phi(M_{f_1},\overline{g_1})
\subseteq
\Phi(F^2,\overline{e})\cup\Phi(M_{f_2},\overline{g_2})
$$
by using Lemma \ref{lem:ring3}(2).
This implies that
\begin{eqnarray*}
\Phi(M_{f_1},\overline{g_1})
&\subseteq&
\Bigl(\Phi(M_{f_1},\overline{g_1})\cap\Phi(F^2,\overline{e})\Bigl)
\cup\Bigl(\Phi(M_{f_1},\overline{g_1})\cap\Phi(M_{f_2},\overline{g_2})\Bigl)\\
&=&\Phi(M_{f_1},\overline{g_1})\cap\Phi(M_{f_2},\overline{g_2})
\end{eqnarray*}
because $\Phi(M_{f_1},\overline{g_1})\cap\Phi(F^2,\overline{e})=\{0\}$ by
Lemma \ref{lem:ring4}.
We first consider the case in which $\overline{g_1}=\overline{g_2}$. We have $M_{f_1}\neq M_{f_2}$
because $\mbox{PO}(f_1)\neq \mbox{PO}(f_2)$. Since $F$ is euclidean, 
it follows that either $M_{f_1}=F^2$ and $M_{f_2}=-F^2$ or $M_{f_1}=-F^2$ and $M_{f_2}=F^2$.
In each case we get
$M_{f_1}\cap M_{f_2}=\{0\}$.
Using Lemma \ref{lem:ring4}, 
\begin{eqnarray*}
\Phi(M_{f_1},\overline{g_1})\cap\Phi(M_{f_2},\overline{g_2})
&=&\Phi(M_{f_1}\cap M_{f_2},\overline{\max\{g_1,g_2\}})\\
&=&\Phi(\{0\},\overline{\max\{g_1,g_2\}})\\
&=&\{0\}.
\end{eqnarray*}
Thus we have $\Phi(M_{f_1},\overline{g_1})=\{0\}$, which shows $\mbox{PO}(f_1)=K^2$.
This is a contradiction.

We next consider the remaining case in which $\overline{g_1}\neq \overline{g_2}$.
However it is immediate from Lemma \ref{lem:ring4} 
that 
$
\Phi(M_{f_1},\overline{g_1})\cap\Phi(M_{f_2},\overline{g_2})=\{0\}.
$
Hence we have $\Phi(M_{f_1},\overline{g_1})=\{0\}$, which shows $\mbox{PO}(f_1)=K^2$.
Contradiction.
\end{proof}

We end this subsection by calculating  monogenic quadratic modules of iterated Laurent series fields over
a euclidean field as follows:
\begin{corollary}\label{cor:i-Laurent}
Let $K=F(\!(t_1)\!)\cdots(\!(t_n)\!)$ be the iterated Laurent 
series field in the indeterminates $t_1,\ldots,t_n$ over a euclidean field $F$.
Then Figure 1 illustrates all monogenic quadratic modules in 
$K$.
\begin{figure}[h]
\begin{center}
\includegraphics[scale=0.46]{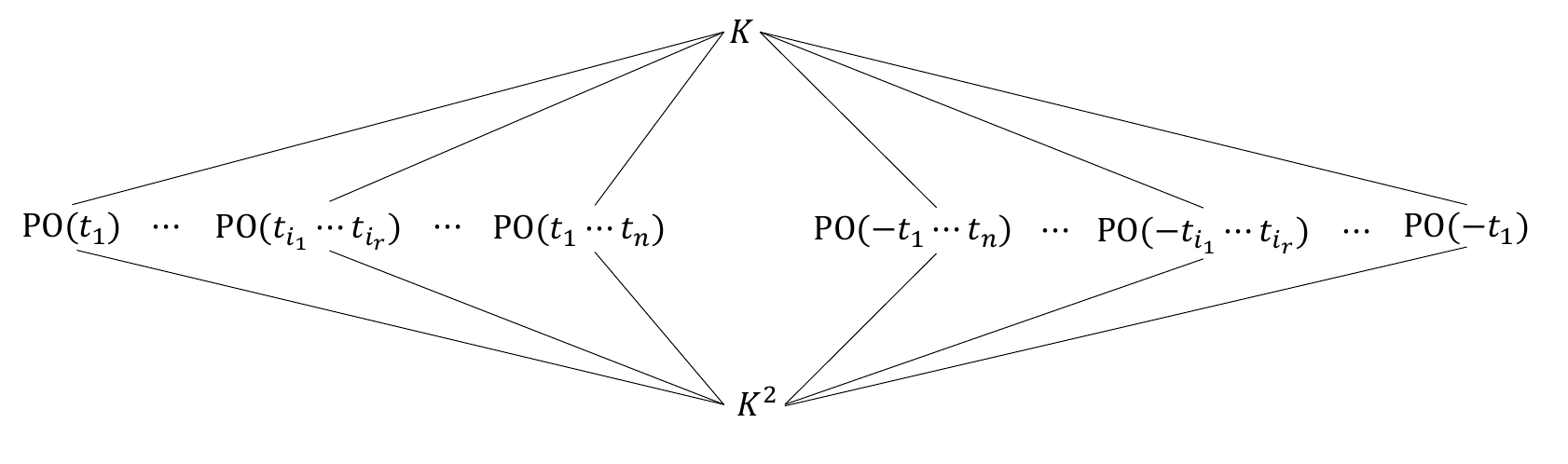}
\captionsetup{labelformat=empty,labelsep=none}
\caption{Figure 1: monogenic quadratic modules in $K$.}
\end{center}
\end{figure}

\hspace{-.46cm} Here, the solid lines denote inclusion.
\end{corollary}

\begin{proof}
As was seen in Section \ref{sec:intro}, $(K,\val)$ is a 2-henselian valued field whose value group is
$\mathbb{Z}^n$ with the lexicographic order under the $t_1>t_2>\cdots>t_n$.
It follows from Proposition \ref{prop:wmap}(1) and Proposition \ref{prop:angular}
that  $(K,\val)$ admits an angular component map.
We set $F_n=F(\!(t_1)\!)\cdots(\!(t_n)\!)$.

We first demonstrate by induction on $n$ that
for every element $f\in K$, there exist $c \in F$,
a unit $u$ of $F_n$ and integers $l_1,\ldots,l_n\in\mathbb{Z}$ such that
$$f=cu^2t_1^{l_1}\cdots t_n^{l_n}.$$

When $f=0$, it is trivial.  We always assume that $f\neq 0$. In case $n=1$, it follows that
$$
f=\sum_{i=l}^{\infty}a_it_1^i
$$
for some $l\in\mathbb{Z}$ and $a_i\in F$ with $a_l\neq 0$.
Then we have 
$$
f=a_lt_1^l(1+qt_1)
$$
for some $q\in F[\![t_1]\!]$. Since the element $1+qt_1$ is a unit and
a square in $F[\![t_1]\!]$, there exists a unit $u\in F[\![t_1]\!]$ such that 
$u^2=1+qt_1$. Hence we have $f=a_lu^2t_1^l$.

Assume $n>1$,  there exist $g,h\in F_{n-1}[\![t_n]\!]$ such that $f=\dfrac{g}{h}$.
Because we can write 
$$
g=b_lt_n^l(1+q_1t_n)\quad\mbox{and}\quad
h=c_mt_n^m(1+q_2t_n)
$$
for some $l,m\in\mathbb{Z}$, 
$b_l,c_m\in F_{n-1}$ 
and $q_1,q_2\in F_{n-1}[\![t_n]\!]$, 
there exist units $u_1, u_2\in F_{n-1}[\![t_n]\!]$ such that
$$
f=\dfrac{b_lt_n^lu_1^2}{c_mt_n^mu_2^2}.
$$
By the induction hypothesis,
there exist $e_1,e_2 \in F$, units $v_1,v_2\in F_{n-1}$ and
integers $l_1,\ldots,l_{n-1}$, $m_1,\ldots,m_{n-1}$ such that
$$
b_l=e_1v_1^2t_1^{l_1}\cdots t_n^{l_n}
\quad\mbox{and}\quad
c_m=e_2v_2^2t_1^{m_1}\cdots t_n^{m_n}.
$$
Therefore we have
$$
f=\dfrac{e_1v_1^2t_1^{l_1}\cdots t_n^{l_n}\cdot t_n^lu_1^2}
{e_2v_2^2t_1^{m_1}\cdots t_n^{m_n}\cdot t_n^mu_2^2}
=\dfrac{e_1}{e_2}\cdot\Bigl(\dfrac{u_1v_1}{u_2v_2}\Bigl)^2
t_1^{{l_1-m_1}}\cdots t_{n-1}^{{l_{n-1}-m_{n-1}}}t_n^{l-m},
$$
where $\dfrac{e_1}{e_2} \in F$ and
$\Bigl(\dfrac{u_1v_1}{u_2v_2}\Bigl)^2$ is a unit of $F_n$.
This completes the proof.

Since $F=F^2\cup(-F^2)$, there exists $d\in F$ such that
$c=\pm d^2$. Hence we have
$$f=\pm(du)^2t_1^{l_1}\cdots t_n^{l_n}.$$
Combining this with Proposition \ref{prop:euclidean},
we have Figure 1.
In particular, it follows from Proposition \ref{prop:euclidean}(3)(4) that
all monogenic quadratic modules in Figure 1 are different 
because $\val(\pm t_{i_1}\!\cdots t_{i_r})\not\equiv(0,\ldots,0)\mbox{ mod }2\mathbb{Z}^n$
for any integers $1\leq r\leq n$ and 
$1\leq i_1<\cdots <i_r\leq n$.

\end{proof}

\section{Appendix: When $F$ is of characteristic two}\label{sec:char2}

We have assumed that the residue class field $F$ is not of characteristic two in the previous part of the paper.
We consider the case in which $F$ is of characteristic two in this section.
We can obtain a simpler structure theorem than the other case we already investigated.

We prove an important technical lemma.
We crucially use the assumption that $F$ is of characteristic two in the proof. 

\begin{lemma}[Basic lemma]\label{lem:char_two1}
Let $(K,\val)$ be a valued field whose residue class field is of characteristic two.
Assume that every strict unit admits a square root.
Let $A$ be a subring of the valued field $K$ containing the valuation ring $B$.
Any element $x \in K^{\times} \setminus B^{\times}$ is the sum of squares of two elements $u$ and $v$ in $K$ with $\val(u) =\val(v) = \min\{e,\val(x)\}$.
Furthermore, we can take $u$ and $v$ in $A$ when $x \in A$.
\end{lemma}
\begin{proof}
We first demonstrate the existence of $u$ and $v$.
We first consider the case in which $x \in B \setminus B^\times$; that is $\val(x)>e$.
We have $\pi(1+x)=1$ and $\pi(-1)=-1=1$ because the residue class field $F$ is of characteristic two.
It implies that $1+x$ and $-1$ are strict units.
We can take $u,v \in K$ with $1+x=u^2$ and $-1=v^2$ by the assumption. 
We have $x=u^2+v^2$.
It is obvious that $\val(u)=\val(v)=e$.
We have proven the lemma in this case.

The remaining case is the case in which $x \in K \setminus B$.
We have $x^{-1} \in B \setminus B^\times$.
We can take $u',v' \in K$ with $x^{-1}=(u')^2+(v')^2$ as we demonstrated above.
We have $\val(u')=\val(v')=e$ and $\val((u')^2+(v')^2)=\val(x^{-1})=(\val(x))^{-1}$.
Set $u=\dfrac{u'}{(u')^2+(v')^2}$ and $v=\dfrac{v'}{(u')^2+(v')^2}$.
We get $x=u^2+v^2$ and $\val(u)=\val(v)=\val(x)$.
When $x\in A$, we have $u=u'x\in A$ and $v=v'x\in A$ because $u',v'\in B$.
%The `furthermore' part follows from Corollary \ref{cor:cor_convex}.
\end{proof}

We give several corollaries of the basic lemma.

\begin{corollary}\label{cor:char_two1}
Let $(K,\val)$ be a valued field whose residue class field is of characteristic two.
Assume that every strict unit in $K$ admits a square root.
Let $A$ be a subring of the valued field $K$ containing the valuation ring $B$.
Let $\mathcal M$ be a quasi-quadratic module of $A$.
If $x \in \mathcal M$, $y \in A$ and $\val(y) > \val(x)$, then $y \in \mathcal M$.
\end{corollary}
\begin{proof}
We have $\val(y/x)>e$ and $y/x \in A$ by Corollary \ref{cor:cor_convex}.
We can take $u,v \in A$ with $y=(u^2+v^2)x$ by Lemma \ref{lem:char_two1}.
It implies that $y \in \mathcal M$ because $\mathcal M$ is a quasi-quadratic module.
\end{proof}

\begin{corollary}\label{cor:char_two2}
Let $(K,\val)$, $A$ and $\mathcal M$ be the same as in Corollary \ref{cor:char_two1}. 
If the quasi-quadratic module $\mathcal M$ of $A$ contains an element $x$ with $\val(x)<e$, then $\mathcal M=A$.
\end{corollary}
\begin{proof}
Fix an arbitrary element $y \in A$.
When $\val(y) \geq e$, we have $\val(y) > \val(x)$.
It implies that $y \in \mathcal M$ by Corollary \ref{cor:char_two1}.

When $g=\val(y)<e$, set $h=\val(x)<e$.
We have $y^2x \in \mathcal M$ because $\mathcal M$ is a quasi-quadratic module.
Since $gh<e$, we get $\val(y^2x)=g^2h<g=\val(y)$.
We obtain $y \in \mathcal M$ by Corollary \ref{cor:char_two1} because $y^2x$ belongs to $\mathcal M$.
\end{proof}

\begin{corollary}\label{cor:char_two3}
Let $(K,\val)$, $A$ and $\mathcal M$ be the same as in Corollary \ref{cor:char_two1}. 
If $\val(\mathcal M)$ has a nonempty intersection with $\val(A^\times)$, one of the following conditions holds true:
\begin{itemize}
\item $\mathcal M=A$;
\item $A=B$ and $e \in \val(\mathcal M)$.
\end{itemize}
\end{corollary}
\begin{proof}
We first consider the case in which the intersection $\val(\mathcal M) \cap \val(A^\times)$ has an element $g$ with $g \neq e$.
When $g < e$, we get $\mathcal M=A$ by Corollary \ref{cor:char_two2}.
When $g>e$, take $x \in \mathcal M$ with $\val(x)=g$.
We have $x \in A^\times$ by Corollary \ref{cor:cor_convex}.
We get $x^{-1}=(x^{-1})^2\cdot x \in \mathcal M$.
We have $g^{-1} \in \val(\mathcal M) \cap \val(A^\times)$ and $g^{-1}<e$.
We obtain $\mathcal M=A$ by Corollary \ref{cor:char_two2}.

The remaining case is the case in which $\val(\mathcal M) \cap \val(A^\times)=\{e\}$.
If $A \neq B$, there exists an element $x\in A\setminus B$. It follows that $\val(x)<e$.
By Lemma \ref{lem:convex_basic2}(1), we have $\val(x)\in H$. It means from 
Corollary  \ref{cor:cor_convex} that $x\in A^{\times}$.
%we can get $x \in A^\times$ with $g=\val(x) \neq e$ by Corollary \ref{cor:cor_convex}.
Take $y \in \mathcal M$ with $\val(y)=e$.
The element $y$ belongs to $B^\times \subseteq A^\times$ because $\val(y)=e$.
We have $x^2y \in \mathcal M \cap A^\times$.
We easily get $e>\val(x^2y) \in \val(\mathcal M) \cap \val(A^\times)$.
Contradiction.
Therefore, we have $A=B$ in this case.
\end{proof}

We start a preparation for presenting a structure theorem of quasi-quadratic modules.

\begin{definition}\label{def:char_two1}
Let $(K,\val)$ be a valued field.
Let $A$ be a subring of the valued field $K$ containing the valuation ring $B$.
A subset $S$ of $\val(A)$ is \textit{well-behaved} if the following conditions are satisfied:
\begin{enumerate}
\item[(i)] We have $h \leq g$ for all $g \in S$ and $h \in \val(A^\times)$;
\item[(ii)] We have $gh \in S$ for all $g \in S$ and $h \in \val(A)$.
\end{enumerate}
Note that the condition (i) is equivalent to the condition that $h < g$ for all $g \in S$ and $h \in \val(A^\times)$
when $A \neq B$ because $A$ has an element $x$ with $\val(x)<e$.
When $A=B$, a subset $S$ of $G_{\geq e}$ is well-behaved if and only if any element $h \in G$ belongs to $S$ whenever $h \geq g$ for some $g \in S$.
Note also that the empty set is well-behaved.
The notation $\mathfrak S_A$ denotes the family of well-behaved subsets of $\val(A)$.
We simply denotes it by $\mathfrak S$ when the ring $A$ is clear.
\end{definition}

The following lemma immediately follows from the definition:
\begin{lemma}\label{lem:well-behaved01}
Let $(K,\val)$ be a valued field. Let $A$ be a subring of the valued field $K$ containing the valuation ring $B$
and $S$ be a well-behaved subset of $\val(A)$.
If for any element $g\in G$, there exists an element $h\in S$ with $g\geq h$, then we have $g\in S$.
\end{lemma}
\begin{proof}
We can take nonzero element $x\in K$ with $\val(x)=gh^{-1}\geq e$. 
It is trivial that $x\in A$. Hence it follows 
from Definition \ref{def:char_two1}(2) that $g=\val(x)h\in S$.
\end{proof}

We introduce several notations.
\begin{notation}
Let $(K,\val)$ be a valued field whose strict units always admit a square root.
Let $A$ be a subring of the valued field $K$ containing the valuation ring $B$.
The notations $\mathfrak S_{A,\min}$ and $\mathfrak S_{\min}$ denote the subfamily of $\mathfrak S_A$ of the subsets which have a smallest element.
Note that, when $A \neq B$, we have $\mathfrak S_{A,\min}=\emptyset$ because $\val(A)$ has an element $g$ with $g<e$.
We set
\begin{align*}
&\Gamma_1(S_1)=\{x \in A\;|\; \val(x) \in S_1\} \cup \{0\}\text{ and }\\
&\Gamma_2(S_2,M)=\{x \in A\;|\; (\val(x)=g_{\min} \text{ and } \mylcp(x) \in M) \text { or } \val(x) > g_{\min}\}
\end{align*}
for any $S_1 \in \mathfrak S$, $S_2 \in \mathfrak S_{\min}$ and quasi-quadratic module $M$ of the residue class field $F$.
Here, $g_{\min}$ denotes the smallest element of $S_2$.
We have $\Gamma_1(\emptyset)=\{0\}$ and $\Gamma_1(\val(A))=A$.
Note that we always have $0 \in \Gamma_2(S_2,M)$ because $\val(0)=\infty>g_{\min}$.
\end{notation}

The sets $\Gamma_1(S)$ and $\Gamma_2(S,M)$ are quasi-quadratic modules.

\begin{proposition}\label{prop:char_two1}
Let $(K,\val)$ be a valued field whose residue class field is of characteristic two.
Assume that every strict unit in $K$ admits a square root.
Let $A$ be a subring of the valued field $K$ containing the valuation ring $B$.
The following assertions hold true:
\begin{enumerate}
\item[(1)] The set $\Gamma_1(S)$ is a quasi-quadratic module of $A$ for any $S \in \mathfrak S_{A}$.
\item[(2)] The set $\Gamma_2(S,M)$ is a quasi-quadratic module of $A$ for any $S \in \mathfrak S_{A,\min}$ and quasi-quadratic module $M$ of the residue class field $F$.
\end{enumerate}
\end{proposition}
\begin{proof}
We demonstrate the assertion (1). It is a routine to prove that 
$\Gamma_1(S)$ is closed under multiplication by the squares of elements in $A$.
Take nonzero elements $x_1,x_2\in \Gamma_1(S)$ with $x_1+x_2\neq 0$.
It follows from Lemma \ref{lem:well-behaved01} that $\val(x_1+x_2)\in S$ because
$\val(x_1+x_2)\geq\min{(\val(x_1), \val(x_2))}\in S$.

We can prove the assertion (2) similarly to Proposition \ref{prop:ring1} other than the assertion that $x_1+x_2 \in \Gamma_2(S,M)$ when $\val(x_1)=\val(x_2)$ and $\mylcp(x_1)+\mylcp(x_2)=0$.
In this case, we get $\val(x_1+x_2)>\val(x_1)$ by Corollary \ref{cor:psudo_lemlem} and we have $x_1+x_2 \in \Gamma_2(S,M)$.
\end{proof}

We give a structure theorem for quasi-quadratic modules.
\begin{theorem}[Structure theorem]\label{thm:char_two1}
Let $(K,\val)$ be a valued field whose residue class field is of characteristic two.
Assume that every strict unit in $K$ admits a square root.
Let $A$ be a subring of the valued field $K$ containing the valuation ring $B$ and $\mathcal M$ be a nonzero proper quasi-quadratic module of $A$.
The subset $S=\val(\mathcal M)$ of $\val (A)$ is well-behaved and only one of the following conditions is satisfied:
\begin{enumerate}
\item[(1)] $S$ has no smallest element, and the equality $$\mathcal M=\Gamma_1(S)$$ holds true.
\item[(2)] We have $A=B$, and $S$ contains a smallest element $g_{\min}$.
In addition, the equality $$\mathcal M=\Gamma_2(S,M_{g_{\min}}(\mathcal M))$$ holds true.
\end{enumerate}
\end{theorem}
\begin{proof}
The first target is to prove that $S$ is well-behaved.
We demonstrate the condition (i) of Definition \ref{def:char_two1} is satisfied.
Take $g \in S$ and $h \in \val(A^\times)$.
We have $g \geq e$ by Corollary \ref{cor:char_two2} because $\mathcal M$ is proper.
Assume for contradiction that  $g < h$.
We have $g \in \val(A^\times)$ because $\val(A^\times)$ is convex by Lemma \ref{lem:convex_basic2}(1).
We get $A=B$ and $g=e$ by Corollary \ref{cor:char_two3} because $\mathcal M$ is proper.
In this case, we get $h \in \val(A^\times)=\val(B^\times)=\{e\}$.
It is a contradiction to the assumption that $g<h$.

We next show that the condition (ii) holds true.
Take $g \in S$ and $h \in \val(A)$.
We want to show that $gh \in S$.
There is nothing to prove when $h=e$.
We may assume that $h \neq e$.
Take $w \in A$ and $x \in \mathcal M$ with $h=\val(w)$ and $g=\val(x)$.
We can take $u,v \in A$ with $w=u^2+v^2$ by Lemma \ref{lem:char_two1}. 
Therefore, we obtain $wx=(u^2+v^2)x \in \mathcal M$ because $\mathcal M$ is a quasi-quadratic module.
It implies that $gh =\val(wx) \in S$.
We have demonstrated that $S$ is well-behaved.

We consider the case in which $S$ does not have a smallest element and demonstrate the equality $\mathcal M=\Gamma_1(S)$.
The inclusion $\mathcal M \subseteq \Gamma_1(S)$ is obvious.
We demonstrate the opposite inclusion.
Take an element $x \in \Gamma_1(S)$.
We have $\val(x) \in S$ by the definition of $\Gamma_1(S)$.
There exists an element $y \in \mathcal M$ with $\val(x)>\val(y)$ because $S$ does not have a smallest element.
We get $x \in \mathcal M$ by Corollary \ref{cor:char_two1}.

We next consider the case in which $S$ has a smallest element $g_{\min}$.
Take $x \in \mathcal M$ with $\val(x)=g_{\min}$.
If $A \neq B$, we can get nonzero element $b \in A$ with $\val(b)<e$.
Since $b$ is a sum of squares in $A$ by Lemma \ref{lem:char_two1}, we have $bx \in \mathcal M$.
It means that $\val(b)g_{\min}=\val(bx) \in S$.
It contradicts the minimality of $g_{\min}$. 
We have demonstrated $A=B$.

The remaining task is to prove the equality $\mathcal M=\Gamma_2(S,M)$, where $M=M_{g_{\min}}(\mathcal M)$.
The inclusion $\mathcal M \subseteq \Gamma_2(S,M)$ is obvious.
We prove the opposite inclusion.
Take $x \in \mathcal M$ with $\val(x)=g_{\min}$.
We prove that $y \in \mathcal M$ for any $y \in  \Gamma_2(S,M)$.
When $\val(y)>g_{\min}=\val(x)$, we get $y \in \mathcal M$ by Corollary \ref{cor:char_two1}.
The remaining case is the case in which $\val(y)=g_{\min}$ and $\mylcp(y) \in M$.
By the definition of $M=M_{g_{\min}}(\mathcal M)$, we can take $z \in \mathcal M$ with  $\val(z)= g_{\min}$ and $\mylcp(z)=\mylcp(y)$.
We can take $u \in B^\times$ with $y=u^2z$ by Lemma \ref{lem:psudo_lemlem}.
It implies that $y \in \mathcal M$.
\end{proof}

\begin{proposition}\label{prop:char_two2}
Let $(K,\val)$ be a valued field whose residue class field is of characteristic two.
Assume that every strict unit in $K$ admits a square root.
Let $A$ be a subring of the valued field $K$ containing the valuation ring $B$.
The following assertions hold true:
\begin{enumerate}
\item[(1)] Let $S_1$ and $S_2$ be well-behaved subsets of $\val(A)$.
Then $S_1 \subseteq S_2$ or $S_2 \subseteq S_1$.
\item[(2)] Let $S_1$ and $S_2$ be well-behaved subsets of $\val(A)$.
Let $M_1$ and $M_2$ be nonzero quasi-quadratic modules of the residue class field $F$.
\begin{enumerate}
\item[(a)] We have $\Gamma_1(S_1) \cap \Gamma_1(S_2)= \Gamma_1(S_1)$ and 
$\Gamma_1(S_1)+\Gamma_1(S_2)=\Gamma_1(S_2)$ when $S_1 \subseteq S_2$;
\item[(b)] When $A=B$ and $S_1, S_2 \in \mathfrak S_{\min}$, we have
\begin{align*}
&\Gamma_2(S_1,M_1) \cap \Gamma_2(S_2,M_2) \\
&= \left\{
\begin{array}{ll}
\Gamma_2(S_1,M_1) & \text{ if } S_1 \subsetneq S_2,\\
\Gamma_2(S_1, M_1 \cap M_2) & \text{ if } S_1=S_2 \text{ and } M_1 \cap M_2 \neq \{0\},\\
\Gamma_1(S') & \text{ if } S_1=S_2 \text{ and } M_1 \cap M_2 = \{0\},
\end{array}
\right.
\end{align*}
where $g_{\min}$ is the smallest element of $S_1$ and $S'=\{g \in S_1\;|\;g > g_{\min}\}$. 
We also have 
\begin{align*}
&\Gamma_2(S_1,M_1) + \Gamma_2(S_2,M_2) 
= \left\{
\begin{array}{ll}
\Gamma_2(S_2,M_2) & \text{ if } S_1 \subsetneq S_2,\\
\Gamma_2(S_1, M_1 + M_2) & \text{ if } S_1=S_2;
\end{array}
\right.
\end{align*}
\item[(c)] When $A=B$ and $S_1 \in \mathfrak S_{\min}$, we have
\begin{align*}
&\Gamma_2(S_1,M_1) \cap \Gamma_1(S_2) 
= \left\{
\begin{array}{ll}
\Gamma_2(S_1,M_1) & \text{ if } S_1 \subseteq S_2,\\
\Gamma_1(S_2) & \text{ if } S_2 \subsetneq S_1,\\
\end{array}
\right.
\end{align*}
and 
\begin{align*}
&\Gamma_2(S_1,M_1) + \Gamma_1(S_2) 
= \left\{
\begin{array}{ll}
\Gamma_1(S_2) & \text{ if } S_1 \subseteq S_2,\\
\Gamma_2(S_1,M_1) & \text{ if } S_2 \subsetneq S_1.\\
\end{array}
\right.
\end{align*}
\end{enumerate}
\end{enumerate}
\end{proposition}
\begin{proof}
The assertion (1) is easily proven.
Assume the contrary. 
We can take $g_1 \in S_1 \setminus S_2$ and $g_2 \in S_2 \setminus S_1$.
We may assume that $g_1 < g_2$ by symmetry.
We have $g_2 \in S_1$ by Lemma \ref{lem:well-behaved01}.
Contradiction.

Our next task is to prove assertion (2).
When $S_1 \subseteq S_2$, we have $\Gamma_1(S_1) \subseteq \Gamma_1(S_2)$.
The assertion (a) is obvious from this inclusion.

We investigate the intersection and the sum of $\Gamma_2(S_1,M_1)$ and $\Gamma_2(S_2,M_2)$ discussed in assertion (b).
When $S_1 \subsetneq S_2$, their smallest elements do not coincide by Lemma \ref{lem:well-behaved01}. 
We obviously have $\Gamma_2(S_1,M_1) \subseteq \Gamma_2(S_2,M_2)$. 
If  $\Gamma_2(S_1,M_1) = \Gamma_2(S_2,M_2)$, it follows from Definition \ref{def:pseudo}(3) that $S_1=S_2$,
which is a contradiction.
Hence we have  $\Gamma_2(S_1,M_1) \subsetneq \Gamma_2(S_2,M_2)$. 
The equalities in the assertion (b) are obvious from this inclusion in this case.
We next consider the case in which $S_1=S_2$.
Let $g_{\min}$ be the smallest element of $S_1$.
The equalities on the intersection $\Gamma_2(S_1,M_1) \cap \Gamma_2(S_2,M_2)$ are not hard to derive.
We omit the details.
For the sum $\Gamma_2(S_1,M_1)+\Gamma_2(S_2,M_2)$, we first demonstrate the inclusion $\Gamma_2(S_1,M_1)+\Gamma_2(S_2,M_2) \subseteq \Gamma_2(S_1,M_1+M_2)$.
Take arbitrary elements $x_i \in \Gamma_2(S_i,M_i)$ for $i=1,2$.
We want to demonstrate $x_1+x_2 \in \Gamma_2(S_1,M_1+M_2)$.
It is obvious when at least one of $x_i$ is zero.
It is also true when $\val(x_1)\neq \val(x_2)$.
We next consider the case in which $\val(x_1)=\val(x_2) > g_{\min}$.
It follows that 
$\val(x_1+x_2)\geq\min\{\val(x_1),\val(x_2)\}>g_{\min}$. Thus we get $x_1+x_2 \in \Gamma_2(S_1,M_1+M_2)$.
We next consider the case in which $\val(x_1)=\val(x_2)=g_{\min}$ and $\mylcp(x_1)+\mylcp(x_2) \neq 0$.
It immediately follows from Definition \ref{def:pseudo}(4).
In the remaining case, we have $\val(x_1)=\val(x_2)=g_{\min}$ and $\mylcp(x_1)+\mylcp(x_2) = 0$.
We have $\val(x_1+x_2)>g_{\min}$ by Corollary \ref{cor:psudo_lemlem} in this case.
We also get $x_1+x_2 \in \Gamma_2(S_1,M_1+M_2)$.

We next prove the opposite inclusion $\Gamma_2(S_1,M_1+M_2) \subseteq \Gamma_2(S_1,M_1)+\Gamma_2(S_2,M_2)$.
Take $x \in \Gamma_2(S_1,M_1+M_2)$.
When $\val(x) > g_{\min}$, we obviously have $x \in \Gamma_2(S_1,M_1)$.
When $\val(x)=g_{\min}$ and $\mylcp(x) \in M_i$ for some $i=1,2$, we obviously have $x \in \Gamma_2(S_i,M_i)$.
The final case is the case in which $\val(x)=g_{\min}$ and $\mylcp(x) \not\in M_i$ for $i=1,2$.
We can take nonzero $b_i \in M_i$ with $b_1+b_2 = \mylcp(x)$ in this case.
We can also take $x_1 \in A$ such that $\val(x_1) = g_{\min}$ and $\mylcp(x_1)=b_1$ by Definition \ref{def:pseudo}(3).
The element $x_1$ belongs to $\Gamma_2(S_1,M_1)$.
Set $x_2 = x - x_1\neq 0$.
We have $\val(x_2)=g_{\min}$ and $\mylcp(x_2)=b_2$ by Definition \ref{def:pseudo}(4).
It means that $x_2 \in \Gamma_2(S_2,M_2)$.
We have proven the inclusion $\Gamma_2(S_1,M_1+M_2) \subseteq \Gamma_2(S_1,M_1)+\Gamma_2(S_2,M_2)$ and finished the proof of the assertion (b).

The assertion (c) is easily seen because we have $\Gamma_2(S_1, M_1) \subseteq \Gamma_1(S_2)$ when $S_1 \subseteq S_2$ and the opposite inclusion holds true otherwise
by Lemma \ref{lem:well-behaved01}.
\end{proof}

Recall that $\mathfrak X_{A}$
denotes the set of all the quasi-quadratic modules in a commutative ring $A$.
We set as follows:
\begin{eqnarray*}
\mathcal O_F&=&
(\mathfrak{S}\setminus\mathfrak{S}_{\min})
\sqcup (\mathfrak{S}_{\min}\times (\mathfrak{X}_F\setminus\{0\})).
\end{eqnarray*}

\begin{theorem}\label{thm:ring4-ch2}
Let $(K,\val)$ be a valued field whose residue class field is of characteristic two.
Assume that every strict unit in $K$ admits a square root.
Let $A$ be a subring of the valued field $K$ containing the valuation ring $B$.
The following assertions hold true:
\begin{enumerate}
\item[(1)] When $A\neq B$, the map $V:\mathfrak{X}_A\to \mathfrak{S}_A$
given by $V(\mathcal M)=\val(\mathcal M)$
is a bijection. 
\item[(2)] When $A=B$, the map $\Phi:\mathfrak{X}_A \to \mathcal O_F$ given by
\begin{align*}
&\Phi(\mathcal M) 
= \left\{
\begin{array}{ll}
\val(\mathcal M) & \text{ if } \val(\mathcal M)\not\in\mathfrak{S}_{\min} ,\\
(\val(\mathcal M), M_{g_{\min}}(\mathcal M)) & \text{ if } \val(\mathcal M)\in\mathfrak{S}_{\min}\\
\end{array}
\right.
\end{align*}
is a bijection.
\end{enumerate}

\end{theorem}
\begin{proof}
(1) It follows from Theorem \ref{thm:char_two1} that the map $V$ is well-defined.
We define the map $\Gamma_1:\mathfrak{S}_{A}\rightarrow \mathfrak{X}_A$ by
$S\mapsto \Gamma_1(S)$. The map $\Gamma_1$ is well-defined by Proposition \ref{prop:char_two1}(1).
Take a quasi-quadratic module $\mathcal M\in \mathfrak{X}_A$. By using Theorem \ref{thm:char_two1}(1),
we have $\Gamma_1(V(\mathcal M))=\mathcal M$. Fix $S\in\mathfrak{S}_A$.
We want to demonstrate that $V(\Gamma_1(S))=S$.
We obviously have $V(\Gamma_1(S))=\{\val(x)\;|\; x\in \Gamma_1(S)\}=S$
by the definition.

(2) We define the map $\Psi:\mathcal O_F\rightarrow \mathfrak{X}_A$ by sending
$S$ to $\Gamma_1(S)$ when $S\in\mathfrak{S}\setminus\mathfrak{S}_{\min}$ and by sending
$(S,M)$ to $\Gamma_2(S, M)$ when $(S,M)\in\mathfrak{S}_{\min}\times(\mathfrak{X}_F\setminus\{0\})$.
For any $S\in\mathfrak{S}\setminus\mathfrak{S}_{\min}$, it immediately follows that $\Phi(\Psi(S))=S$.
For any  $(S,M)\in\mathfrak{S}_{\min}\times (\mathfrak{X}_F\setminus\{0\})$, we have $\val(\Gamma_2(S,M))=S\in\mathfrak{S}_{\min}$ by Definition \ref{def:pseudo}(3).
We next demonstrate that $M_{g_{\min}}(\Gamma_2(S,M))=M$.
Since $M_{g_{\min}}(\Gamma_2(S,M))=\{\mylcp(x)\;|\;x\in\Gamma_2(S,M)\;\text{and}\;\val(x)=g_{\min}\}\cup\{0\}$, 
we get $M_{g_{\min}}(\Gamma_2(S,M))\subseteq M$. To show the opposite inclusion, take a nonzero element $c\in M$. By Definition 
\ref{def:pseudo}(3), there exists an element $x\in K$ with $\mylcp(x)=c$ and $\val(x)=g_{\min}$.
This implies that $c\in M_{g_{\min}}(\Gamma_2(S,M))$.
Hence we have 
$$
\Phi(\Psi(S,M))=(\val(\Gamma_2(S,M)),M_{g_{\min}}(\Gamma_2(S,M)))=(S,M).
$$

Take a quasi-quadratic module $\mathcal M\in\mathfrak{X}_A$.
We next demonstrate that $\Psi(\Phi(\mathcal M))=\mathcal M$.
When $\val(\mathcal{M})\not\in\mathfrak{S}_{\min}$, 
we have
$$
\Psi(\Phi(\mathcal M))=\Psi(\val(\mathcal M))=\Gamma_1(\val(\mathcal M))=\mathcal M
$$
from Theorem \ref{thm:char_two1}(1).
When $\val(\mathcal M)\in\mathfrak{S}_{\min}$, it follows from Theorem \ref{thm:char_two1}(2)
that
$$
\Psi(\Phi(\mathcal M))=\Psi(\val(\mathcal M),M_{g_{\min}}(\mathcal M))=
\Gamma_2(\val(\mathcal M),M_{g_{\min}}(\mathcal M))=\mathcal M.
$$
We have finished to prove that $\Phi$ and $\Psi$ are the inverses of the others.
\end{proof}

\section*{Acknowledgment} 
The authors wish to express our gratitude to the anonymous referee
for carefully reading the paper and for comments and valuable suggestions
which helped us to improve considerably the manuscript.

\end{document}